\documentclass[11pt]{article}

\usepackage[T1]{fontenc}
\usepackage[utf8]{inputenc}
\usepackage{lmodern}
\usepackage{dsfont}

\usepackage{geometry}
\usepackage{amsmath, amssymb, amsthm, mathtools, bm}
\usepackage{bbm}
\usepackage{mathrsfs}
\usepackage{stmaryrd}
\usepackage[normalem]{ulem}
\usepackage{yhmath}

\usepackage{algorithm}
\usepackage{algpseudocode}
\usepackage{algorithmicx}

\usepackage{graphicx}
\usepackage{xcolor}

\usepackage{hyperref}
\hypersetup{
    colorlinks=true,
    linkcolor=red,
    citecolor=green,
    urlcolor=teal
}

\usepackage{enumitem}
\usepackage{authblk}
\usepackage{multicol} 
\usepackage{todonotes}
\usepackage{comment}

\newtheorem{theorem}{Theorem}[section]
\newtheorem{lemma}{Lemma}[section]
\newtheorem{proposition}{Proposition}[section]

\theoremstyle{definition}

\newtheorem{assumption}{Assumption}[section]
\newtheorem{remark}{Remark}[section]

\newcommand{\R}{\mathbb{R}} 
\newcommand{\N}{\mathbb{N}}
\newcommand{\Z}{\mathbb{Z}}

\renewcommand{\P}{\mathbb{P}}  
\newcommand{\E}{\mathbb{E}} \newcommand{\Var}{\mathrm{Var}} \newcommand{\Cov}{\mathrm{Cov}}

\newcommand{\eqlaw}{\stackrel{\Lc}{=}}

\newcommand{\half}{\frac12}

\newcommand{\1}{\mathbbm{1}} 

\newcommand{\F}{\mathbb{F}} 
\newcommand{\Fc}{\mathcal{F}} 

\newcommand{\Lc}{\mathcal{L}} \newcommand{\Nc}{\mathcal{N}}

\newcommand{\Oc}{\mathcal{O}}
\newcommand{\Ac}{\mathcal{A}}
\newcommand{\Mc}{\mathcal{M}}

\newcommand{\Xc}{\mathcal{X}}
\newcommand{\Yc}{\mathcal{Y}}

\newcommand{\Mk}{\mathfrak{m}}

\numberwithin{equation}{section}
\numberwithin{figure}{section}
\numberwithin{table}{section}

\title
      {A Martingale Approach To Fluctuations of Rank Estimators in Sensitivity Analysis}

\author[1]{Reda Chhaibi}
\author[2,3]{Fabrice Gamboa}
\author[2]{Cl\'ement Pellegrini}

\affil[1]{Universit\'e C\^ote d'Azur, LJAD, Nice, France}
\affil[2]{Univ Toulouse, UT2J, INSA, CNRS, IMT, Toulouse, France.}
\affil[3]{Facultad de Ingeniería UdM. Medellín, Colombia.}

\begin{document}

\maketitle

\begin{abstract}
Given a bivariate random pair $(X,Y)$, a natural problem is to estimate, from a single sample $(X_i,Y_i)_{1\le i\le n}$, quantities such as $\E\!\left[\E[Y\mid X]^2\right]$. More broadly, sensitivity indices are designed to quantify the possibly nonlinear influence of an input variable $X$ on an output variable $Y$. A classical example is the Sobol' index
$$
\frac{\Var(\E[Y\mid X])}{\Var(Y)} \in [0,1] \ .
$$

Another important example is the Cram\'er--von Mises (CvM) index. Following the pioneering work of Chatterjee \cite{chatterjee2021new}, consistent rank-based estimators are now available for such quantities.

In this paper, we prove sharp fluctuation results using martingale methods. Our framework yields a unified treatment of the univariate Sobol' index, a multivariate extension involving several functions of the same scalar input, and the CvM index. As a consequence, we recover, unify, and simplify results from Gamboa et al. \cite{gamboa2022global, gamboa2023erratum}, Lin--Han \cite{lin2022limit}, and Kroll \cite{kroll2024asymptotic}. In particular, we work under minimal regularity assumptions. Furthermore, while the Gaussian fluctuation phenomenon itself was already known, the novelty lies in the structure of the asymptotic variance: for the CvM index, we obtain, to the best of our knowledge, the first explicit formula, while for the Sobol' index, we derive a new expression with a more structured form.
\end{abstract}

\medskip

\hrule
\tableofcontents
\bigskip
\hrule


\section{Introduction}

{\bf Historical context.}
The quantification of stochastic dependence has deep roots in both theoretical and applied probability and statistics. In classical statistics, Pearson \cite{Pearson1901} introduced the moment correlation coefficient to measure linear association between variables. Shortly thereafter, Spearman \cite{Spearman1904} proposed a rank-based correlation measure, offering robustness to non-normality and monotonic transformations. More formally, a dependence measure between two random variables $X$ and $Y$ with joint distribution $\P_{X,Y}$ is a function $\rho$ mapping $\P_{X,Y}$ to the interval $[0,1]$. Intuitively, $\rho(\P_{X,Y}) = 0$ when $X$ and $Y$ are independent in some well-defined sense, while $\rho(\P_{X,Y}) = 1$ when $Y$ is a deterministic function of $X$. For instance, the Pearson correlation (defined for square-integrable random variables) is the ratio of their covariance to the product of their standard deviations. It equals zero for uncorrelated variables and reaches one if and only if $Y$ is an affine function of $X$. In general, one may require additional properties of a dependence measure, such as monotonicity with respect to a suitable ordering on the associated copula \cite{Sklar1959}, or invariance under certain classes of transformations. For a detailed discussion of classical dependence measures and their desirable properties, we refer the reader to the comprehensive review \cite{wu2010new}. For their key role in an industrial context, including model simplification, risk quantification, and decision-making, we refer to \cite{Rocquigny2008}.

There are two main aspects concerning dependence measures. The first aspect is the definition of a good dependence measure $\rho(\P_{X,Y})$ (hereafter simply denoted by $\rho$), that is, a measure satisfying desirable properties such as vanishing if and only if $X$ and $Y$ are independent, monotonicity, and others. Once such a measure is defined, the second aspect concerns the statistical estimation of $\rho$. More precisely, given an i.i.d. sample
$$
(X_1,Y_1), \ldots, (X_n,Y_n)
$$
drawn from the distribution of $(X,Y)$
with $n>1$, how can one estimate $\rho$? Furthermore, once such a statistical estimator $\widehat{\rho}_n$ is constructed, one seeks to establish its convergence toward the target parameter $\rho$ at the appropriate rate. Note that this setting corresponds to a semiparametric problem (see \cite{Vaart_1998}), for which the expected rate of convergence is $\sqrt{n}$.
Our paper fits within this statistical framework. Using a martingale-based approach, we provide a complete treatment of the estimation problem for two widely used dependence measures, namely the Sobol' index \cite{sobol1990sensitivity} and the Cramér-von Mises dependence measure \cite{chatterjee2021new}, when employing the rank-based estimator introduced by Chatterjee \cite{chatterjee2021new}. Before presenting further details on these two dependence measures and their statistical estimation, we emphasize that the study of dependence measures has experienced a remarkable resurgence over the past decade, driven by practical needs \cite{chatterjee2024survey}. First, within the paradigm of so-called statistical computer experiments \cite{santner2003design}, one seeks to quantify the contribution of randomness induced by an uncertain physical parameter $X$ on a physical quantity of interest $Y$, which are related through a deterministic mapping (typically a scientific or empirical law),
\begin{equation}
Y = f(X,\varepsilon) \ .
\label{lolo}
\end{equation}
Here, $\varepsilon$ represents the randomness arising from other uncertain physical quantities that are stochastically independent of $X$.
Second, in the context of machine learning, the function $f$ in~\eqref{lolo} may correspond to a learned regression model, while $X$ represents an input feature. In this setting, practitioners are often interested in providing sensitivity or interpretability measures that quantify the effect of the feature $X$ on the output $Y$.
In both settings, a common approach consists in considering dependence measures based on a mean discrepancy either between the distribution of $Y$ and its conditional distribution given $X$, or between the joint distribution of $(X,Y)$ and the product of their marginal distributions. In the machine learning literature, we refer to the pioneering work of Gretton et al., who introduced methods based on reproducing kernel Hilbert space (RKHS) embeddings (see \cite{gretton2005measuring} and the references therein). In the context of computer experiments, we refer to the comprehensive exposition by Da Veiga \cite{da2015global}, further developed in \cite{da2021basics}.
In the setting of square-integrable random variables, the most widely used dependence measure is based on the quadratic deviation between the unconditional mean of $Y$ and its conditional mean given $X$. The so-called Sobol' index \cite{sobol1990sensitivity} is defined through the normalized discrepancy
\begin{align}
\label{def:index_sobol}
\rho^{\mathrm{Sobol'}} := \frac{\Var \ \E(Y \mid X)}{\Var \ Y} \ .
\end{align}
Although very popular, this dependence measure may vanish even when $X$ and $Y$ are stochastically dependent. To overcome this limitation while preserving the appealing quadratic structure, Dette et al.\ \cite{dette2013copula} proposed using a normalized Cramér-von Mises discrepancy between the distribution function of $Y$ and its conditional distribution given $X$. This leads to our second dependence measure of interest,
\begin{equation}
\label{def:index_cramer}
\rho^{\mathrm{CvM}} := \frac{\E\!\left( F_Y(Y) - F_{Y \mid X}(Y) \right)^2}{\E\!\left[ F_Y(Y)\big(1 - F_Y(Y)\big) \right]} \ .
\end{equation}
Here, $F_Y$ (respectively $F_{Y \mid X}$) denotes the distribution function (respectively the conditional distribution function) of $Y$.
We note that the dependence measure $\rho^{\mathrm{CvM}}$ was later introduced independently in the context of computer code experiments in \cite{gamboa2018sensitivity}.
The main difficulty in estimating either $\rho^{\mathrm{Sobol'}}$ or $\rho^{\mathrm{CvM}}$ lies in the estimation of expectations involving the square of a conditional expectation. The well-known Pick-Freeze method, originally proposed by Sobol' \cite{sobol1990sensitivity} and later studied from a mathematical statistics perspective in \cite{janon2014asymptotic}, achieves the optimal $\sqrt{n}$ rate of convergence, but requires the use of specific supplementary samples.
Using order statistics, Chatterjee \cite{chatterjee2021new} elegantly proposed estimators that avoid the need for supplementary samples while retaining the optimal rate of convergence. In that work, the author establishes a non-asymptotic concentration inequality and, in the case of independence, proves a central limit theorem with the correct convergence rate. The Chatterjee estimators $\widehat{\rho}^{\mathrm{Sobol'}}_n$ and $\widehat{\rho}^{\mathrm{CvM}}_n$ are defined and discussed in Subsection~\ref{subsection:prerequisites}.

\medskip

{\bf Literature review of state of the art.}
The estimation of the Sobol' index using the Pick-Freeze method, orthogonal basis expansions, or quasi-Monte Carlo techniques has been extensively studied in recent years. We refer to \cite{da2021basics} for a comprehensive review of this literature. A central limit theorem for the estimator $\widehat{\rho}^{\mathrm{Sobol'}}_n$ is established in \cite{gamboa2022global}, where the proof relies on detailed asymptotic expansions and on a representation of order statistics via exponential distributions. The elegant approach introduced by Chatterjee \cite{chatterjee2021new} has subsequently attracted considerable attention, and we refer to \cite{chatterjee2024survey} for a recent overview. Regarding asymptotic normality, two submitted works are currently available: in \cite{lin2022limit}, the authors establish asymptotic normality using H\'ajek representations, while in \cite{kroll2024asymptotic}, mixing techniques are employed to obtain the result. We view our approach as a salient and complementary contribution, as it provides a unified framework for rank-based estimators and paves the way for numerous extensions.

\medskip

{\bf Our contribution.}
Our contribution consists of a complete analysis of the asymptotic properties of $\widehat{\rho}^{\mathrm{Sobol'}}_n$ and $\widehat{\rho}^{\mathrm{CvM}}_n$, including consistency and central limit theorems, based on martingale techniques. We further exploit the functional relationship between the Sobol' and Cramér-von Mises indices established in \cite{gamboa2018sensitivity}.
Indeed, our approach is entirely novel and proves to be robust in more general settings, including multidimensional and functional-valued regression models. This will be treated in a forthcoming paper.
Moreover, it paves the way for a general central limit theorem for rank-based estimators of general Sobol' indices (see, for instance, \cite{da2021basics}), as well as for multidimensional generalizations of $\rho^{\mathrm{CvM}}$ proposed and studied in \cite{azadkia2021simple, azadkia2021foci, azadkia2025bias}. In these latter settings, the main remaining challenge concerns the treatment of bias (see \cite{azadkia2025bias}).

\medskip

{\bf Organization of the paper.} Section \ref{section:main_results} states the three main results of the paper after the necessary prerequisites: Theorem \ref{thm:limit_sobol} deals with the fluctuations of the Sobol' rank estimator $\widehat{\rho}^{\mathrm{Sobol'}}_n$. Theorem \ref{thm:limit_sobol_multi} deals with a multivariate extension. It paves the way towards analyzing fluctuations of  $\widehat{\rho}^{\mathrm{CvM}}_n$ which is stated as Theorem \ref{thm:limit_cvm}. In particular, the first statement is followed by a sketch of proof that presents the core ideas and techniques used throughout the paper. 

Section \ref{section:process_decomposition} contains key decompositions of processes, which implement the main trick of the paper. The decompositions are with respect to a natural filtration $\F = \left( \Fc_i \ ; \ i \in \N \right)$. We state an exact decomposition built from a double Doob decomposition, and an approximate decomposition. Their relevance is to force the appearance of an $\Fc_0$-measurable contribution, an $\F$-martingale contribution and a remainder.

Section \ref{section:consistency} shows consistency of the Sobol' estimator in a very simple fashion. It can be skipped on a first read or serve as a warm-up which illustrates the strength of the decompositions from Section \ref{section:process_decomposition}.

The remaining Sections \ref{section:proof_univariate}, \ref{section:proof_multivariate} and \ref{section:proof_cvm} detail the proofs of the main theorems.



\section{Main results}
\label{section:main_results}

\subsection{Definitions and prerequisites}
\label{subsection:prerequisites}

We consider the framework of input--output models. The inputs are modeled by 
a couple $(X,\varepsilon)$, where $X$ and $\varepsilon$ are assumed to be independent. 
The output is described by a random variable $Y$. In the sequel $X$ is a real-valued random variable.

\emph{Global Sensitivity Analysis} (GSA) aims to identify which inputs have the 
greatest influence on the output. In particular, the identification of relevant 
parameters is achieved by constructing indices that quantify the dependence 
between random variables. Central indices in Global Sensitivity Analysis (GSA) include the so-called \emph{Sobol' index} and the \emph{Cramér--von Mises index}.

\medskip

\paragraph{Sensitivity indices.} Let us recall the definitions of Eqs. \eqref{def:index_sobol} and \eqref{def:index_cramer}, and give alternate expressions. The \textit{Sobol' index} \cite{Sobol2001} quantifies the contribution of an input variable \( X \) to the variance of an output \( Y \). It is defined as
\begin{align}
\label{def:index_sobol2}
\rho^{\mathrm{Sobol'}} 
      := \frac{\Var \ \E(Y \mid X)}{\Var \ Y}
       = \frac{ \E\left[ \E[ Y | X ]^2 \right] - \E[Y]^2 }
              { \E[Y^2] - \E[Y]^2 } \ .
\end{align}
Following \cite{gamboa2018sensitivity}, we define the Cramér–von Mises sensitivity index of \(X\) as
\begin{align}
\label{def:index_cramer2}
\rho^{\mathrm{CvM}} 
:= 
\frac{\E\!\left( F_Y(Y) - F_{Y \mid X}(Y) \right)^2}{\E\!\left[ F_Y(Y)\big(1 - F_Y(Y)\big) \right]}
=
\frac{ \int_{\mathbb{R}} \E \left[ \left( F_{Y|X}(t) - F_Y(t) \right)^2 \right] \, dF_Y(t) }{ \int_{\mathbb{R}} F_Y(t)\left( 1 - F_Y(t) \right) \, dF_Y(t) } \ ,
\end{align}
where $$F_Y(t) := \P(Y \leq t) $$ is the cumulative distribution function (CDF) of \(Y\), and $$ F_{Y|X=x}(t) := \P(Y \leq t \mid X=x) $$ is the conditional CDF given $X=x$. Also, $F_{Y|X} = \P(Y \leq \cdot \mid X)$ is the random conditional CDF when $x$ is sampled according to the distribution of $X$.

\medskip

\paragraph{Rank statistic.} Classical estimators, such as those based on the Pick-Freeze method, can be computationally expensive and typically require specific experimental designs.
Following the ideas of \cite{chatterjee2021new}, a highly efficient estimation from a single sample can be achieved thanks to rank statistics. We now introduce the notation used throughout the paper. Let
\[
(X_1,Y_1),\ldots,(X_n,Y_n)
\]
be an i.i.d. sample, where the distribution of \(X\) is assumed to be diffuse (so that ties occur with probability zero). Let \(\sigma_n\) denote the permutation of \(\{1,\ldots,n\}\) such that
\[
X_{\sigma_n(1)} < \cdots < X_{\sigma_n(n)} .
\]

For \(j \in \{1,\ldots,n\}\), define
\[
N_n(j) =
\begin{cases}
\sigma_n\bigl(\sigma_n^{-1}(j)+1\bigr), & \text{if } \sigma_n^{-1}(j) < n \ ,\\
\sigma_n(1), & \text{if } \sigma_n^{-1}(j) = n \ .
\end{cases}
\]

Note that \(N_n(j)\) is the index of the immediate right neighbor of \(X_j\) in the ordered sample. By convention, the successor of the largest observation is the smallest one.

\subsection{Univariate Sobol' Estimator}
\label{subsection:univariate_sobol}

This section is devoted to the statement of limit theorems for the rank-based estimator of Sobol' indices. This estimator appeared in Gamboa et al.~\cite{gamboa2022global,gamboa2023erratum}, following the rank-based ideas of \cite{chatterjee2021new}.
Given an i.i.d. sample $ (X_1, Y_1), \dots, (X_n, Y_n) $, the estimator is defined by
\begin{align}
\widehat{\rho}^{\mathrm{Sobol'}}_n 
&:= 
\frac{
\displaystyle \frac{1}{n} \sum_{i=1}^n Y_i Y_{N_n(i)} 
- \left( \frac{1}{n} \sum_{i=1}^n Y_i \right)^2
}{
\displaystyle \frac{1}{n} \sum_{i=1}^n Y_i^2 
- \left( \frac{1}{n} \sum_{i=1}^n Y_i \right)^2
} \ ,
\label{eq:Sobol_estimator_rank}
\end{align}
where $ N_n(i) $ denotes the index of the observation whose rank with respect to the input variable corresponds to that of $ X_i $. This estimator is straightforward to compute and requires only a single i.i.d. sample, in contrast with classical Pick-Freeze approaches.

We consider throughout the input--output model
$$
Y = f(X,\varepsilon) \ ,
$$
whose relevance is discussed in Appendix~\ref{appendix:representation_Y}. As discussed in Theorem \ref{thm:transfer}, such an expression is generic for pairs $(X,Y)$, leading to general measurable $f$.  
In many practical applications, the function \( f \) is regarded as a black box and cannot be accessed explicitly. This can be the case if $f$ is a complex simulator, or a legacy computer code.
Using the dependence structure induced by \( f \), the output variables \( Y_i \) can be written as
\begin{align}
Y_i = f(X_i,\varepsilon_i) \ ,
\label{eq:Y_i_def}
\end{align}
where \( (\varepsilon_i)_{i \in \mathbb{N}^*} \) is an i.i.d. sequence of random variables, independent of the input sequence \( (X_i)_{i \in \mathbb{N}^*} \).
We will also frequently make use of the conditional expectation
$$
\varphi(x) := \E_\varepsilon[f(x,\varepsilon)] = \E[Y \mid X = x] \ ,
$$
as well as the associated conditional second moment
$$
\E_\varepsilon[f(x,\varepsilon)^2] = \E[Y^2 \mid X = x] \ .
$$

In order to define martingales, we introduce the filtration $\mathbb{F} = \left( \mathcal{F}_i \ ; \ i \in \N \right)$ as
\begin{align}
    \mathcal{F}_i := 
    \sigma \left(X_k \ ; \ k \in \mathbb{N}^*\right)
    \vee \sigma\left( \varepsilon_1, \dots, \varepsilon_i \right), \quad \text{for } i \geq 0 \ .
    \label{eq:Filtration}
\end{align}
This filtration gradually reveals the noise variables \( \varepsilon_i \), while all the $X_n$ variables are assumed to be known from the beginning. In particular, the initial filtration $\Fc_0$ will play an important role, which we highlight by repeating the definition valid for $i=0$
\begin{align}
    \mathcal{F}_0 = 
    \sigma \left(X_k \ ; \ k \in \mathbb{N}^*\right) \ .
    \label{eq:Filtration0}
\end{align}

We now introduce several matrices that will play a central role in the expression of the asymptotic covariance structures. These matrices arise naturally in the proofs of the limit theorems.  
In particular, we express the covariance matrices in terms of two generic functions \( f \) and \( g \), a formulation that will prove convenient when extending the results to the multidimensional setting. Let $f$ and $g$ be two functions. We set
\begin{align}
\label{eq:cov_sigma0_fg}
   \Sigma_{0}( f, g)
:= & \ \Cov\left( 
\begin{pmatrix}
\E_\varepsilon[f(X,\varepsilon)]^2\\
\E_\varepsilon[f(X,\varepsilon)]\\
\E_\varepsilon[f^2(X,\varepsilon)]
\end{pmatrix} ,
\begin{pmatrix}
\E_\varepsilon[g(X,\varepsilon)]^2\\
\E_\varepsilon[g(X,\varepsilon)]\\
\E_\varepsilon[g^2(X,\varepsilon)]
\end{pmatrix}
\right) \ .
\end{align}
Now let us introduce the two random matrices 
\begin{align}
\label{eq:sigma_a}
\Sigma_a(X,f,g) := \begin{pmatrix}
      2 \E_\varepsilon[ f(X, \varepsilon) ] \\
      1 \\
      1
      \end{pmatrix}
      \begin{pmatrix}
      2 \E_\varepsilon[ g(X, \varepsilon) ] \\
      1 \\
      1
      \end{pmatrix}^T
    + \Cov_\varepsilon[ f(X, \varepsilon), g(X, \varepsilon) ]
     \begin{pmatrix}
      1  \\
      0 \\
      0
      \end{pmatrix}
      \begin{pmatrix}
      1  \\
      0\\
      0
      \end{pmatrix}^T \ ,
\end{align}
and
\begin{align}
\label{eq:sigma_b}
\Sigma_b(X, f, g)=\Cov_\varepsilon\left( 
\begin{pmatrix}
f(X, \varepsilon)\\
f(X, \varepsilon)\\
f^2(X, \varepsilon)
\end{pmatrix} ,
\begin{pmatrix}
g(X, \varepsilon)\\
g(X, \varepsilon)\\
g^2(X, \varepsilon)
\end{pmatrix}
\right) \ . 
\end{align}
Here $\Cov_\varepsilon$ means that the covariance is computed by averaging over $\varepsilon$. For example 
$$ \Cov_\varepsilon(f(X, \varepsilon),g(X, \varepsilon))
  =\E_{\varepsilon}(f(X, \varepsilon)g(X, \varepsilon))-\E_{\varepsilon}(f(X, \varepsilon))\E_{\varepsilon}(g(X, \varepsilon))
  =\Cov(f(X, \varepsilon),g(X, \varepsilon)\vert X) \ .
$$  
Finally we define
\begin{align}
\label{eq:cov_sigma1_fg}
\Sigma_1(f,g):=\E[\Sigma_a(X, f, g)\odot\Sigma_b(X, f, g)] \ ,
\end{align}
where the symbol $\odot$ stands for the Hadamard (component-wise) product.
We are now able to express the main theorem concerning univariate Sobol' indices.  

\begin{theorem}[Main Theorem -- Limit Theorem for \texorpdfstring{$\widehat{\rho}^{\mathrm{Sobol'}}_n$}{Sobol' index estimator}]
\label{thm:limit_sobol}
Assume that $f$ is bounded and $X$ has a continuous distribution (no atoms). Then the following holds.

\begin{itemize}
\item (Consistency) The rank-based estimator $\widehat{\rho}^{\mathrm{Sobol'}}_n$ is consistent in the almost-sure sense
\begin{align}
\lim_{n \to \infty} \widehat{\rho}^{\mathrm{Sobol'}}_n
=
\rho^{\mathrm{Sobol'}}
=
\frac{ \E\left[ \E[ Y | X ]^2 \right] - \E[Y]^2 }
     { \E[Y^2] - \E[Y]^2 } \ \quad \text{almost surely} \ .
    \label{eq:main_thm_LLN}
\end{align}

\item (Fluctuations)
Define
\begin{align}
\label{eq:def_Delta_n}
\Delta_n := & \ \E \left[\left( \varphi(X_{1}) - \varphi(X_{N_n(1)}) \right)^2\right]
= \frac{1}{n} \E\left[  \sum_{i=1}^n \left( \varphi(X_{\sigma_n(i-1)}) - \varphi(X_{\sigma_n(i)}) \right)^2 \right] 
\stackrel{n \rightarrow \infty}{\longrightarrow} 0
\ .
\end{align}
We have the following Gaussian fluctuations
\begin{align}
    \sqrt n \left[ \widehat{\rho}_n^{\mathrm{Sobol'}}
        - \left( \rho^{\mathrm{Sobol'}} - \frac{\Delta_n}{2\Var(Y)} \right)
        \right] &\xrightarrow{\Lc} \mathcal{N}(0, \sigma^2_{\mathrm{Sobol'}}) \ ,
     \label{eq:main_thm_TCL}
\end{align}
where the explicit asymptotic variance $\sigma^2_{\mathrm{Sobol'}}$ is given by
\begin{align}
    \label{eq:def_sigmaXY}
    \sigma^2_{\mathrm{Sobol'}} & := v(f)^\top \left( \Sigma_0(f,f) + \Sigma_1(f,f) \right) v(f) \ ,
\end{align}
with
\begin{align}
\label{eq:def_g}
v(f) & := \frac{1}{\operatorname{Var}(Y)} 
       \begin{pmatrix} 
        1 \\ 2 \E[Y]( \rho^{\mathrm{Sobol'}} - 1) \\ - \rho^{\mathrm{Sobol'}}
       \end{pmatrix} \ .
\end{align}

In particular, if  $\lim_{n \rightarrow \infty} \sqrt{n} \Delta_n = 0$, then $\sqrt n \left[ \widehat{\rho}_n^{\mathrm{Sobol'}} - \rho^{\mathrm{Sobol'}} \right]$
has the same fluctuations, i.e.
\begin{align}
    \sqrt n \left[ \widehat{\rho}_n^{\mathrm{Sobol'}}
        -  \rho^{\mathrm{Sobol'}}
        \right] &\xrightarrow{\Lc} \mathcal{N}(0, \sigma^2_{\mathrm{Sobol'}}) \ .
\end{align}
This is for example implied by $\varphi = \E[ Y | X=\cdot ]$ having a bounded quadratic variation (also known as the $2$-variation).
\end{itemize}
\end{theorem}

A comparison to the result of~\cite{gamboa2022global} is given in Subsection \ref{subsection:further_remarks}.
Let us now present the key elements of our approach, in a two-layered fashion. First, we shall present our main trick based on exchangeability. Then we give a more complete sketch of proof.

\paragraph{Key idea for handling the main term.} To alleviate notations of the proof, it is useful to extend \( \sigma_n \) from $\{1, \dots, n\}$ to all of \( \mathbb{N} \) by periodicity. We have
\begin{align}
    T_n(f) &= \sum_{i=1}^{n} Y_{i} Y_{N_n(i)}\nonumber\\ 
    &= \sum_{i=1}^{n} Y_{\sigma_n(i)} Y_{\sigma_n(i+1)}\nonumber\\ 
    &= \sum_{i=1}^{n} f( X_{\sigma_n(i)}, \varepsilon_{\sigma_n(i)} ) f(X_{\sigma_n(i+1)}, \varepsilon_{\sigma_n(i+1)} ) \ ,
\end{align}
where in the previous sum $\sigma_n(n+1) = \sigma_n(1)$ because of the definition of $N_n(n)$. 

By shifting indices by $1$, and invoking the periodicity once again, we obtain
\begin{align}
    T_n(f) &= \sum_{i=1}^{n} f( X_{\sigma_n(i-1)}, \varepsilon_{\sigma_n(i-1)} ) f(X_{\sigma_n(i)}, \varepsilon_{\sigma_n(i)} ) \ .
    \label{eq:Tn_def}
\end{align}

Now, fix \( n \), and since the variables $\left( \varepsilon_i \ ; \ i \in \{1, \dots, n\} \right)$ are exchangeable, we obtain the following equality in law
\begin{align}
  T_n(f) \stackrel{\Lc}{=} \sum_{i=1}^n f( X_{\sigma_n(i-1)}, \varepsilon_{i-1} ) f(X_{\sigma_n(i)}, \varepsilon_{i} ) \ .
  \label{eq:Tn_def_law}
\end{align}

In order to reflect the periodicity, here $\varepsilon_{0}$ needs to be understood as $\varepsilon_n$. This is a problem if one needs to use martingale arguments using the filtration \( \sigma(\varepsilon_1, \dots, \varepsilon_n) \) (or a finer filtration as we shall see). We need to discard the last term, thereby breaking the periodicity.

To that end, we {\it take as a definition} $\varepsilon_0=0$ and set a different quantity
\begin{align}
    Z_n(f) &:= \sum_{i=1}^n f(X_{\sigma_n(i-1)}, \varepsilon_{i-1}) f(X_{\sigma_n(i)} \ , \varepsilon_i) \ ,
    \label{eq:Zn_def}
\end{align}
so that the previous equality in law Eq.~\eqref{eq:Tn_def_law} does not hold anymore between \( T_n(f) \) and \( Z_n(f) \). It holds only up to an error term
\begin{align}
    T_n(f)
    &\stackrel{\Lc}{=} f\left( X_{\sigma_n(n)}, \varepsilon_n \right) f\left( X_{\sigma_n(1)}, \varepsilon_1 \right) + \sum_{i=2}^{n} f\left( X_{\sigma_n(i-1)}, \varepsilon_{i-1} \right) f\left( X_{\sigma_n(i)}, \varepsilon_i \right) \nonumber \\
    &= Z_n(f) - f\left( X_{\sigma_n(n)}, 0 \right) f\left( X_{\sigma_n(1)}, \varepsilon_1 \right)
           + f\left( X_{\sigma_n(n)}, \varepsilon_n \right) f\left( X_{\sigma_n(1)}, \varepsilon_1 \right) \nonumber \\
    &= Z_n(f) + \mathcal{O}(\|f\|_\infty^2) \ .
    \label{eq:Tn_vs_Zn}
\end{align}

In order to build a suitable martingale, we set for $n\in\mathbb N$ and \( 0\leq j\leq n \)
\begin{align}
    Z^{(n)}_j(f) &= \sum_{i=1}^j f(X_{\sigma_n(i-1)}, \varepsilon_{i-1}) f(X_{\sigma_n(i)}, \varepsilon_i) \ ,
    \label{eq:Z_mathfrak_j_n}
\end{align}
so that $Z^{(n)}_n(f) = Z_n(f)$.

\begin{remark}
It is worth noting that we shall deal with a martingale triangular array. 
In this context, saying that a process \( (M_j^{(n)}) \) is an $\F$-martingale triangular array means that we have a sequence (indexed by $n$) of martingales, that each have a finite time horizon $0 \leq j \leq n$. In formulas, we have
\[
\E\left[ M_{j+1}^{(n)} \mid \mathcal{F}_j \right] = M_j^{(n)}, \quad \text{for } 0 \leq j < n \ .
\]

Also, note that due to the periodicity conditions imposed on the permutation $(\sigma_n)$, there is no simple connection between $\sigma_n$ and $\sigma_{n+1}.$
\end{remark}

\paragraph{The other terms.} In order to handle all the terms appearing in the estimator \eqref{eq:Sobol_estimator_rank}, we need a joint approximate equality in law similar to Eq.~\eqref{eq:Tn_vs_Zn}. It is simply given by
\begin{align}
    \label{eq:vector_eqlaw}
    \begin{pmatrix}
    T_n(f) \\ \sum_{i=1}^n Y_i \\ \sum_{i=1}^n Y_i^2
    \end{pmatrix}
    \stackrel{\Lc}{=}
    \begin{pmatrix}
    Z_n(f) + \Oc( \|f\|_\infty^2 ) \\ \sum_{i=1}^n f(X_{\sigma_n(i)}, \varepsilon_i) \\ \sum_{i=1}^n f(X_{\sigma_n(i)}, \varepsilon_i)^2
    \end{pmatrix}
    \ .
\end{align}

The RHS in Eq. \eqref{eq:vector_eqlaw} will be the main object of study in order to prove Theorem \ref{thm:limit_cvm}. To that end, we shall use the notation
\begin{align}
    \label{eq:def_theta_n}
    \widehat{\theta}_n(f) & :=
    \frac{1}{n}
    \begin{pmatrix}
    Z_n(f) + \Oc( \|f\|_\infty^2 ) \\
    S_n(f) \\
    S_n(f^2)
    \end{pmatrix} \\
    & :=
    \frac{1}{n}
    \begin{pmatrix}
    \sum_{i=1}^n f(X_{\sigma_n(i-1)}, \varepsilon_{i-1}) f(X_{\sigma_n(i)}, \varepsilon_i) 
    + \Oc( \|f\|_\infty^2 )
    \\
    \sum_{i=1}^n f(X_{\sigma_n(i)}, \varepsilon_i) \\
    \sum_{i=1}^n f^2(X_{\sigma_n(i)}, \varepsilon_i)
    \end{pmatrix} \ \nonumber \ .
\end{align}
We will see in Proposition \ref{proposition:consistency} that $\widehat{\theta}_n(f)$ is a consistent ``estimator''\footnote{Because of the equality in law \eqref{eq:vector_eqlaw}, $\widehat{\theta}_n(f)$ is not an estimator in the strict sense as it is not necessarily a function of the original sample $\left( X_i, Y_i \right)_{1 \leq i \leq n}$.} of
\begin{align}
    \label{eq:def_theta_star}
    \theta^*(f) & :=
    \begin{pmatrix}
    \E[ \E[Y|X]^2 ] \\
    \E[Y] \\
    \E[Y^2]
    \end{pmatrix} \ .
\end{align}

Let us present a sketch of proof which reveals the general strategy of proof that is fully developed in Section \ref{section:proof_univariate}.

\begin{proof}[Sketch of proof] 
Using (approximate) Doob decompositions, we shall obtain an expression of the form $Z_n(f)+ \Oc( \|f\|_\infty^2 ) = R_n + A_n^{(n)} + M_n^{(n)}$, for all $n\in\mathbb N$ where $(M_j^{(n)})$ is an $\F$-martingale, $A_n^{(n)}$ is $\Fc_0$ measurable and $R_n$ is a remainder.

This way, we shall obtain the following decomposition.
\begin{align*}
\widehat{\theta}_n(f) & =
\frac{1}{n}
\begin{pmatrix}
Z_n(f) + \Oc( \|f\|_\infty^2 ) \\
S_n(f) \\
S_n(f^2)
\end{pmatrix} \\
&=
\frac{1}{n}
\begin{pmatrix}
R_n^{(n)}(f) + A_n^{(n)}(f) + M_n^{(n)}(f) \\
\sum_{i=1}^{n} \E_\varepsilon\big(f(X_{\sigma_n(i)}, \varepsilon)\big) + \sum_{i=1}^{n} \left( f(X_{\sigma_n(i)}, \varepsilon_i) - \E_\varepsilon\big(f(X_{\sigma_n(i)}, \varepsilon)\big) \right)  \\
\sum_{i=1}^{n} \E_\varepsilon\big(f^2(X_{\sigma_n(i)}, \varepsilon)\big) + \sum_{i=1}^{n} \left( f^2(X_{\sigma_n(i)}, \varepsilon_i) - \E_\varepsilon\big(f^2(X_{\sigma_n(i)}, \varepsilon)\big) \right)
\end{pmatrix} \ .
\end{align*}
Note that it is easy to check that the two terms $$\sum_{i=1}^{n} \left( f(X_{\sigma_n(i)}, \varepsilon_i) - \E_\varepsilon\big(f(X_{\sigma_n(i)}, \varepsilon)\big) \right),\quad \sum_{i=1}^{n} \left( f^2(X_{\sigma_n(i)}, \varepsilon_i) - \E_\varepsilon\big(f^2(X_{\sigma_n(i)}, \varepsilon)\big) \right)$$ are $\F$-martingales. Pushing further, we shall finally obtain an expression of the following form
\begin{align}
& \widehat{\theta}_n(f) =
\frac{1}{n}
\begin{pmatrix}
Z_n(f) \\
S_n(f) \\
S_n(f^2)
\end{pmatrix}  \ \nonumber \\
 & =
\frac{1}{n}
\begin{pmatrix}
R_n^{(n)}(f) \\
0 \\
0
\end{pmatrix}
+
\frac{1}{n}
\begin{pmatrix}
\sum_{i=1}^n \E_\varepsilon\big(f(X_{\sigma_n(i)}, \varepsilon)\big)^2 \\
\sum_{i=1}^n \E_\varepsilon\big(f(X_{\sigma_n(i)}, \varepsilon)\big) \\
\sum_{i=1}^n \E_\varepsilon\big(f^2(X_{\sigma_n(i)}, \varepsilon)\big)
\end{pmatrix}
+
\frac{1}{n}
\begin{pmatrix}
M_n^{(n)}(f) \\
\sum_{i=1}^n \left( f(X_{\sigma_n(i)}, \varepsilon_i) - \E_\varepsilon\big(f(X_{\sigma_n(i)}, \varepsilon)\big) \right) \\
\sum_{i=1}^n \left( f^2(X_{\sigma_n(i)}, \varepsilon_i) - \E_\varepsilon\big(f^2(X_{\sigma_n(i)}, \varepsilon)\big) \right)
\end{pmatrix} \  \nonumber\\
& =
\frac{1}{n}
\begin{pmatrix}
R_n^{(n)}(f) \\
0 \\
0
\end{pmatrix}
+
\frac{1}{n}
\underbrace{
\begin{pmatrix}
\sum_{i=1}^n \E_\varepsilon\big(f(X_i, \varepsilon)\big)^2 \\
\sum_{i=1}^n \E_\varepsilon\big(f(X_i, \varepsilon)\big) \\
\sum_{i=1}^n \E_\varepsilon\big(f^2(X_i, \varepsilon)\big)
\end{pmatrix}}_{\text{(1)}=\Ac_n^{(n)}(f)}
+
\frac{1}{n}
\underbrace{
\begin{pmatrix}
M_n^{(n)}(f) \\
\sum_{i=1}^n \left( f(X_{\sigma_n(i)}, \varepsilon_i) - \E_\varepsilon\big(f(X_{\sigma_n(i)}, \varepsilon)\big) \right) \\
\sum_{i=1}^n \left( f^2(X_{\sigma_n(i)}, \varepsilon_i) - \E_\varepsilon\big(f^2(X_{\sigma_n(i)}, \varepsilon)\big) \right)
\end{pmatrix}}_{\text{(2)}=\Mc_n^{(n)}(f)} \ , \label{eq:decomp_sketch}
\end{align}
where $R_n^{(n)}(f)$ is a remainder term.
The term \textbf{(1)} in Eq.~\eqref{eq:decomp_sketch} is \( \mathcal{F}_0 \)-measurable, 
while the term \textbf{(2)} is a square-integrable, vector-valued \( \mathbb{F} \)-martingale. 
Regarding the limit theorems, we first observe that term \textbf{(1)} satisfies a strong law of large numbers for i.i.d. random variables. 
On the other hand, term \textbf{(2)} being a vector-valued \( \mathbb{F} \)-martingale, will satisfy a strong law of large numbers for martingales.

Furthermore, we shall prove that term \textbf{(1)} satisfies a central limit theorem with asymptotic covariance matrix $\Sigma_0(f,f)$, 
while term \textbf{(2)} satisfies a central limit theorem with asymptotic covariance matrix $\Sigma_1(f,f)$. 
Although the combination of these two results is not entirely immediate, a conditioning argument allows us to conclude that 
the full decomposition in Eq.~\eqref{eq:decomp_sketch} satisfies a central limit theorem with covariance matrix \( \Sigma_0(f,f) + \Sigma_1(f,f) \).
\end{proof}

\subsection{Multivariate Sobol' Estimator}
This section is devoted to the multivariate setting, where we consider a random vector $Y=f(X,\varepsilon)$ taking values in $\mathbb{R}^d$, that is,
$$
Y=(Y^1,\ldots,Y^d)=(f_1(X,\varepsilon),\ldots,f_d(X,\varepsilon))\ .
$$
Such a representation of the pair $(Y,X)$ as $Y=f(X,\varepsilon)$ is again justified by Theorem \ref{thm:transfer}, which remains valid in the multivariate framework. The Sobol' estimators in $\mathbb{R}^d$, introduced in Eq.~\eqref{eq:Sobol_estimator_rank}, are defined componentwise by
\begin{align}
\widehat{\rho}_n^{\mathrm{Sobol'}}(f_k) & := \frac{ \displaystyle{\frac{1}{n} \sum_{i=1}^n Y_i^k Y_{N_n(i)}^k - \left( \frac{1}{n} \sum_{i=1}^n Y_i^k \right)^2} }
                { \displaystyle{\frac{1}{n} \sum_{i=1}^n (Y_i^k)^2 - \left( \frac{1}{n} \sum_{i=1}^n Y_i^k \right)^2} } \ , k=1,\ldots,d\,.
\label{eq:Sobol_phi_estimator_rank}
\end{align}
This particularly estimates the vector
\begin{align}
\left(
\rho^{\mathrm{Sobol'}}(f_k)
:= 
\frac{ \E\left[ \E[  f_k(X,\varepsilon) | X ]^2 \right] - \E[ f_k(X,\varepsilon)]^2 }
     { \E[f_k^2(X,\varepsilon)] - \E[f_k(X,\varepsilon)]^2 } 
\ ; k=1,\ldots,d\,.
\right) \ , 
    \label{eq:Sobol_phi}
\end{align}
We denote for all $k=1,\ldots,d$
$$
\varphi_k(x) := \E_\varepsilon[f_k(x,\varepsilon)]=\E[Y^k\vert X=x]  \ .
$$

We can now state the following.
\begin{theorem} 
\label{thm:limit_sobol_multi}
Define for $k=1,\ldots,d$
\begin{align}
\label{eq:def_Delta_nk}
\Delta_n(f_k) := & \ \E \left[\left( \varphi_k(X_{1}) - \varphi_k(X_{N_n(1)}) \right)^2\right]
= \frac{1}{n} \E\left[  \sum_{i=1}^n \left( \varphi_k(X_{\sigma_n(i-1)}) - \varphi_k(X_{\sigma_n(i)}) \right)^2 \right] \ .
\end{align}
We have the convergence to a Gaussian vector
\begin{align*}
    \left( \ \sqrt{n} \left[ \widehat{\rho}_n^{\mathrm{Sobol'}}(f_k)
    - \left( \rho^{\mathrm{Sobol'}}(f_k) - \frac{\Delta_n(f_k)}{2\Var(Y^k)} \right) \right]
    \ ; \ k=1,\ldots,d \ \right)
    & \stackrel{\Lc}{\longrightarrow} 
 \ \Nc(0,\Gamma^d)  \ ,
\end{align*}
where $\Gamma^d=(\Gamma^d_{k,\ell})$ with prescribed covariance
$$
\Gamma^d_{k,\ell} :=
v(f_k)^\top
\left[ \Sigma_{0}( f_k,  f_\ell)+\Sigma_{1}( f_k, f_\ell)\right]
v( f_\ell) \ ,
$$
with
\begin{align}
v(f_j) & := \frac{1}{\operatorname{Var}(f_j(X,\varepsilon))} 
       \begin{pmatrix} 
        1 \\ 2 \E[f_j(X,\varepsilon)]( \rho^{\mathrm{Sobol'}}(f_j) - 1) \\ - \rho^{\mathrm{Sobol'}}(f_j)
       \end{pmatrix} \ .
\end{align}
In particular, if  $\lim_{n \rightarrow \infty} \sqrt{n} \Delta_n(f_k) = 0$ for all $k=1,\ldots,d$, then \begin{align*}
    \left( \ \sqrt{n} \left[ \widehat{\rho}_n^{\mathrm{Sobol'}}(f_k)
    -  \rho^{\mathrm{Sobol'}}(f_k) \right]
    \ ; \ k=1,\ldots,d \ \right)
    & \stackrel{\Lc}{\longrightarrow} 
 \ \Nc(0,\Gamma^d)  \ ,
\end{align*}
 This is for example implied by $\varphi_k$ having a bounded $2$-variation for all $k=1,\ldots,d$.
\end{theorem}
\begin{remark}
\label{remark:field_meaning}
This statement can be equivalently formulated using a family of functions $\Phi$. We would have the convergence to a Gaussian field
\begin{align*}
    \left( \ \sqrt{n} ( \widehat{\rho}_n^{\mathrm{Sobol'}}(\phi)
    - \left( \rho^{\mathrm{Sobol'}}(\phi) - \frac{\Delta_n(\phi)}{2\Var(Y)} \right)
    \ ; \ \phi \in \Phi \right)
    & \stackrel{\Lc}{\longrightarrow} 
    \left( \ \Nc(\phi) \ ; \ \phi \in \Phi \right) \ ,
\end{align*}
with prescribed covariance
$$
\Cov\left[ \Nc(\phi), \Nc(\psi) \right] :=
v(\phi)^\top
\left[ \Sigma_{0}( \phi, \psi)+\Sigma_{1}( \phi, \psi)\right]
v(\psi) \ .
$$
This convergence in distribution is to be understood only at the level of finite-dimensional marginals, that is, for every finite subfamily of test functions. Establishing the existence of a genuine process indexed by $\Phi$ would require an additional analysis of appropriate regularity properties. Since we do not have a motivated example in mind, we do not pursue this question here. 
\end{remark}
\begin{proof}[Sketch of proof]
The proof of the univariate case essentially carries over verbatim. This is only possible because of two elements: (1) Gaussian fluctuations are determined by a covariance structure. (2) We already chose notations that preempt the multivariate setting such as $\Sigma_0(f,g)$ and $\Sigma_1(f,g)$. 
Details are in Section \ref{section:proof_multivariate}.
\end{proof}

\subsection{Cramér–von Mises case}

\paragraph{Rewriting the Cramér–von Mises index.}
We now simplify Eq.~\eqref{def:index_cramer2} under the assumption that $F_Y$ is continuous. If $Y$ has a density, the change of variable formula gives
$$
\int_{\mathbb{R}} h(t) \, dF_Y(t) = \int_0^1 h(F_Y^{\langle -1 \rangle}(u)) \, du \ ,
$$
for any bounded measurable function \( h \). More generally, $F_Y(Y)$ is uniform on $[0,1]$ as soon as $F_Y$ is continuous.

We apply this to both the numerator and denominator of Eq.~\eqref{def:index_cramer2}.

{\it Denominator.} We write
$$
\int_{\mathbb{R}} F_Y(t)(1 - F_Y(t)) \, dF_Y(t) = \int_0^1 u(1 - u) \, du  = \frac{1}{6} \ .
$$

{\it Numerator.} We use the identity
$$
\E\left[ (F_{Y|X}(t))^2 \right] = \E\left[ \left( \P(Y \leq t \mid X) \right)^2 \right] \ ,
\quad \text{and} \quad
\E[F_{Y|X}(t)] = F_Y(t) \ .
$$
Then,
\begin{align*}
\E \left[ \left( F_{Y|X}(t) - F_Y(t) \right)^2 \right]
&= \E[(F_{Y|X}(t))^2] - 2 F_Y(t) \E[F_{Y|X}(t)] + F_Y(t)^2 \\
&= \E[(F_{Y|X}(t))^2] - 2F_Y(t)^2 + F_Y(t)^2 \\
&= \E[(F_{Y|X}(t))^2] - F_Y(t)^2 \ .
\end{align*}

Therefore, the numerator becomes
$$
\int_{\mathbb{R}} \left( \E[(F_{Y|X}(t))^2] - F_Y(t)^2 \right) \, dF_Y(t)
= \int_{\mathbb{R}} \E[(F_{Y|X}(t))^2] \, dF_Y(t) - \int_{\mathbb{R}} F_Y(t)^2 \, dF_Y(t)
$$

Again, using the change of variable \( u = F_Y(t) \), we find
$$
\int_{\mathbb{R}} F_Y(t)^2 \, dF_Y(t) = \int_0^1 u^2 \, du = \frac{1}{3} \ .
$$

Hence, the numerator reduces to
$$
\int_{\mathbb{R}} \E[F_{Y|X}(t)^2] \, dF_Y(t) - \frac{1}{3} \ .
$$

Combining numerator and denominator, we obtain the simplified form
\begin{align}
\rho^{\mathrm{CvM}} 
= \frac{ \int_{\mathbb{R}} \E[F_{Y|X}(t)^2] \, dF_Y(t) - \frac{1}{3} }{ \frac{1}{6} }
= 6 \int_{\mathbb{R}} \E[F_{Y|X}(t)^2] \, dF_Y(t) - 2 \ .
\label{eq:cvm_simplified}
\end{align}

Therefore, under our assumptions, estimating \( \rho^{\mathrm{CvM}} \) reduces to estimating the scalar quantity
$$
T^{\mathrm{CvM}} := \int_{\mathbb{R}} \E[F_{Y|X}(t)^2] \, dF_Y(t) \ .
$$

For each threshold \( t \in \mathbb{R} \), define the binary indicator $\Upsilon_i(t) := \mathds{1}_{\{ Y_i \leq t \}}$.
Let \( \sigma_n \) denote the order permutation of the inputs \( \left( X_i \right)_{1 \leq i \leq n} \) defined by 
$$
X_{\sigma_n(1)} \leq \cdots \leq X_{\sigma_n(n)} \ .
$$
To approximate the expectation 
\begin{align}
\label{eq:def_T}
T(t) := & \ \E[F_{Y|X}(t)^2] \ ,    
\end{align}
we define the rank-based empirical statistic
\begin{align}
\label{eq:Tn_t_fixed}
T_n(t) := & \ \frac{1}{n} \sum_{i=1}^{n} \Upsilon_{\sigma_n(i)}(t) \cdot \Upsilon_{\sigma_n(i+1)}(t) \\
        = & \ \frac{1}{n} \sum_{i=1}^{n} \mathds{1}_{ Y_{\sigma_n(i)} \leq t } \mathds{1}_{ Y_{\sigma_n(i+1)} \leq t } \nonumber \\
        = & \ \frac{1}{n} \sum_{i=1}^{n} \mathds{1}_{ \max ( Y_{\sigma_n(i)} , Y_{\sigma_n(i+1)} ) \leq t } \nonumber \\
        = & \ \frac{1}{n} \sum_{i=1}^{n} \mathds{1}_{ \max ( Y_{i} , Y_{N_n(i)} ) \leq t } \nonumber
\ .
\end{align}
The convention of cyclic structure for $\sigma_n$ still applies. This statistic relies on the assumption that close values of \( X \) induce similar conditional distributions \( F_{Y|X}(t) \).  
Hence, the product \( \Upsilon_{\sigma_n(i)}(t) \cdot \Upsilon_{\sigma_n(i+1)}(t) \) approximates \( F_{Y|X}(t)^2 \), and the average over all adjacent pairs provides a consistent estimate of its expectation.
We now define the full Cramér--von Mises rank-based statistic by integrating over \( t \) with respect to the empirical distribution \( F_n \) of \( (Y_i) \)
\begin{align}
\label{eq:Tn_CVM}
T_n^{\mathrm{CvM}} := \int_\mathbb{R} T_n(t) \, dF_n(t) 
= \frac{1}{n} \sum_{i=1}^{n} \int_\R \Upsilon_{\sigma_n(i)}(t) \Upsilon_{\sigma_n(i+1)}(t) \ dF_n(t) \ .
\end{align}

\paragraph{Exact relation to Chatterjee's rank estimator.}
Let \(R_j\in\{1,\dots,n\}\) be the (strict) rank of \(Y_j\) in the sample. From the previous definition Eq.~\eqref{eq:Tn_CVM}
\begin{align*}
T_n^{\mathrm{CvM}} 
  &= \int_\R T_n(t) dF_n(t) \\
  &= \frac{1}{n^2} \sum_{i=1}^{n}\sum_{k=1}^{n}
        \1_{\{Y_{\sigma_n(i)}\le Y_k\}}
        \1_{\{Y_{\sigma_n(i+1)}\le Y_k\}} \\
  &= \frac{1}{n(n-1)}\sum_{i=1}^{n}\sum_{k=1}^{n}
        \1_{\{Y_k \ge \max(Y_{\sigma_n(i)},Y_{\sigma_n(i+1)})\}} \\
  &= \frac{1}{n^2}\sum_{i=1}^{n}
        \left(n - \max\!\bigl(R_{\sigma_n(i)},R_{\sigma_n(i+1)}\bigr) + 1\right) \ .
\end{align*}
Recalling that $\max(a,b) = \half \left( a+b + |a-b| \right)$, we have
\begin{align*}
T_n^{\mathrm{CvM}} 
  & = \frac{1}{n^2}\sum_{i=1}^{n}
        \left(n + 1 - \half(R_{\sigma_n(i)} + R_{\sigma_n(i+1)} ) - \half | R_{\sigma_n(i+1)} - R_{\sigma_n(i)} | \right) \\
  & = \frac{n+1}{n} - \frac{1}{n^2} \sum_{i=1}^n R_i - \frac{1}{2 n^2} \sum_{i=1}^n | R_{\sigma_n(i+1)} - R_{\sigma_n(i)} | \\
  & = \frac{n+1}{n} - \frac{1}{n^2} \sum_{i=1}^n i - \frac{1}{2 n^2} \sum_{i=1}^n | R_{\sigma_n(i+1)} - R_{\sigma_n(i)} | \\
  & = \frac{n+1}{2n} - \frac{1}{2 n^2} \sum_{i=1}^n | R_{\sigma_n(i+1)} - R_{\sigma_n(i)} | \ .
\end{align*}
As such
\begin{align*}
\widehat{\rho}_n^{\mathrm{CvM}}
  & = 6 T_n^{\mathrm{CvM}}  - 2\\
  & = \frac{3(n+1)}{n} - 2 - \frac{3}{n^2} \sum_{i=1}^n | R_{\sigma_n(i+1)} - R_{\sigma_n(i)} | \\
  & = 1 + \frac{3}{n} - \frac{3}{n^2} \sum_{i=1}^n | R_{\sigma_n(i+1)} - R_{\sigma_n(i)} | \ .
\end{align*}
This expression is to be compared to Chatterjee's \cite[Eq. (1.1)]{chatterjee2021new} or rather the expression which follows, simpler in the absence of ties
\begin{align}
    \label{eq:chatterjee_def}
    \widehat{\rho}_n^{\mathrm{Chatterjee}} & := 1 - \frac{3}{n^2-1} \sum_{i=1}^{n-1} | r_{i+1} - r_i | \ .
\end{align}
Notice that Chatterjee uses a different convention for the right-most neighbor. Nevertheless, one can check that
$$
R_i \;=\; \#\{\, j \in \{1,\dots,n\} : Y_j \le Y_i \,\},\quad r_i \;=\; \#\{\, j \in \{1,\dots,n\} : Y_{\sigma_n(j)} \le Y_{\sigma_n(i)} \,\} \ .
$$
and thus $r_i = R_{\sigma_n(i)}$ for all $i\in\{1,\ldots,n\}$. But all in all $\widehat{\rho}_n^{\mathrm{Chatterjee}} - \widehat{\rho}_n^{\mathrm{CvM}} = \Oc\left( \frac1n \right)$, so that our setups do not differ in a meaningful way.

\paragraph{Statement.} Let us introduce the useful notation 
\begin{align}
    \label{eq:def_phi_tx}
    \varphi(t,x) := & \ \P\left( Y \leq t \mid X = x \right)
                  = \P\left( f(x, \varepsilon) \leq t \right) \ .
\end{align}
We are now in the position to express the limit theorem.
\begin{theorem}
\label{thm:limit_cvm}
Assume that $F_Y$ and $F_X$ are continuous (no atoms for the marginals). The rank-based estimator $\widehat{\rho}^{\mathrm{CvM}}_n$ is consistent in the sense that
\begin{align}
\lim_{n \to \infty} \widehat{\rho}^{\mathrm{CvM}}_n
=
\rho^{\mathrm{CvM}}
\ \quad \text{almost surely} \ .
    \label{eq:main_thm_cvm_LLN}
\end{align}

Furthermore, define
\begin{align}
    \label{eq:def_cvm_Delta}
    \Delta_n := & \
    \frac{1}{n}
    \sum_{i=1}^{n}
    \int_\R  
    \E\left[ \left[ 
  \varphi(t, X_{\sigma_n(i-1)})
  -
  \varphi(t, X_{\sigma_n(i)})
  \right]^2
  \right] dF_Y(t)\\
  = & 
  \int_\R \E\left[ 
  \left(
  \varphi(t, X_{1})
  -
  \varphi(t, X_{N_n(1)})
  \right)^2
  \right] dF_Y(t) \stackrel{n \rightarrow \infty}{\longrightarrow} 0 \nonumber \ . 
\end{align}

We have the convergence in law
$$
\sqrt{n}\left( \widehat{\rho}_n^{\mathrm{CvM}} - \rho^{\mathrm{CvM}} + 3 \Delta_n \right)
\stackrel{\Lc}{\longrightarrow}
\Nc( 0, \sigma^2_{\mathrm{CvM}} ) \ ,
$$
where 
\begin{align*}
    \frac{1}{36}\sigma^2_{\mathrm{CvM}} 
= & \ \int_\R \int_\R C_{\Xc, \Xc} dF_Y(t) dF_Y(s)
    + \int_\R \int_\R C_{\Yc, \Yc} dT(t) dT(s)
    -2\int_\R \int_\R C_{\Xc, \Yc} dF_Y(t) dT(s) \ ,
\end{align*}
and 
\begin{align*}
C_{\Xc,\Xc}(t,s)
=&\ \E\!\big[F_{Y|X}(t)^2F_{Y|X}(s)^2\big]
-\E\!\big[F_{Y|X}(t)^2\big]\E\!\big[F_{Y|X}(s)^2\big]\\&+\E\Bigl[
  \bigl(
    F_{Y|X}(t\wedge s)
    +3\,F_{Y|X}(t)\,F_{Y|X}(s)
  \bigr) \times
  \bigl(
    F_{Y|X}(t\wedge s)
    -F_{Y|X}(t)\,F_{Y|X}(s)
  \bigr)
\Bigr]\\
C_{\Xc,\Yc}(t,s)
=&\ \E\!\big[F_{Y|X}(s)^2F_{Y|X}(t)\big]
-\E\!\big[F_{Y|X}(s)^2\big]\E\!\big[F_{Y|X}(t)\big]\\
&+\E\Bigl[
  2\,F_{Y|X}(s) \times
  \bigl(
    F_{Y|X}(t\wedge s)
    -F_{Y|X}(s)\,F_{Y|X}(t)
  \bigr)
\Bigr]\\
C_{\Yc,\Yc}(t,s)
=&\ F_Y(t\wedge s)-F_Y(t)F_Y(s) \ .
\end{align*}
In particular, if  $\lim_{n \rightarrow \infty} \sqrt{n} \Delta_n = 0$, then $\sqrt n \left[ \widehat{\rho}_n^{\mathrm{CvM}} - \rho^{\mathrm{CvM}} \right]$
has the same fluctuations, i.e.
\begin{align}
    \sqrt n \left[ \widehat{\rho}_n^{\mathrm{CvM}}
        -  \rho^{\mathrm{CvM}}
        \right] &\xrightarrow{\Lc} \mathcal{N}(0, \sigma^2_{\mathrm{CvM}}) \ .
\end{align}
This is for example implied by $\varphi(t,\cdot)$ having a bounded quadratic variation (also known as the $2$-variation), uniformly in $t$.
\end{theorem}
\begin{proof}
See Section \ref{section:proof_cvm}.
\end{proof}

\subsection{Further remarks}
\label{subsection:further_remarks}

\paragraph{On the expression of the asymptotic variances.} Lin--Han's result \cite[Theorem 1.1]{lin2022limit} gives an asymptotic variance which is not explicit. A statistical estimator and numerical estimates are provided in their Theorem 1.2 and Proposition 1.2. Likewise, in Kroll's approach \cite{kroll2024asymptotic}, based on mixing, the variance seems difficult to track.
In comparison, our method yields a particularly structured and streamlined expression for both asymptotic variances $\sigma^2_{\mathrm{Sobol'}}$ and $\sigma^2_{\mathrm{CvM}}$.

\medskip

\paragraph{Comparison with \cite{gamboa2022global}.}
The method developed in \cite{gamboa2022global} differs significantly from ours in the following aspects. 
On the one hand, our approach is more structural and allows one to generalize to other cases. 
Indeed, the matrices $\Sigma_0(f,f)$ and $\Sigma_1(f,f)$ emerge in a natural way from the underlying martingale structure. 
Furthermore, this martingale-based viewpoint provides a unified and flexible framework that extends seamlessly to a special multivariate setting and to the Cramér--von Mises case.
In all cases, the covariance has a similar structure.

On the other hand, we have optimal regularity hypotheses. Assuming uniformity of the $X_i$'s on $[0,1]$, the authors of \cite{gamboa2022global} center the order statistics $X_{\sigma_n(i)}$ around $i/(n+1)$ and use Taylor expansions to handle $f(X_{\sigma_n(i)}, \ \cdot\ )$. While effective, this imposes stronger smoothness assumptions on $f$ and requires very careful control of higher-order terms to obtain precise constants. The resulting complexity can make the derivation error-prone in practice, as evidenced by the erratum \cite{gamboa2023erratum} addressing issues in the asymptotic variance $\sigma^2_{\mathrm{Sobol'}}$. In future versions of the paper, we will include numerical simulations aimed at validating our formula.

\medskip

\paragraph{On non-asymptotic estimates.} In our opinion, the non-asymptotic concentration result of Chatterjee \cite[Lemma A.11 in supplementary material]{chatterjee2021new}, based on the McDiarmid inequality, is already sharp. Nevertheless, one could derive similar results from our decompositions and concentration of martingales. At this point, it is unclear which result would be more useful.

\medskip

\paragraph{Open directions.}
The estimation of nonlinear functionals of conditional expectations lies at the core of global sensitivity analysis \cite{da2021basics} and of methods designed to quantify dependence between random variables \cite{dette2013copula, chatterjee2021new}. This problem also naturally arises in the estimation of residual variance in general regression models \cite{devroye2018nearest}. The rank-based approach, initially introduced in \cite{chatterjee2021new} and subsequently extended to multidimensional conditioning in \cite{azadkia2021simple}, proves to be elegant and effective for constructing powerful estimators. However, as already emphasized in \cite{devroye2018nearest, broto2020variance}, a major drawback of such estimators is the emergence of bias, as the ambient dimension of the conditioning variable increases. This bias contaminates the convergence rate in the central limit theorem. To address this issue, recent works propose ad hoc bias-correction procedures \cite{azadkia2025bias, da2024efficient}. Our martingale-based approach is flexible and offers a sharp characterization of the fluctuation term. As such, one would hope that it can be combined with these bias-removal techniques, and pave the way for a complete asymptotic analysis of general nearest-neighbor estimators for nonlinear functionals of conditional expectations.

\section{Approximate and exact algebraic decompositions for \texorpdfstring{$\left(Z^{(n)}_j\right)_{0 \leq j \leq n}$}{Znj}}
\label{section:process_decomposition}

In this section, the main result is as follows.
\begin{proposition}[Approximate Decomposition]
\label{proposition:double_doob}
For each $n\in\mathbb N$ and \( 1\leq j\leq n \), we have the decomposition
$$
Z_j^{(n)}(f) = R_j^{(n)}(f) + A_j^{(n)}(f) + M_j^{(n)}(f) \ ,
$$
where
\begin{align}
    R_j^{(n)}(f)
:=  \Oc( \|f\|_\infty^2 )
-
\half
\sum_{i=1}^j \left( \E_\varepsilon[f(X_{\sigma_n(i-1)}, \varepsilon)] -
\E_\varepsilon[f(X_{\sigma_n(i)}, \varepsilon)] 
\right)^2
\ ,
\label{eq:main_residue_part}
\end{align}
and
\begin{align}
    A_j^{(n)}(f)
:= \sum_{i=1}^j \E_\varepsilon[f(X_{\sigma_n(i)}, \varepsilon)]^2 \ ,
\label{eq:main_preditable_part}
\end{align}
and 
\begin{align}
 M_j^{(n)}(f) & :=  \sum_{k=1}^{j} \left( f(X_{\sigma_n(k-1)}, \varepsilon_{k-1}) + \E_\varepsilon[f(X_{\sigma_n(k+1)}, \varepsilon)] \right)
     \cdot \Delta \Mk_k^{(n)}(f) 
\label{eq:main_martingale_part}
\end{align}
with the martingale increments defined for all \( 1\leq i\leq n \) by
\begin{align}
\Delta\Mk_i^{(n)}(f)
:= f(X_{\sigma_n(i)}, \varepsilon_i) - \E_\varepsilon[f(X_{\sigma_n(i)}, \varepsilon)] \ .
    \label{eq:martingale_increment}
\end{align}
Moreover, we have the following properties
\begin{itemize}
    \item \( M_j^{(n)}(f) \) is an $\F$-martingale,
    \item \( A_j^{(n)}(f) \) is $\Fc_0$-measurable.
    \item The implicit constant in the $\Oc$ is absolute.
\end{itemize}
\end{proposition}

Proposition \ref{proposition:double_doob} is an immediate consequence of the upcoming exact algebraic decomposition of the process $\left(Z^{(n)}_j \ ; \ 0 \leq j \leq n \right)$, given in Proposition \ref{proposition:algebraic_double_doob}. When rearranging sums and grouping terms, it suffices to gather undesirable boundary terms in the remainder $R_j^{(n)}(f)$ where they fall under the $\Oc( \|f\|_\infty^2 )$ error.

This algebraic decomposition is based on the double application of the Doob decomposition theorem. The result holds without any approximation and allows us to clearly separate the predictable and martingale components of the dynamics under a natural filtration.

\begin{proposition}[Algebraic Doob Decomposition]
\label{proposition:algebraic_double_doob}
For each $n\in\mathbb N$ and \( 1\leq j\leq n \), we have the decomposition
\[
Z_j^{(n)}(f) = A_j^{(n)}(f) + M_j^{(n)}(f) \ ,
\]
where
\begin{align}
    A_j^{(n)}(f) 
= f(X_{\sigma_n(0)}, 0) \cdot \E_\varepsilon[f(X_{\sigma_n(1)}, \varepsilon)]
+ \sum_{i=2}^j \E_\varepsilon[f(X_{\sigma_n(i-1)}, \varepsilon)] \cdot \E_\varepsilon[f(X_{\sigma_n(i)}, \varepsilon)] \ ,
\label{eq:algebraic_preditable_part}
\end{align}
and $M_1^{(n)}(f)=f(X_{\sigma_n(0)}, 0)  \cdot \Delta \Mk_{1}^{(n)}(f)$, and for all $j\geq 2$
\begin{align}
  & \ M_j^{(n)} (f) \nonumber \\
= & \ \sum_{k=1}^{j-1} \left( f(X_{\sigma_n(k-1)}, \varepsilon_{k-1}) + \E_\varepsilon[f(X_{\sigma_n(k+1)}, \varepsilon)] \right) 
\cdot \Delta \Mk_k^{(n)}(f)+f(X_{\sigma_n(j-1)}, \varepsilon_{j-1}) \cdot \Delta \Mk_j^{(n)}(f) 
\label{eq:algebraic_martingale_part}
\end{align}
with the martingale increments $\Delta \Mk_i^{(n)}(f)$ being the same as before.
\end{proposition}

\begin{proof}
The proof of this proposition relies on a double Doob decomposition.

\paragraph{The First Doob Decomposition.}
Let $n\in\mathbb N$ and let \( j \in \{1, \dots, n\} \). For each \( i \in \{1, \dots, j\} \), define the increment of the process \( \left(Z_i^{(n)}\right) \) by
\begin{align}
    \Delta Z_i^{(n)} := Z_i^{(n)}-Z_{i-1}^{(n)}= f\left( X_{\sigma_n(i-1)}, \varepsilon_{i-1} \right) f\left( X_{\sigma_n(i)}, \varepsilon_i \right) \ .
    \label{eq:Delta_Z_def}
\end{align}

Then, the Doob decomposition of the adapted process \( \left(Z_j^{(n)}\right) \) with respect to the filtration \( (\mathcal{F}_j) \) is given by
\begin{align}
    Z_j^{(n)} = A_j^{(n,1)} + M_j^{(n,1)} \ ,
    \label{eq:first_doob}
\end{align}
where
\begin{align}
    A_{j}^{(n,1)} 
    &= \sum_{i=1}^j \E\left[ \Delta Z_i^{(n)} \,\middle|\, \mathcal{F}_{i-1} \right] \nonumber \\
    &= f\left( X_{\sigma_n(0)}, 0 \right) \cdot \E_\varepsilon\left[ f\left( X_{\sigma_n(1)}, \varepsilon \right) \right] 
    + \sum_{i=2}^j f\left( X_{\sigma_n(i-1)}, \varepsilon_{i-1} \right) \cdot \E_\varepsilon\left[ f\left( X_{\sigma_n(i)}, \varepsilon \right) \right] \ ,
    \label{eq:A_j1}
\end{align}
and the martingale
\begin{align*}
    M_{j}^{(n,1)}
    &= \sum_{i=1}^j \left( \Delta Z_i^{(n)} - \E\left[ \Delta Z_i^{(n)} \,\middle|\, \mathcal{F}_{i-1} \right] \right) \ , \nonumber \\
    &= \sum_{i=1}^j \left(
    f\left( X_{\sigma_n(i-1)}, \varepsilon_{i-1} \right) f\left( X_{\sigma_n(i)}, \varepsilon_i \right)
    - \E\left[
    f\left( X_{\sigma_n(i-1)}, \varepsilon_{i-1} \right) f\left( X_{\sigma_n(i)}, \varepsilon_i \right)
    \,\middle|\, \mathcal{F}_{i-1}
    \right] \right)   \ , \nonumber \\
    &= \sum_{i=1}^j f\left( X_{\sigma_n(i-1)}, \varepsilon_{i-1} \right)
    \left(
    f\left( X_{\sigma_n(i)}, \varepsilon_i \right)
    - \E_\varepsilon\left[ f\left( X_{\sigma_n(i)}, \varepsilon \right) \right]
    \right)  \ , \nonumber \\
    &= f\left( X_{\sigma_n(0)}, 0 \right) \cdot \Delta \Mk_1^{(n)} + \sum_{i=2}^j f\left( X_{\sigma_n(i-1)}, \varepsilon_{i-1} \right) \cdot \Delta \Mk_i^{(n)} \ .
\end{align*}
with the martingale increments given by Eq.~\eqref{eq:martingale_increment}.

At this stage one can check that the process $(A_j^{(n)})$ is predictable as required by the Doob decomposition. Nevertheless the fluctuations of this term are not easy to understand. A second Doob decomposition will clarify the situation.

\paragraph{The Second Doob Decomposition.}
We now decompose again the process \( \left( A_{j}^{(n,1)} \right) \) defined in Eq.~\eqref{eq:A_j1}. This decomposition is taken with respect to the shifted filtration \( \mathbb{F}^{(-)} = (\mathcal{F}_i^{(-)})_{i \geq 1} \), where
\begin{align}
    \mathcal{F}_i^{(-)} := \mathcal{F}_{i-1} \ , \quad \text{for all } i \geq 1 \ ,
    \label{eq:F_1}
\end{align}
with \( \mathcal{F}_0^{(-)} = \mathcal{F}_0 \). We recall that \( \varepsilon_0 = 0 \) by convention.

Define for each \( i \in \{1, \dots, j\} \) the shifted increment
\begin{align}
    \Delta Z_i^{(-)} := f\left( X_{\sigma_n(i-1)}, \varepsilon_{i-1} \right) \cdot \E_\varepsilon\left[ f\left( X_{\sigma_n(i)}, \varepsilon \right) \right] \ .
    \label{eq:shifted_increment}
\end{align}

By Doob decomposition, we can write
\begin{align}
    A_{j}^{(n,1)} = A_{j}^{(n,2)} + M_{j}^{(n,2)} \ ,
    \label{eq:doob_second}
\end{align}
where this time
\begin{itemize}
    \item \( A_{j}^{(n,2)} \) is predictable with respect to \( \mathbb{F}^{(-)} \),
    \item \( M_{j}^{(n,2)} \) is a martingale adapted to \( \mathbb{F}^{(-)} \), and also adapted to \( \mathbb{F} \).
\end{itemize}

We namely obtain the following formula
\begin{align}
    A_{j}^{(n,2)}
    &= \sum_{i=1}^j \E \left[ \Delta Z_i^{(-)} \mid \mathcal{F}^{(-)}_{i-1} \right] \nonumber \\
    &= \E \left[ \Delta Z_1^{(-)} \mid \mathcal{F}_{0} \right]+\sum_{i=2}^j \E \left[ \Delta Z_i^{(-)} \mid \mathcal{F}_{i-2} \right] \nonumber\\
    &= f\left( X_{\sigma_n(0)}, 0 \right) \cdot \E_\varepsilon\left[ f\left( X_{\sigma_n(1)}, \varepsilon \right) \right]
    + \sum_{i=2}^j \E_\varepsilon\left[ f\left( X_{\sigma_n(i-1)}, \varepsilon \right) \right] \cdot \E_\varepsilon\left[ f\left( X_{\sigma_n(i)}, \varepsilon \right) \right] \ .
    \label{eq:A_j2}
\end{align}

For the martingale component, we compute
\begin{align}
    M_{j}^{(n,2)}
    &=
     \sum_{i=1}^j \left( \Delta Z_i^{(-)} - \E \left[ \Delta Z_i^{(-)} \mid \mathcal{F}^{(-)}_{i-1} \right] \right) \nonumber\\
     &= \left( \Delta Z_1^{(-)} - \E \left[ \Delta Z_1^{(-)} \mid \mathcal{F}_0 \right] \right)
    + \sum_{i=2}^j \left( \Delta Z_i^{(-)} - \E \left[ \Delta Z_i^{(-)} \mid \mathcal{F}_{i-2} \right] \right) \nonumber \\
     &= \sum_{i=2}^j \left( \Delta Z_i^{(-)} - \E \left[ \Delta Z_i^{(-)} \mid \mathcal{F}_{i-2} \right] \right) \nonumber \\
   &= \sum_{i=2}^j \left( 
f(X_{\sigma_n(i-1)}, \varepsilon_{i-1}) 
\cdot \E_\varepsilon[f(X_{\sigma_n(i)}, \varepsilon)]
- \E \left[ f(X_{\sigma_n(i-1)}, \varepsilon_{i-1}) \,\middle|\, \mathcal{F}_{i-2} \right]
\cdot \E_\varepsilon[f(X_{\sigma_n(i)}, \varepsilon)] \right) \nonumber \\
&= \sum_{i=2}^j \left(
f(X_{\sigma_n(i-1)}, \varepsilon_{i-1}) 
- \E_\varepsilon[f(X_{\sigma_n(i-1)}, \varepsilon)]
\right)
\cdot \E_\varepsilon[f(X_{\sigma_n(i)}, \varepsilon)] \nonumber \\
&= \sum_{i=2}^j \E_\varepsilon[f(X_{\sigma_n(i)}, \varepsilon)] 
\cdot \Delta \Mk_{i-1}^{(n)} \ .
    \label{eq:M_j2}
\end{align}
This concludes the second Doob decomposition.

\paragraph{Final Decomposition of \texorpdfstring{$Z_j^{(n)}$}{Zjn} in Terms of Predictable and Martingale Parts.}

Now we can put everything together.
Combining the two Doob decompositions, we obtain
\begin{align*}
    Z_j^{(n)} = A_j^{(n)} + M_j^{(n)} \ ,
\end{align*}
where
\begin{itemize}
    \item \( A_j^{(n)} := A_{j}^{(n,2)} \) 
    \item \( M_j^{(n)} := M_{j}^{(n,1)} + M_{j}^{(n,2)} \) 
\end{itemize}

Using the explicit forms of \( M_j^{(n,1)} \) and \( M_j^{(n,2)} \), we rewrite the total martingale as
\begin{align}
  \ & M_j^{(n)} \\
= \ & M_j^{(n,1)} + M_j^{(n,2)} \nonumber \\[1.5ex]
= \ & f(X_{\sigma_n(0)}, 0) \cdot \Delta \Mk_1^{(n)} 
+ \sum_{i=2}^j f(X_{\sigma_n(i-1)}, \varepsilon_{i-1}) \cdot \Delta \Mk_i^{(n)}   + \sum_{i=1}^{j-1} \E_\varepsilon[f(X_{\sigma_n(i+1)}, \varepsilon)] \cdot \Delta \Mk_i^{(n)} \nonumber \\[2ex]
= \ & \left(f(X_{\sigma_n(0)}, 0) + \E_\varepsilon[f(X_{\sigma_n(2)}, \varepsilon)]\right) \cdot \Delta \Mk_{1}^{(n)}+ f(X_{\sigma_n(j-1)}, \varepsilon_{j-1}) \cdot \Delta \Mk_j^{(n)}\nonumber \\[1.5ex]
  & \quad + \sum_{k=2}^{j-1} \left( f(X_{\sigma_n(k-1)}, \varepsilon_{k-1}) + \E_\varepsilon[f(X_{\sigma_n(k+1)}, \varepsilon)] \right) 
\cdot \Delta \Mk_k^{(n)} \ .
\label{eq:m_final_bords}
\end{align}

It is then clear that $(A_j^{(n)})$ is $\mathcal F_0$-measurable and $(M_j^{(n)})$ is an $\F$-martingale. This completes the proof of Proposition~\ref{proposition:double_doob}.
\end{proof}

\section{Martingale approach to consistency}
\label{section:consistency}

This section is devoted to the consistency of the estimator which is the first part of Theorem \ref{thm:limit_sobol}. While that result is known, as is the corresponding result for the Cramér--von Mises case, this section serves to illustrate the approach.

\begin{proposition}
\label{proposition:consistency}
We have the following almost sure convergence
$$
\lim_{n\to\infty}
\widehat{\theta}_n(f)
=
\theta^*(f) \ .$$
As a consequence we have the consistency result in Eq.~\eqref{eq:main_thm_LLN}, which we recall.
\begin{align*}
\lim_{n \to \infty} \widehat{\rho}_n^{\mathrm{Sobol'}}
= \rho^{\mathrm{Sobol'}}
= \frac{ \E\left[ \E[Y|X]^2 \right] - \E[Y]^2 }{ \E[Y^2] - \E[Y]^2 }\ , \quad \text{almost surely} \ .
\end{align*}
\end{proposition}
One may argue whether convergence in probability is more natural. We refer to Remark \ref{remark:cv_as_or_proba} for a discussion, after the proof.

\begin{proof} Recall the decomposition from Eq. \eqref{eq:decomp_sketch}
\begin{align*}
\widehat{\theta}_n(f)
&=
\frac{1}{n}
\begin{pmatrix}
R_n^{(n)}(f) \\
0 \\
0
\end{pmatrix}
+
\frac{1}{n}
\underbrace{
\begin{pmatrix}
\sum_{i=1}^n \E_\varepsilon\big(f(X_i, \varepsilon)\big)^2 \\
\sum_{i=1}^n \E_\varepsilon\big(f(X_i, \varepsilon)\big) \\
\sum_{i=1}^n \E_\varepsilon\big(f^2(X_i, \varepsilon)\big)
\end{pmatrix}}_{ (1)=\Ac_n^{(n)}(f) }
+
\frac{1}{n}
\underbrace{
\begin{pmatrix}
M_n^{(n)}(f) \\
\sum_{i=1}^n \left( f(X_{\sigma_n(i)}, \varepsilon_i) - \E_\varepsilon\big(f(X_{\sigma_n(i)}, \varepsilon)\big) \right) \\
\sum_{i=1}^n \left( f^2(X_{\sigma_n(i)}, \varepsilon_i) - \E_\varepsilon\big(f^2(X_{\sigma_n(i)}, \varepsilon)\big) \right)
\end{pmatrix}}_{ (2)=\Mc_n^{(n)}(f) } \ .
\end{align*}

From this decomposition, we observe that term \textbf{(1)} will follow from the classical strong law of large numbers for i.i.d. random variables, while term \textbf{(2)} will be handled using the Azuma--Hoeffding inequality. Then we control the remainder term.

\medskip

{\bf Law of large numbers.}
By the strong law of large numbers and independence of the $X_i$'s, we obtain
\begin{align}
    \lim_{n \to \infty} \frac{\Ac_n^{(n)}}{n}(f) 
    = \begin{pmatrix}
    \E\left[  \E_\varepsilon\left[ f\left( X, \varepsilon \right) \right]^2 \right] \\
    \E(Y) \\
    \E(Y^2)
    \end{pmatrix} 
    =
    \begin{pmatrix}
    \E\left[  \E[Y|X]^2 \right] \\
    \E(Y) \\
    \E(Y^2)
    \end{pmatrix} 
    \quad \text{a.s.}
    \label{eq:limit_A_n}
\end{align}

\medskip

{\bf Azuma-Hoeffding.} Let us show that the martingale term \( \Mc_n^{(n)} \) divided by \( n \) vanishes asymptotically. We detail only the argument for the first coordinate.
To this end, note that all components are martingales with bounded increments. We apply the Azuma–Hoeffding inequality since 
$$
| \Delta M_i^{(n)}(f) | \leq 2 \| f \|_\infty^2 \ ,
$$
we have for any $\eta > 0$,
$$
\P \left( \left| M_n^{(n)} \right| \geq \eta n \right) 
\leq 2 \exp \left( - \frac{\eta^2 n}{8 \| f \|_\infty^4} \right) \ .
$$

By the Borel–Cantelli lemma, it follows that $\frac{M_n^{(n)}}{n} \xrightarrow{\text{a.s.}} 0$. The same argument for each coordinate yields
\begin{align}
    \frac{\Mc_n^{(n)}}{n} \xrightarrow{\text{a.s.}} 0 \ .
    \label{eq:limit_M_n}
\end{align}

\medskip

{\bf Controlling the remainder.}
We now examine the expression of the remainder from Proposition~\ref{proposition:double_doob}. Recall that
\begin{align*}
    R_n^{(n)}(f) = \Oc( \|f\|_\infty^2 )
- \half
\sum_{i=1}^n \left( \E_\varepsilon[f(X_{\sigma_n(i-1)}, \varepsilon)] -
\E_\varepsilon[f(X_{\sigma_n(i)}, \varepsilon)] 
\right)^2 \ .
\end{align*}
Notice that, upon playing with indices then invoking exchangeability
\begin{align*}
  & \frac{1}{n} \sum_{i=1}^n \E \left( \E_\varepsilon[f(X_{\sigma_n(i-1)}, \varepsilon)] -
\E_\varepsilon[f(X_{\sigma_n(i)}, \varepsilon)] 
\right)^2 \\
= & \frac{1}{n} \sum_{i=1}^n \E \left( \E_\varepsilon[f(X_{i}, \varepsilon)] -
\E_\varepsilon[f(X_{N_n(i)}, \varepsilon)] 
\right)^2 \\
= & \E \left[\left( \E_\varepsilon[f(X_{1}, \varepsilon)] -
\E_\varepsilon[f(X_{N_n(1)}, \varepsilon)] 
\right)^2\right] \ .
\end{align*}
Applying the estimates of Chatterjee \cite[Corollary A9]{chatterjee2021new}, the integrand under the expectation converges to $0$ in probability. Since $f$ is bounded, the expectation itself goes to $0$.
Hence we have the limit in $L^1(\Omega, \P)$
\begin{align*}
    \lim_{n \rightarrow \infty} \frac{ R_n^{(n)} }{n} = 0 \ .
\end{align*}

At this stage, let us warn the reader that controlling the remainder at the scale of fluctuations is slightly different, hence the precautions taken in the statement of Theorem \ref{thm:limit_sobol}.

\medskip

{\bf Conclusion.} 
Consider the function \begin{align}
    h(t, s_1, s_2) := \frac{t - s_1^2}{s_2 - s_1^2} \ .
    \label{eq:Xi_as_function}
\end{align}
Combining the convergence to zero in probability of the remainder, and the almost sure convergences of ~\eqref{eq:limit_A_n} and~\eqref{eq:limit_M_n}, we have the convergence in probability
\begin{align}
\P-\lim_{n\to\infty} h\left( \frac{Z_n(f)}{n}, \frac{S_n(f)}{n}, \frac{S_n(f^2)}{n} \right)
=
\frac{ \E\left[ \E[Y|X]^2 \right] - \E[Y]^2 }
     { \E[Y^2] - \E[Y]^2 } \ .
\end{align}
Recalling the equality in law of Eq. \eqref{eq:vector_eqlaw}, we obtain the desired convergence in probability.
\end{proof}

\begin{remark}[Convergence]
\label{remark:cv_as_or_proba}
Notice that equality in law for every fixed $n$ translates convergence in probability to convergence in probability, and almost sure convergence to the weaker convergence in probability only. As such, there is no need to attempt upgrading the previous convergence to almost sure convergence. In order to truly recover almost sure convergence, one can proceed as follows. First, invoke concentration around the mean, which is proved classically thanks to the McDiarmid inequality \cite{gamboa2022global}. Then, convergence in probability can be upgraded to convergence in $L^1(\Omega, \P)$  from the convergence in probability and concentration. Finally, a Borel-Cantelli argument and convergence of the mean yield almost sure convergence. 
\end{remark}

\section{Proof of univariate fluctuations (Main Theorem \ref{thm:limit_sobol})}
\label{section:proof_univariate}

Let us recall our decomposition from Eq. \eqref{eq:decomp_sketch}. Before diving into the proof, we start with a series of lemmas that will allow us to study the fluctuations of term $(1)$ and term $(2)$. 

\subsection{Preliminary lemmas}

Let us start with the easiest term $(1)$.
\begin{lemma} 
\label{lemma:preliminary1}
Recall the term $(1)$ defined for any $f \in \Phi$ by
$$
\Ac_n^{(n)} = \Ac_n^{(n)}(f) = \begin{pmatrix}
 \sum_{i=1}^n \E_\varepsilon[f(X_i, \varepsilon)]^2 \\
 \sum_{i=1}^n \E_\varepsilon[f(X_i, \varepsilon)] \\
 \sum_{i=1}^n \E_\varepsilon[f(X_i, \varepsilon)^2]
\end{pmatrix} \ .
$$
Then
$$
\frac{ \Ac_n^{(n)}(f) - \E[\Ac_n^{(n)}(f)] }{\sqrt n}
\stackrel{\Lc}{\longrightarrow} \mathcal{N}(0, \Sigma_0(f,f)) \ ,
$$
where the asymptotic covariance matrix \( \Sigma_0 \in \mathbb{R}^{3 \times 3} \) is given indeed by Eq. \eqref{eq:cov_sigma0_fg}.
\end{lemma}

\begin{proof}
It is just an application of the usual multivariate Central Limit Theorem for i.i.d. random variables.
\end{proof}

Now we shall prepare the ingredient for proving the fluctuations of term $(2)$. The first result concerns the almost sure convergence of the predictable quadratic variation. Recall that the predictable bracket between two discrete-time (vector-valued) martingales  $X = \left( X_n \ ; \ n \geq 0 \right)$ and $Y = \left( Y_n \ ; \ n \geq 0\right)$ is defined by
$$
\langle X, Y \rangle_n = \sum_{i=1}^n \E\left[ \Delta X_i \, (\Delta Y_i)^T \, \big| \, \mathcal{F}_{i-1} \right].
$$

\begin{lemma}
\label{lemma:preliminary2}
Recall
$$ \Mc_n^{(n)} 
 = \Mc_n^{(n)}(f)
 = 
\begin{pmatrix}
M_n^{(n)}(f) \\
N_n(f) \\
N_n(f^2)
\end{pmatrix}
:= \begin{pmatrix}
M_n^{(n)}(f) \\
\sum_{i=1}^n \left( f(X_{\sigma_n(i)}, \varepsilon_i) - \E_\varepsilon\big(f(X_{\sigma_n(i)}, \varepsilon)\big) \right) \\
\sum_{i=1}^n \left( f^2(X_{\sigma_n(i)}, \varepsilon_i) - \E_\varepsilon\big(f^2(X_{\sigma_n(i)}, \varepsilon)\big) \right)
\end{pmatrix}
\ .$$
We have the convergence in probability, for any pair of functions $(f,g) \in \Phi \times \Phi$
\begin{align}
& \P-\lim_{n\to\infty}\frac{1}{n}\langle \Mc^{(n)}(f), \Mc^{(n)}(g) \rangle_n\\
=&
\ \P-\lim_{n\to\infty}\frac{1}{n}
\begin{pmatrix}
\langle M^{(n)}(f),M^{(n)}(g) \rangle_n & \langle M^{(n)}(f), N(g) \rangle_n & \langle M^{(n)}(f), N(g^2) \rangle_n \\
\langle N(f), M^{(n)}(g)\rangle_n & \langle N(f),N(g) \rangle_n & \langle N(f), N(g^2) \rangle_n \\
\langle N(f^2), M^{(n)}(g) \rangle_n & \langle N(f^2), N(g) \rangle_n & \langle N(f^2),N(g^2) \rangle_n
\end{pmatrix}
= \Sigma_1(f,g), \quad a.s.
\end{align}
where $\Sigma_1$ is given indeed by Eq. \eqref{eq:cov_sigma1_fg}.
\end{lemma}
\begin{proof}
See Subsection \ref{subsection:proof_lemma_preliminary2}.
\end{proof}

In the final steps of the proof, we shall combine two convergences in law. One concerns $(\Ac_n{(n)}(f) \ ; \ n \in \N^*)$ from Lemma \ref{lemma:preliminary1} and the other concerns the martingale part $(\Mc_n^{(n)}(f)  \ ; \ n \in \N^*)$. To this end we shall use the following Lemma.
\begin{lemma} 
\label{lemma:preliminary3}
Let us consider a filtration $\F = \left( \mathcal{F}_n \ ; \ n \in \N \right)$ and two sequences of vector-valued random variables $\alpha = ( \alpha_n \ ; \ n \in \N ) $ and $\beta = (\beta_n\ ; \ n \in \N) $ such that:
\begin{itemize}
    \item All the elements in the sequence $\alpha$ are $ \mathcal{F}_0 $-measurable and the sequence converges in distribution to a Gaussian random variable $\mathcal{N}(0, \Sigma_\alpha )$.
    \item $\beta_n$ is $\mathcal{F}_n $-measurable and converges in distribution, conditionally on $ \mathcal{F}_0 $, to a Gaussian random variable $ \mathcal{N}(0, \Sigma_\beta) $, where the covariance matrix $\Sigma_\beta$ is deterministic.
\end{itemize}
Then, the sum $\left( \alpha_n + \beta_n \ ; \ n \in \N \right)$ converges in distribution to a Gaussian random variable $ \mathcal{N}(0, \Sigma_\alpha + \Sigma_\beta ) $.
\end{lemma}
\begin{proof}
Thanks to the classical Cram\'er--Wold device, convergence of vector-valued random variables is deduced from the scalar setting using linear combinations. We use characteristic functions in the scalar setting. Write for $t \in \R$,
\begin{align*}
   \E\left( e^{it(\alpha_n+\beta_n)} \right) 
&=  \E\left( \E\left(  e^{it(\alpha_n+\beta_n)} \ | \ \Fc_0 \right) \right) \\
& =  \E\left( e^{it \alpha_n} \E\left(  e^{it\beta_n} \ | \ \Fc_0 \right) \right) \ .
\end{align*}
Now we invoke the convergence in law, conditionally on $\Fc_0$ so that
\begin{align*}
   \E\left( e^{it(\alpha_n+\beta_n)} \right) 
&=  \E\left( e^{it \alpha_n} \left( \exp\left(-\frac{t^2 \Sigma_\beta^2}{2}\right) + o(1) \right) \right) \ .
\end{align*}
By dominated convergence the $o(1)$ remains an $o(1)$ upon integration. In the end
\begin{align*}
   \E\left( e^{it(\alpha_n+\beta_n)} \right) 
&=  o(1) + 
    \exp\left( -\frac{t \Sigma_\beta^2}{2} \right)
    \E\left( e^{it \alpha_n} \right) \\
&=  o(1) + 
    \exp\left( -\frac{t^2 (\Sigma_\beta^2+\Sigma_\alpha^2)}{2} \right) \ .
\end{align*}
\end{proof}

Furthermore, without any conditions on the bounded measurable function $f$, we have the following control on the remainder $R_n^{(n)}(f)$.
\begin{lemma}
\label{lemma:remainder_control}
We have the limit in $L^2(\Omega, \P)$ and in probability
$$
  \lim_{n \rightarrow \infty }n^{-\half} 
  \left( R_n^{(n)}(f)- \E R_n^{(n)}(f) \right)
  = 0 \ .
$$
Furthermore, the quantity at hand remains bounded in $L^2(\Omega, \P)$.
\end{lemma}

\begin{proof}
See Subsection \ref{subsection:remainder_control}.
\end{proof}

Now we have all the ingredients to finish the proof of asymptotic normality.
 
\subsection{Proof of the asymptotic normality \texorpdfstring{\eqref{eq:main_thm_TCL}}{} }

Essentially we need to prove the asymptotic normality of the martingale part. To this end we shall apply the multivariate martingale Central Limit Theorem to $\Mc^{(n)}$. The usual Lindeberg conditions are
\begin{itemize}
    \item \textbf{The convergence of the bracket:}    
     \begin{equation}
        \frac1n \langle \Mc^{(n)}(f), \Mc^{(n)}(g) \rangle_n 
        \xrightarrow[n \to \infty]{\P} \Sigma_1(f,g) \ ,
    \end{equation}
    where $\Sigma_1(f,g)$ is provided by Lemma \ref{lemma:preliminary2}.

    \item \textbf{The Lindeberg condition:}
    \begin{equation}
        \forall \varepsilon > 0, \quad \frac 1n \sum_{k=1}^n \E\left[ \|\Delta \Mc^{(n)}_k(f)\|^2 \, \1_{\{\|\Delta \Mc^{(n)}_k(f)\| > \varepsilon\}} \mid \mathcal{F}_{k-1} \right]
        \xrightarrow[n \to \infty]{\P} 0,
    \end{equation}
    is obvious since increments $\left( \|\Delta \Mc^{(n)}_k(f)\| \ ; \ k \geq 1 \right)$ are bounded. This holds also for $g$.
\end{itemize}
Under these two conditions we have the convergence in law
$$
\frac{1}{\sqrt n} \Mc^{(n)}_n(f) \xrightarrow[n \to \infty]{\mathcal{L}} \mathcal{N}(0, \Sigma_1(f,f)) \ .
$$
Then since this convergence holds conditionally to $\Fc_0$, this can be combined with the convergence from Lemma \ref{lemma:preliminary1} and \ref{lemma:preliminary3} in order to obtain that
\begin{align}
\frac{ \Ac_n^{(n)}(f) - \E[\Ac_n^{(n)}(f)] }{\sqrt n}
+
\frac{1}{\sqrt n} \Mc^{(n)}_n(f)
\xrightarrow[n \to \infty]{\mathcal{L}}
\Nc\left( 0, \Sigma_0(f,f)+\Sigma_1(f,f) \right) \ .
\label{eq:clt_partial}
\end{align}

Now recall from the decomposition \eqref{eq:decomp_sketch} that
$$
    \widehat{\theta}_n = \widehat{\theta}_n(f) = \frac{1}{n}\begin{pmatrix}
R_n^{(n)}(f) \\ 0 \\ 0
\end{pmatrix}
+ \frac{\Ac_n^{(n)}(f)}{n}
+ \frac{\Mc_n^{(n)}(f)}{n}
\ .
$$
Then we invoke Slutsky's Lemma on Eq. \eqref{eq:clt_partial} and Lemma \ref{lemma:remainder_control}. This yields
\begin{align}
& 
\sqrt{n}\left( \widehat{\theta}_n - \E[\widehat{\theta}_n] \right)
\nonumber \\
= & \ 
\frac{1}{\sqrt n}
\begin{pmatrix}
R_n^{(n)}(f) - \E[ R_n^{(n)}(f) ]\\
0 \\
0
\end{pmatrix}
+
\frac{ \Ac_n^{(n)}(f) - \E[\Ac_n^{(n)}(f)] }{\sqrt n}
+
\frac{1}{\sqrt n} \Mc^{(n)}_n(f)
\nonumber 
\\
\xrightarrow[n \to \infty]{\mathcal{L}} & \ 
\mathcal{N}(0,\Sigma_0(f,f)+\Sigma_1(f,f)) \ .
\label{eq:clt_final}
\end{align}

\medskip

{\bf Analysis of $\Delta_n$, relation to remainder and $\E[\widehat{\theta}_n(f)]$.}
The two expressions for $\Delta_n$ in Eq. \eqref{eq:def_Delta_n} coincide because the $\left( (X_i, X_{N_n(i)}) \ ; \ 1 \leq i \leq n \right)$ are exchangeable
\begin{align*}
    \Delta_n
= & \ \frac{1}{n} \sum_{i=1}^n \E \left( \varphi(X_{\sigma_n(i-1)}) - \varphi(X_{\sigma_n(i)}) \right)^2 \\
= & \ \frac{1}{n} \sum_{i=1}^n \E \left( \varphi(X_{i}) - \varphi(X_{N_n(i)}) \right)^2 \\
= & \ \E \left[\left( \varphi(X_{1}) - \varphi(X_{N_n(1)}) \right)^2\right] \ .
\end{align*}
Recalling the expression of the remainder from Proposition~\ref{proposition:double_doob},
\begin{align*}
    R_n^{(n)}(f) & = \Oc( \|f\|_\infty^2 )
- \half
\sum_{i=1}^n \left( \E_\varepsilon[f(X_{\sigma_n(i-1)}, \varepsilon)] -
\E_\varepsilon[f(X_{\sigma_n(i)}, \varepsilon)] 
\right)^2 \\
& = \Oc( \|f\|_\infty^2 )
- \half
\sum_{i=1}^n \left( \varphi(X_{\sigma_n(i-1)}) - \varphi(X_{\sigma_n(i)}) \right)^2 
\ ,
\end{align*}
we have
\begin{align*}
    \E[ R_n^{(n)}(f)] & = \Oc( \|f\|_\infty^2 )
- \half n \Delta_n \ .
\end{align*}
As such
\begin{align}
\label{eq:expectation_theta_n}
\E[\widehat{\theta}_n(f)]
    & = 
    \begin{pmatrix}
    \frac{1}{n} \E[ R_n^{(n)}(f) ] \\ 0 \\ 0
    \end{pmatrix} 
    + 
    \theta^*(f)
    = 
    \begin{pmatrix}
    \Oc\left(\frac{\|f\|_\infty}{n}\right) - \half \Delta_n \\ 0 \\ 0
    \end{pmatrix} 
    + 
    \theta^*
    \xrightarrow[n \to \infty]{}
    \ \theta^*(f)
    \ .    
\end{align}

\medskip

{\bf Delta method.} At this stage, to obtain the final result we just need to apply the so-called Delta method to the function $h$ from Eq. \eqref{eq:Xi_as_function}, which we recall
\begin{align*}
    h(t, s_1, s_2) := \frac{t - s_1^2}{s_2 - s_1^2} \ .
\end{align*}
Set $\mathcal{D} = \{ \theta=(t,s_{1},s_{2}) \in \R^{3} \! : \;s_{2}>s_{1}^{2}\}$ and specialize the value $\E[ \widehat{\theta}_n(f) ]$ in Eq. \eqref{eq:expectation_theta_n}.
This evaluates to
$$
h\left( \E[ \widehat{\theta}_n] \right)
=
h(\theta^*) - \frac{\Delta_n}{2 \Var(Y)}
=
\rho^{\mathrm{Sobol'}} - \frac{\Delta_n}{2\Var(Y)}
\ .
$$
Since $\Var(Y)=s_{2}-s_{1}^{2}>0$, the limiting point $\theta^*$ lies in the interior of $\mathcal D$, where $h$ is smooth.
Its gradient vector is non-vanishing and takes the form
$$
\nabla h(t,s_1,s_2)=
\left(
\dfrac{1}{s_{2}-s_{1}^{2}},\;
\dfrac{2s_{1}(t-s_{2})}{(s_{2}-s_{1}^{2})^{2}},\;
-\dfrac{t-s_{1}^{2}}{(s_{2}-s_{1}^{2})^{2}}
\right)^{\!\top} \ .
$$
The usual regularity conditions are satisfied, so the multivariate Delta method applies to the joint CLT~\eqref{eq:clt_final}. Consequently,
$$
\sqrt n \left[ \widehat{\rho}_n^{\mathrm{Sobol'}}
        - \left( \rho^{\mathrm{Sobol'}} - \frac{\Delta_n}{2\Var(Y)} \right)
        \right] 
= 
\sqrt n \left[ h\left( \widehat{\theta}_n \right)
        - h\left( \E[ \widehat{\theta}_n] \right) 
        \right]
\xrightarrow[n \to \infty]{\mathcal{L}}
\Nc\left( 0, \sigma^2_{\mathrm{Sobol'}} \right) \ ,
$$
with
$$
\sigma^2_{\mathrm{Sobol'}} =
\nabla h(\theta^*)^\top \ \left( \Sigma_{0}(f,f)+\Sigma_{1}(f,f) \right) \ \nabla h(\theta^*) \ .
$$
Upon checking that the gradient with specialized values is indeed the vector $v(f)$ of Eq. \eqref{eq:def_g}, we recover indeed the expression announced in Eq. \eqref{eq:def_sigmaXY}.
This completes the proof of the Sobol estimator’s central limit theorem. 

\subsection{Proof of Lemma \ref{lemma:preliminary2}}
\label{subsection:proof_lemma_preliminary2}
Before diving into the details of the proof, let us give the general strategy.
We will implement the following procedure for the bracket $\langle M^{(n)}(f),M^{(n)}(g) \rangle_n$.
\begin{itemize}
    \item Step 1: We shall compute the bracket $\langle M^{(n)}(f),M^{(n)}(g) \rangle_j$ by using discrete rules of stochastic calculus.
    \item Step 2: We shall compute the conditional expectation with respect to $\mathcal F_0$, i.e. 
    $$ \E\left[ \langle M^{(n)}(f),M^{(n)}(g) \rangle_j \mid \mathcal F_0 \right]$$
    and observe that 
    $$ \left( \langle M^{(n)}(f),M^{(n)}(g) \rangle_j - \E[ \langle M^{(n)}(f),M^{(n)}(g) \rangle_j \mid \mathcal F_0] \ ; \ 1 \leq j \leq n \right) 
    $$
    is a martingale with bounded increments. This implies that, almost surely
    \begin{align}
        \label{eq:zero_limit}
        \lim_{n\to\infty}\frac 1n \left( \langle M^{(n)}(f),M^{(n)}(g) \rangle_n - 
        \E\left[ \langle M^{(n)}(f),M^{(n)}(g) \rangle_n \ | \ \Fc_0 \right] \right) = 0 \ .
    \end{align}
    We are thus reduced to computing the limit $\lim_{n} \frac1n \E\left[ \langle M^{(n)}(f),M^{(n)}(g) \rangle_n \ | \ \Fc_0 \right]$.
    \item Step 3: The expressions of $\E\left[ \langle M^{(n)}(f),M^{(n)}(g) \rangle_n \ | \ \Fc_0 \right]$ will be given by sum with terms involving $X_{\sigma_n(i-1)}$, $X_{\sigma_n(i+1)}$ and $X_{\sigma_n(i)}$. We replace $X_{\sigma_n(i-1)}$, $X_{\sigma_n(i+1)}$ by $X_{\sigma_n(i)}$ in the relevant formula by adding and subtracting suitable terms accordingly. The final aim is to obtain a formula of the type
    $$ \E\left[ \langle M^{(n)}(f),M^{(n)}(g) \rangle_n \ | \ \Fc_0 \right]
       = R_n^\Mc + \sum_{i=1}^n H_\Mc(X_i) \ ,$$
    for a suitable function $H_\Mc$ and a remainder term. 
    In fact, in agreement with our notations in \eqref{eq:cov_sigma1_fg}, we shall see that
    \begin{align} 
      \label{eq:HM}
    H_{\Mc}(x) = & \ \Sigma_a(x,f,g) \odot \Sigma_b(x,f,g)
    \end{align}
    \item Step 4: We will prove the limit in probability
    \begin{align}
      \label{eq:remainder_RM_asymptotics}
      \P-\lim_{n\to\infty} \frac{R_n^\Mc}{n} & = \ 0 \ .
    \end{align}
    \item Conclusion: The final limit is obtained by the strong law of large numbers for the i.i.d. random variables $\left( H_\Mc(X_i) \ ; \ i \in \N^* \right)$. All in all, we find that
    \begin{align*}
        \P-\lim_{n\to\infty}\frac{1}{n}\langle \Mc^{(n)}(f), \Mc^{(n)}(g) \rangle_n
    = & \ \E\left[ H_\Mc(X) \right] \ .
    \end{align*}    
\end{itemize}
{\bf Step 1:} It is worth noticing that
we have
$
N_n(f) = \sum_{i=1}^n \Delta \Mk_i^{(n)}(f) 
$
and likewise for $N_n(f^2)$.
Recalling the expression of Eq. \eqref{eq:main_martingale_part}
$$
M_j^{(n)} = M_j^{(n)}(f) :=  \sum_{i=1}^{j} \left( f(X_{\sigma_n(i-1)}, \varepsilon_{i-1}) + \E_\varepsilon[f(X_{\sigma_n(i+1)}, \varepsilon)] \right)
     \cdot \Delta \Mk_i^{(n)}(f) \ ,
$$
we see that we need to compute the bracket of a vector martingale, where all components are discrete stochastic integrals with respect to the {\it same} basic martingale $\left( \Delta \Mk_i^{(n)}(f) ; \ 1 \leq i \leq n \right)$. Although we are using different $f$'s, the basic martingale remains the same.
Write
\begin{align*}
  \Mc_j^{(n)}(f)
= & \
\begin{pmatrix}
M_j^{(n)}(f) \\
N_j(f) \\
N_j(f^2)
\end{pmatrix}
= \ \sum_{i=1}^j
a_i(f)
\odot
\xi_i(f)
\ ,
\end{align*}
where the symbol $\odot$ stands for the Hadamard (component-wise) product, and
$$
a_i(f) := 
\begin{pmatrix}
f(X_{\sigma_n(i-1)}, \varepsilon_{i-1}) + \E_\varepsilon[f(X_{\sigma_n(i+1)}, \varepsilon)] \\
1 \\
1
\end{pmatrix}
\ , \quad
\xi_i(f) :=
\begin{pmatrix}
\Delta \Mk_i^{(n)}(f) \\
\Delta \Mk_i^{(n)}(f) \\
\Delta \Mk_i^{(n)}(f^2)
\end{pmatrix} \ .
$$

Now, we invoke the rules of discrete stochastic calculus and matrix algebra to obtain
\begin{align*}
  & \langle \Mc^{(n)}(f), \Mc^{(n)}(g) \rangle_j \\
= & \ \sum_{i=1}^j \E\left[ \left( a_i(f) \odot \xi_i(f) \right)
                            \left( a_i(g) \odot \xi_i(g) \right)^T
                            \ | \ \Fc_{i-1} \right] \\
= & \ \sum_{i=1}^j \E\left[ \left( a_i(f) a_i(g)^T \right)
                            \odot
                            \left( \xi_i(f) \xi_i(g)^T \right)
                            \ | \ \Fc_{i-1} \right] \\
= & \ \sum_{i=1}^j \left( a_i(f) a_i(g)^T \right)
                            \odot
                   \E\left[ \xi_i(f) \xi_i(g)^T
                            \ | \ \Fc_{i-1} \right] \ .
\end{align*}

Then recall the notation of Eq. \eqref{eq:sigma_b} and notice that
$$
\E\left[ \xi_i(f) \xi_i(g)^T \ | \ \Fc_{i-1} \right]
= 
\Cov_\varepsilon\left( 
\begin{pmatrix}
f(x, \varepsilon)\\
f(x, \varepsilon)\\
f^2(x, \varepsilon)
\end{pmatrix} ,
\begin{pmatrix}
g(x, \varepsilon)\\
g(x, \varepsilon)\\
g^2(x, \varepsilon)
\end{pmatrix}
\right)_{|x = X_{\sigma_n(i)}} 
= \Sigma_b(X_{\sigma_n(i)} , f, g)
\ , 
$$
which is in fact $\Fc_0$-measurable. Here $\Cov_\varepsilon$ means that the covariance is computed by averaging over $\varepsilon$ 
while keeping $x$ fixed. The variable $x$ is then specialized to $x=X_{\sigma_n(i)}$.

In the end
\begin{align*}
    \langle \Mc^{(n)}(f), \Mc^{(n)}(g) \rangle_j
= & \ \sum_{i=1}^j \left( a_i(f) a_i(g)^T \right)
                            \odot
                   \Sigma_b( X_{\sigma_n(i)}, f, g) \ .
\end{align*}

\medskip 

{\bf Step 2:}
Now notice that
\begin{align*}
   & \langle \Mc^{(n)}(f), \Mc^{(n)}(g) \rangle_j
   - \E\left[ \langle \Mc^{(n)}(f), \Mc^{(n)}(g) \rangle_j \ | \ \Fc_0 \right] \\
 = & \ \sum_{i=1}^j \left[ a_i(f) a_i(g)^T 
                      - \E\left( a_i(f) a_i(g)^T \ | \ \Fc_0\right)
                \right]
                            \odot
                   \Sigma_b( X_{\sigma_n(i)}, f, g) \ .
\end{align*}
is an $\F$-martingale with bounded increments. Thus Eq. \eqref{eq:zero_limit} holds. 
In order to prove the desired result, as announced, we only need to study 
$$
\E\left[ \langle \Mc_j^{(n)}(f), \Mc_j^{(n)}(g) \rangle_j \ | \ \Fc_0 \right]
= \ \sum_{i=1}^j \E\left( a_i(f) a_i(g)^T \ | \ \Fc_0\right)
                  \odot\Sigma_b( X_{\sigma_n(i)}, f, g)
$$

\medskip

{\bf Step 3:}
Write
$$
   b_i(f) := \E\left( a_i(f) \ | \ \Fc_0\right)
          = 
\begin{pmatrix}
\E_\varepsilon[ f(X_{\sigma_n(i-1)}, \varepsilon) ] + \E_\varepsilon[f(X_{\sigma_n(i+1)}, \varepsilon)] \\
1 \\
1
\end{pmatrix} 
\ .
$$
By the mean-covariance decomposition
\begin{align*}
  & \ \E\left( a_i(f) a_i(g)^T \ | \ \Fc_0\right) \\
= & \ b_i(f) b_i(g)^T
    + \E\left( \left( a_i(f) - b_i(f) \right)
               \left( a_i(g) - b_i(f) \right)^T \ | \ \Fc_0\right) \\
= & \ b_i(f) b_i(g)^T
    + \Cov_\varepsilon[ f(X_{\sigma_n(i-1)}, \varepsilon),
      g(X_{\sigma_n(i-1)}, \varepsilon)
      ]
      E_{11} \ .
\end{align*}
Here $E_{11} = e_1 e_1^T \in M_3(\R)$ is the elementary matrix associated to the first canonical basis vector $e_1 = \begin{pmatrix} 1 \\ 0 \\ 0 \end{pmatrix}$.
Upon substituting and permuting indices, we find
\begin{align*}
  & \ \E\left[ \langle \Mc_n^{(n)}(f), \Mc_j^{(n)}(g) \rangle \ | \ \Fc_0 \right]\\
= & \ \sum_{i=1}^n \E\left( a_i(f) a_i(g)^T \ | \ \Fc_0\right)
                  \odot
                  \Sigma_b( X_{\sigma_n(i)}, f, g) \\
= & \ \sum_{i=1}^n \left(  b_i(f) b_i(g)^T
    + \Cov_\varepsilon[ f(X_{\sigma_n(i-1)}, \varepsilon),
      g(X_{\sigma_n(i-1)}, \varepsilon)
      ]
      E_{11} \right)
                  \odot
                  \Sigma_b( X_{\sigma_n(i)}, f, g) \\
= & \ R_n^\Mc
    + \sum_{i=1}^n H_\Mc(X_i) \ ,
\end{align*}
where the remainder $R_n^\Mc$ is of the form
$$
   R_n^\Mc := \sum_{i=1}^n R_{n,i} \odot \Sigma_b(X_{\sigma_n(i)}, f, g) \ ,
$$
\begin{align*}
  R_{n,i} := & 
  \begin{pmatrix}
    \E_\varepsilon[ f(X_{\sigma_n(i-1)}, \varepsilon) ] + \E_\varepsilon[f(X_{\sigma_n(i+1)}, \varepsilon)] \\
    1 \\
    1
    \end{pmatrix}
    \begin{pmatrix}
    \E_\varepsilon[ g(X_{\sigma_n(i-1)}, \varepsilon) ] + \E_\varepsilon[g(X_{\sigma_n(i+1)}, \varepsilon)] \\
    1 \\
    1
    \end{pmatrix}^T 
    \\ & -
    \begin{pmatrix}
      2 \E_\varepsilon[ f(X_{\sigma_n(i)}, \varepsilon) ] \\
      1 \\
      1
      \end{pmatrix}
      \begin{pmatrix}
      2 \E_\varepsilon[ g(X_{\sigma_n(i)}, \varepsilon) ] \\
      1 \\
      1
      \end{pmatrix}^T
  \\ & \ +
  \left(
    \Cov_\varepsilon[ f(X_{\sigma_n(i-1)}, \varepsilon),
      g(X_{\sigma_n(i-1)}, \varepsilon)
      ]
    - 
    \Cov_\varepsilon[ f(X_{\sigma_n(i)}, \varepsilon),
      g(X_{\sigma_n(i)}, \varepsilon)
      ]
  \right)
      E_{11} \ .
\end{align*}
and
\begin{align*}
H_{\Mc}(x) &= 
\left(
\begin{pmatrix}
2\,\E_{\varepsilon} \left[f(x,\varepsilon)\right] \\
1 \\
1
\end{pmatrix}
\begin{pmatrix}
2\,\E_{\varepsilon} \left[g(x,\varepsilon)\right] \\
1 \\
1
\end{pmatrix}^{\top}
+ \operatorname{Cov}_{\varepsilon}\!\left[f(x,\varepsilon),\,g(x,\varepsilon)\right]\,
E_{11}
\right)
\odot
\Sigma_b(x,f,g)\nonumber\\
&=\Sigma_a(x,f,g)\odot
\Sigma_b(x,f,g) \ .
\end{align*}
We are thus done with step 3. Notice that $H_\Mc$ agrees with the announced form in Eq. \eqref{eq:HM}.

\medskip

{\bf Step 4:}
Now let us consider the remainder $R^\Mc_n$. To that end, we consider the $\ell_1$ norm of matrices denoted $\Vert\cdot\Vert,$ that is if $A=(a_{ij})$, we have
$$\Vert A\Vert=\sum_{ij}\vert a_{ij}\vert.$$ 
Notice from the definition $\Sigma_b(x,f,g)$ in Eq. \eqref{eq:sigma_b}, we have the elementary bound
\begin{align*}
\max_{i,j} |(\Sigma_b(X,f,g))_{i,j}|
\leq & \ 2 \max_{k,l \in \{1,2\} }( \|f\|_\infty^{k} \|g\|_\infty^{l} ) \\
\leq & \ 2 \ (1+\|f\|_\infty)^2 \ (1+\|g\|_\infty)^2 \ .
\end{align*}
Because $\Vert A\odot B\Vert\leq \Vert A\Vert \max_{ij}\vert b_{ij}\vert\leq \Vert A\Vert\Vert B\Vert$, we see that
$$
\| R^\Mc_n \| 
\leq 
\ \sum_{i=1}^n \| R_{n,i} \| \ \max_{i,j} |(\Sigma_b(X,f,g))_{i,j}|
\leq
\ 2 \ (1+\|f\|_\infty)^2 \ (1+\|g\|_\infty)^2
\sum_{i=1}^n \| R_{n,i} \|
\ .
$$
As such, we are done upon proving
$$
\P-\lim_{n \rightarrow \infty} \frac{1}{n} \sum_{i=1}^n \| R_{n,i} \| = 0 \ .
$$

Let us write
$$R_{n,i}=A_{n,i}+B_{n,i}$$
with
\begin{align*}
  A_{n,i} := & 
  \begin{pmatrix}
    \E_\varepsilon[ f(X_{\sigma_n(i-1)}, \varepsilon) ] + \E_\varepsilon[f(X_{\sigma_n(i+1)}, \varepsilon)] \\
    1 \\
    1
    \end{pmatrix}
    \begin{pmatrix}
    \E_\varepsilon[ g(X_{\sigma_n(i-1)}, \varepsilon) ] + \E_\varepsilon[g(X_{\sigma_n(i+1)}, \varepsilon)] \\
    1 \\
    1
    \end{pmatrix}^T 
    \\ & -
    \begin{pmatrix}
      2 \E_\varepsilon[ f(X_{\sigma_n(i)}, \varepsilon) ] \\
      1 \\
      1
      \end{pmatrix}
      \begin{pmatrix}
      2 \E_\varepsilon[ g(X_{\sigma_n(i)}, \varepsilon) ] \\
      1 \\
      1
      \end{pmatrix}^T
  \\ B_{n,i}:=& \ 
  \left(
    \Cov_\varepsilon[ f(X_{\sigma_n(i-1)}, \varepsilon),
      g(X_{\sigma_n(i-1)}, \varepsilon)
      ]
    - 
    \Cov_\varepsilon[ f(X_{\sigma_n(i)}, \varepsilon),
      g(X_{\sigma_n(i)}, \varepsilon)
      ]
  \right)
      E_{11} \ .
\end{align*}
By the triangle inequality, we have $\Vert R_{n,i}\Vert\leq \Vert A_{n,i}\Vert+\Vert B_{n,i}\Vert.$ In order to control, we shall use the following structure already visible in $B_{n,i}$. There exist a bounded measurable function $B$ such that
\begin{align}
\label{eq:control_B}
\Vert B_{n,i}\Vert
\leq
\left|
B(X_{\sigma_n(i)}) - B(X_{\sigma_n(i-1)}) 
\right| \ .
\end{align}
One can clearly take $B(x) = \Cov_\varepsilon\left[ f(x,\varepsilon), g(x,\varepsilon) \right]$. Let us show that there exist a bounded measurable $A$ such that
\begin{align}
\label{eq:control_A}
\Vert A_{n,i}\Vert
\leq & 
\left|
A^{(1)}(X_{\sigma_n(i)}) - A^{(1)}(X_{\sigma_n(i-1)}) 
\right| 
+
\left|
A^{(1)}(X_{\sigma_n(i+1)}) - A^{(1)}(X_{\sigma_n(i)}) \right| \\
  & \quad +
\left|
A^{(2)}(X_{\sigma_n(i)}) - A^{(2)}(X_{\sigma_n(i-1)}) 
\right| 
+
\left|
A^{(2)}(X_{\sigma_n(i+1)}) - A^{(2)}(X_{\sigma_n(i)}) 
\right|
\nonumber
\ .
\end{align}
In order to do so, we compute explicitly the individual entries of $A_{n,i}$. We have
$$
\forall (k,l) \in \{2,3\}^2, \ 
[ A_{n,i} ]_{k,l} = 0 \ .
$$
For the remaining off-diagonal entries, we have
$$
  [ A_{n,i} ]_{1,2} 
= [ A_{n,i} ]_{1,3}
=  \E_\varepsilon[ f(X_{\sigma_n(i-1)}, \varepsilon) ]
 + \E_\varepsilon[ f(X_{\sigma_n(i+1)}, \varepsilon) ]
- 2 \E_\varepsilon[f(X_{\sigma_n(i)}, \varepsilon)] \ ,
$$
$$
  [ A_{n,i} ]_{2,1} 
= [ A_{n,i} ]_{3,1}
=  \E_\varepsilon[ g(X_{\sigma_n(i-1)}, \varepsilon) ]
 + \E_\varepsilon[ g(X_{\sigma_n(i+1)}, \varepsilon) ]
- 2 \E_\varepsilon[g(X_{\sigma_n(i)}, \varepsilon)] \ .
$$
And 
\begin{align*}
  [ A_{n,i} ]_{1,1} 
= &
 \left(
   \E_\varepsilon[ f(X_{\sigma_n(i-1)}, \varepsilon) ]
 + \E_\varepsilon[ f(X_{\sigma_n(i+1)}, \varepsilon)]
 \right)
 \left(
   \E_\varepsilon[ g(X_{\sigma_n(i-1)}, \varepsilon) ]
 + \E_\varepsilon[ g(X_{\sigma_n(i+1)}, \varepsilon)]
 \right)\\
  & \quad 
- 4 \E_\varepsilon[f(X_{\sigma_n(i)}, \varepsilon)]
    \E_\varepsilon[g(X_{\sigma_n(i)}, \varepsilon)] \ .
\end{align*}

Let $a_i$, $b_i$ be two sequences of real numbers bounded by the same constant $K>0$, we have the following identity
\begin{align*}
  & (a_{i-1}+a_{i+1})(b_{i-1}+b_{i+1}) - 4 a_i b_i\\
= & \ (a_{i+1}-a_{i})(b_{i-1}+b_{i+1})
    + (a_{i-1}-a_{i})(b_{i-1}+b_{i+1})
    + 2 a_{i} (b_{i-1}+b_{i+1}) - 4 a_i b_i \\
= & \ (a_{i+1}-a_{i})(b_{i-1}+b_{i+1})
    + (a_{i-1}-a_{i})(b_{i-1}+b_{i+1}) \\
  & \ 
    + 2 a_{i} (b_{i-1}+b_{i+1} - 2b_i) \ .
\end{align*}
As a consequence, we have by the triangle inequality and simple bounds
$$
\left| (a_{i-1}+a_{i+1})(b_{i-1}+b_{i+1}) - 4 a_i b_i \right|
\leq 2 K
     \left( |a_{i+1}-a_{i}| + |a_{i}-a_{i-1}|
          + |b_{i+1}-b_{i}| + |b_{i}-b_{i-1}| \right) \ .
$$
We shall apply this identity for $a_i=\E_\varepsilon[f(X_{\sigma_n(i)}, \varepsilon)]$ and $b_i=\E_\varepsilon[g(X_{\sigma_n(i)}, \varepsilon)]$. We can take $K = \max( \|f\|_\infty, \|g\|_\infty )$. This yields that
\begin{align*}
       \left| [ A_{n,i} ]_{1,1}  \right|
\leq & \ 2 K
     \left( |a_{i+1}-a_{i}| + |a_{i}-a_{i-1}|
          + |b_{i+1}-b_{i}| + |b_{i}-b_{i-1}| \right) \ .
\end{align*}
Now notice that all the other entries of $A_{n,i}$ satisfy the same bound (with a different $K>0$ for each). As such there exist a $\widetilde{K}>0$ such that
\begin{align*}
       \left\| A_{n,i} \right\|
\leq & \ \widetilde{K}
     \left( |a_{i+1}-a_{i}| + |a_{i}-a_{i-1}|
          + |b_{i+1}-b_{i}| + |b_{i}-b_{i-1}| \right) \ .
\end{align*}
By taking $A_i^{(1)}(x) := \widetilde{K} \E_\varepsilon[f(x, \varepsilon)]$, and
$A_i^{(2)}(x) := \widetilde{K} \E_\varepsilon[g(x, \varepsilon)]$, we have indeed established Eq. \eqref{eq:control_A}.

Now starting from Eq. \eqref{eq:control_A} and Eq. \eqref{eq:control_B}, we have by changing indices and using exchangeability
\begin{align*}
     & \ \frac{1}{n} \E \sum_{i=1}^n \| R_{n,i} \| \\
\leq & \ \frac{1}{n} \sum_{i=1}^n \E\left[ 
\left|
A^{(1)}(X_{\sigma_n(i)}) - A^{(1)}(X_{\sigma_n(i-1)}) 
\right| + 
\left|
A^{(2)}(X_{\sigma_n(i)}) - A^{(2)}(X_{\sigma_n(i-1)}) 
\right| + 
\left|
B(X_{\sigma_n(i)}) - B(X_{\sigma_n(i-1)}) 
\right| \right] \\
\leq & \ \E \left| A^{(1)}(X_{1}) - A^{(1)}(X_{N_n(1)}) \right|
       + \E \left| A^{(2)}(X_{1}) - A^{(2)}(X_{N_n(1)}) \right|
       + \E \left| B(X_{1}) - B(X_{N_n(1)}) \right| \ .
\end{align*}

The estimates of Chatterjee \cite[Corollary A9]{chatterjee2021new} tells us this expectation vanishes as $n \rightarrow \infty$. Hence convergence in $L^1( \Omega, \P)$ to zero, which yields convergence in probability as announced.

\subsection{Proof of Lemma \ref{lemma:remainder_control}}
\label{subsection:remainder_control}

We now examine the expression of the remainder from Proposition~\ref{proposition:double_doob}. Recall that
\begin{align*}
    R_n^{(n)}(f) 
& = \Oc( \|f\|_\infty^2 )
- \half
\sum_{i=1}^n \left( \E_\varepsilon[f(X_{\sigma_n(i-1)}, \varepsilon)] -
\E_\varepsilon[f(X_{\sigma_n(i)}, \varepsilon)] 
\right)^2 \\
& = \Oc( \|f\|_\infty^2 )
- \half
\sum_{i=1}^n \left( \varphi(X_{\sigma_n(i-1)}) - \varphi(X_{\sigma_n(i)}) \right)^2 \ .
\end{align*}

For shorter notations, let us write
$$
  h(x,y) := \left( \varphi(x) - \varphi(y) \right)^2 \ .
$$
Following Chatterjee's estimate \cite[Corollary A9]{chatterjee2021new}, a useful statement is that, for $i=1,2$
\begin{align}
    \label{eq:CV_h}
    \lim_{n \rightarrow \infty} h(X_i, X_{N_n(i)}) = 0
\end{align}
in probability and in $L^p(\Omega, \P)$. To establish the result, we compute the variance
\begin{align*}
  & \ \Var\left[ n^{-\half} R_n^{(n)}(f) \right]
=   \ \E\left[ \left| n^{-\half} R_n^{(n)}(f)
                  - \E n^{-\half} R_n^{(n)}(f) \right|^2 \right] \\
\leq & \ \frac{1}{n} \Oc( \|f\|_\infty^2 )
     + 
     \frac{1}{n}
     \Var\left[ \sum_{i=1}^n \left( \varphi(X_{\sigma_n(i-1)}) - \varphi(X_{\sigma_n(i)}) \right)^2 \right] \\
= & \ \frac{1}{n} \Oc( \|f\|_\infty^2 )
     + 
     \frac{1}{n}
     \sum_{i,j=1}^n  \Cov\left(
      \left( \varphi(X_{\sigma_n(i-1)}) - \varphi(X_{\sigma_n(i)}) \right)^2 ,
      \left( \varphi(X_{\sigma_n(j-1)}) - \varphi(X_{\sigma_n(j)}) \right)^2
\right) \ .
\end{align*}
Upon re-indexing the double sum, and forcing the appearance of nearest neighbors, we have
\begin{align*}
   & \E\left[ \left| n^{-\half} R_n^{(n)}(f)
                  - \E n^{-\half} R_n^{(n)}(f) \right|^2 \right] \\
\leq & \ \frac{1}{n} \Oc( \|f\|_\infty^2 )
     + 
     \frac{1}{n}
     \Var\left[ \sum_{i=1}^n \left( \varphi(X_{i}) - \varphi(X_{N_n(i)}) \right)^2 \right] \\
= & \ \frac{1}{n} \Oc( \|f\|_\infty^2 )
     + 
     \frac{1}{n}
     \sum_{i,j=1}^n  \Cov\left(
      \left( \varphi(X_{i}) - \varphi(X_{N_n(i)}) \right)^2 ,
      \left( \varphi(X_{j}) - \varphi(X_{N_n(j)}) \right)^2
\right) \\
= & \ \frac{1}{n} \Oc( \|f\|_\infty^2 )
     + 
     \frac{1}{n} n \Var\left(
      \left( \varphi(X_{1}) - \varphi(X_{N_n(1)}) \right)^2 \right) \\
  & \quad 
     + \frac{1}{n} n(n-1)
       \Cov\left(
      \left( \varphi(X_{1}) - \varphi(X_{N_n(1)}) \right)^2 ,
      \left( \varphi(X_{2}) - \varphi(X_{N_n(2)}) \right)^2
\right) \\
= & \ \frac{1}{n} \Oc( \|f\|_\infty^2 )
     + 
     \Var\left(
      h(X_{1}, X_{N_n(1)}) \right) \\
  & \quad 
     + (n-1)
       \Cov\left(
      \left( \varphi(X_{1}) - \varphi(X_{N_n(1)}) \right)^2 ,
      \left( \varphi(X_{2}) - \varphi(X_{N_n(2)}) \right)^2
\right) \ .
\end{align*}
The first term clearly goes to zero, the second as well by Chatterjee's estimates, while remaining bounded. As such
\begin{align*}
   & \limsup_{n \rightarrow \infty} \E\left[ \left| n^{-\half} R_n^{(n)}(f)
                  - \E n^{-\half} R_n^{(n)}(f) \right|^2 \right] \\
\leq & \ \limsup_{n \rightarrow \infty}
       (n-1)
       \Cov\left(
      \left( \varphi(X_{1}) - \varphi(X_{N_n(1)}) \right)^2 ,
      \left( \varphi(X_{2}) - \varphi(X_{N_n(2)}) \right)^2
\right) \ .
\end{align*}
The rest of the proof is focused on proving that, while staying bounded, we have
\begin{align}
   \label{eq:remainder_to_prove}
   \limsup_{n \rightarrow \infty}
       (n-1)
       \Cov\left(
      \left( \varphi(X_{1}) - \varphi(X_{N_n(1)}) \right)^2 ,
      \left( \varphi(X_{2}) - \varphi(X_{N_n(2)}) \right)^2
      \right)
    & \leq \ 0 \ .
\end{align}

\medskip

{\bf Reductions:} Without loss of generality, we can assume that the $X_i$'s are uniform on $[0,1]$, by replacing $\varphi$ by $\varphi \circ F_X^{\langle -1 \rangle}$. Furthermore, it is convenient to identify the unit interval $[0,1]$ to the circle of unit length $\R / \Z$. Given two points $(x,y) \in (\R / \Z)^2$, the arc $\wideparen{xy}$ goes from $x$ to $y$ counter-clockwise. Its length is also denoted $\wideparen{xy}$.

\medskip

{\bf Marginals:} $(X_{1}, X_{N_n(1)} )$ and $(X_{2}, X_{N_n(2)} )$ have the same known distribution. $X_1$ is uniform and $X_{N_n(1)}-X_1$ is a independent $Beta(1, n-1)$. Indeed, since the $X_i$'s are independent and conditionally to $X_1=x_1$, we have
\begin{align*}
    \P\left( X_{N_n(1)}-X_1 \geq y \ | \ X_1 = x_1 \right)
= & \ \P\left( \forall i=2, \dots, n, \ X_i \notin \wideparen{x_1 (x_1+y)} \right)\\
= & \ \P\left( X_2 \notin \wideparen{x_1 (x_1+y)} \right)^{n-1}\\
= & \ \left( 1 - y \right)^{n-1} \ .
\end{align*}

\medskip

{\bf Joint distribution :}
Let us now explicit the distribution of $\left( X_1, X_{N_n(1)}, X_2, X_{N_n(2)} \right)$. The two points $X_1$ and $X_2$ are independent and uniform. Conditionally on $X_1$, $X_2$ being fixed, the remaining $n-2$ points can fall either in the arc $\wideparen{X_1 X_2}$ or in the arc $\wideparen{X_2 X_1}$. The number of points in each arc are written $k_{12}+k_{21}=n-2$. The random variable $k_{12}$ is a binomial random variable $Bin(n-2, p=\wideparen{X_1 X_2})$.
Finally, given $k_{12}$, $k_{21}$, we have the following cases
\begin{itemize}
    \item If $(k_{12}, k_{21}) = (0, n-2)$, then $X_{N_n(1)} = X_2$ and $(X_{N_n(2)}-X_2)$ is $\wideparen{X_2 X_1} Beta(1, k_{21}=n-2)$. 
    \item If $(k_{12}, k_{21}) = (n-2, 0)$, then $X_{N_n(2)} = X_1$ and $(X_{N_n(1)}-X_1)$ is $\wideparen{X_1 X_2} Beta(1, k_{12}=n-2)$. This is the symmetric case.
    \item If $0<k_{12}<n-2$, then $(X_{N_n(1)}-X_1)$ is $\wideparen{X_1 X_2} Beta(1, k_{12})$ and $(X_{N_n(2)}-X_2)$ is $\wideparen{X_2 X_1} Beta(1, k_{21})$. The other cases can be included in this case with the natural convention that $Beta(1, k=0) = 1$. We adopt this convention in what follows.
\end{itemize}
Notice that in the first two cases, the law of $\left( X_1, X_{N_n(1)}, X_2, X_{N_n(2)} \right)$ is supported on a three dimensional subspace.

Hence 
\begin{align*}
  & \E \left[ f_1(X_{1}, X_{N_n(1)}) f_2(X_{2}, X_{N_n(2)} ) \right] \\
= & \iint_{(\R/\Z)^2} dx_1 dx_2
    \E\left[ f_1(x_1, X_{N_n(1)}-X_1+x_1) f_2(X_{2}, X_{N_n(2)}-X_2+x_2 ) \ | \ X_1=x_1, X_2=x_2 \right]\\
= & \iint_{(\R/\Z)^2} dx_1 dx_2 \ 
    \E_{k \sim Bin(n-2, \wideparen{x_1 x_2})} \ 
    \E\left[ f_1( x_1, x_1 + \wideparen{x_1 x_2} \beta_{k} )     \right] 
    \E\left[ f_2( x_2, x_2 + \wideparen{x_2 x_1} \beta_{n-2-k} ) \right] \ .
\end{align*}

\medskip

{\bf Explicit densities:} Let us make the densities explicit while distinguishing the singular part and the absolutely continuous part. On the one hand, we have
\begin{align*}
  & \E \left[ 
    \1_{\left\{ X_{1}=X_{N_n(2)} \textrm{ or } X_{2}=X_{N_n(1)} \right\}}
    f_1(X_{1}, X_{N_n(1)}) f_2(X_{2}, X_{N_n(2)} ) \right] \\
= & \iint_{(\R/\Z)^2} dx_1 dx_2 \ 
    \wideparen{x_2 x_1}^{n-2}
    f_1( x_1, x_1 + \wideparen{x_1 x_2} )
    \E\left[ f_2( x_2, x_2 + \wideparen{x_2 x_1} \beta_{n-2} ) \right] \\
  & + \iint_{(\R/\Z)^2} dx_1 dx_2 \ 
    \wideparen{x_1 x_2}^{n-2}
    \E\left[ f_1( x_1, x_1 + \wideparen{x_1 x_2} \beta_{n-2} )     \right] 
    f_2( x_2, x_2 + \wideparen{x_2 x_1} ) \\
= & \frac{1}{(n-1)} \iint_{(\R/\Z)^2} dx_1 dx_2 \ 
    (n-1) \wideparen{x_2 x_1}^{n-2}
    f_1( x_1, x_2 )
    \E\left[ f_2( x_2, x_2 + \wideparen{x_2 x_1} \beta_{n-2} ) \right] \\
  & + \frac{1}{(n-1)} \iint_{(\R/\Z)^2} dx_1 dx_2 \ 
    (n-1) \wideparen{x_1 x_2}^{n-2}
    \E\left[ f_1( x_1, x_1 + \wideparen{x_1 x_2} \beta_{n-2} )     \right] 
    f_2( x_2, x_1 ) \ .
\end{align*}

On the other hand, we have
\begin{align*}
  & \E \left[ 
    \1_{\left\{ X_{1} \neq X_{N_n(2)} \textrm{ and } X_{2} \neq X_{N_n(1)} \right\}}
    f_1(X_{1}, X_{N_n(1)}) f_2(X_{2}, X_{N_n(2)} ) \right] \\
= & \iint_{(\R/\Z)^2} dx_1 dx_2 \ 
    \E_{k \sim Bin(n-2, \wideparen{x_1 x_2})} \ \1_{ 0 < k < n-2} \ 
    \E\left[ f_1( x_1, x_1 + \wideparen{x_1 x_2} \beta_{k} )     \right] 
    \E\left[ f_2( x_2, x_2 + \wideparen{x_2 x_1} \beta_{n-2-k} ) \right] \\
= & \iint_{(\R/\Z)^2} dx_1 dx_2 \ \iint_{[0,1]^2} dy_1 dy_2 \  
    \sum_{k=1}^{n-3} \binom{n-2}{k} \wideparen{x_1 x_2}^{k} \wideparen{x_2 x_1}^{n-2-k} \ \\
  & \quad \times 
    k \left( 1-y_1 \right)^{k-1} f_1( x_1, x_1 + \wideparen{x_1 x_2} y_1 ) \ 
    (n-2-k) \left( 1-y_2 \right)^{n-2-k-1} f_2( x_2, x_2 + \wideparen{x_2 x_1} y_2 ) \\
= & \iint_{(\R/\Z)^2} dx_1 dx_2 \ \iint_{[0,1]^2} dy_1 dy_2 \  
    f_1( x_1, x_1 + \wideparen{x_1 x_2} y_1 ) \ 
    f_2( x_2, x_2 + \wideparen{x_2 x_1} y_2 ) \\
  & \quad \times 
    \frac{\partial}{\partial y_1} \frac{\partial}{\partial y_2}
    \left( 
    \sum_{k=1}^{n-3} \binom{n-2}{k}
    \left( \wideparen{x_1 x_2}(1-y_1) \right)^{k} \ 
    \left( \wideparen{x_2 x_1}(1-y_2) \right)^{n-2-k}
    \right)
    \\
= & \iint_{(\R/\Z)^2} dx_1 dx_2 \ \iint_{[0,1]^2} dy_1 dy_2 \  
    f_1( x_1, x_1 + \wideparen{x_1 x_2} y_1 ) \ 
    f_2( x_2, x_2 + \wideparen{x_2 x_1} y_2 ) \ 
    \frac{\partial}{\partial y_1} \frac{\partial}{\partial y_2}
    \left( 1 - \wideparen{x_1 x_2} y_1
             - \wideparen{x_2 x_1} y_2
    \right)^{n-2} \\
= & \iint_{(\R/\Z)^2} dx_1 dx_2 \ \iint_{[0,\wideparen{x_1 x_2}] \times [0,\wideparen{x_2 x_1}]} dy_1 dy_2 \  
    f_1( x_1, x_1 + y_1 ) \ 
    f_2( x_2, x_2 + y_2 ) \ 
    \frac{\partial}{\partial y_1} \frac{\partial}{\partial y_2}
    \left( 1 - y_1 - y_2 \right)^{n-2} \ .
\end{align*}

\medskip

{\bf A non-trivial identity:} From the computations of the marginal law, we have the non-trivial and interesting identity
\begin{align*}
  & \E \left[ h(X_{1}, X_{N_n(1)})  \right] \\
= & \iint dx_1 dx_2 \ 
    \E_{k \sim Bin(n-2, \wideparen{x_1 x_2})} \ 
    \E\left[ h( x_1, x_1 + \wideparen{x_1 x_2} \beta_{k} )     \right]  \\
= & \int dx_1 \
    \E\left[ h( x_1, x_1 + \beta_{n-1} )     \right] \ .
\end{align*}
The following computation serves as a reality check.
\begin{align*}
  & \E \left[ h(X_{1}, X_{N_n(1)})  \right] \\
= & \ \iint dx_1 dx_2 \ 
    \E_{k \sim Bin(n-2, \wideparen{x_1 x_2})} \ 
    \E\left[ h( x_1, x_1 + \wideparen{x_1 x_2} \beta_{k} )     \right]  \\
= & \ \iint dx_1 dx_2 \ 
    \wideparen{x_1 x_2}^{n-2}
    h( x_1, x_1 + \wideparen{x_1 x_2} )\\
  & + \iint dx_1 dx_2 \ 
    \int_{ [0,\wideparen{x_1 x_2}] } dy_1 \frac{\partial}{\partial y_1}\left( 1-y_1 \right)^{n-2}
    h( x_1, x_1 + y_1 )\\
= & \ \frac{1}{n-1} \E\left[ h( x_1, x_1 + \beta_{n-1} ) \right] \\
  & + \int dx_1 \int_{ [0,1] } dy_1 (1-y_1) \ 
     \frac{\partial}{\partial y_1}\left( 1-y_1 \right)^{n-2}
    h( x_1, x_1 + y_1 )\\
= & \ \frac{1}{n-1} \E\left[ h( X_1, X_1 + \beta_{n-1} ) \right] 
  + \frac{n-2}{n-1} \E\left[ h( X_1, X_1 + \beta_{n-1} ) \right] \\
= & \ \E\left[ h( X_1, X_1 + \beta_{n-1} ) \right] \ .
\end{align*}

\medskip

{\bf Computing the singular part:} Taking $f_1=f_2=h$, we have by symmetry
\begin{align*}
  & \left|
    \E \left[ 
    \1_{\left\{ X_{1}=X_{N_n(2)} \textrm{ or } X_{2}=X_{N_n(1)} \right\}}
    h(X_{1}, X_{N_n(1)}) h(X_{2}, X_{N_n(2)} ) \right] 
    \right| \\
= & \Big| \frac{2}{(n-1)} \iint_{(\R/\Z)^2} dx_1 dx_2 \ 
    (n-1) \wideparen{x_2 x_1}^{n-2}
    h( x_1, x_2 )
    \E\left[ h( x_2, x_2 + \wideparen{x_2 x_1} \beta_{n-2} ) \right] 
    \Big| \\
\leq & \|h \|_\infty \Big| \frac{2}{(n-1)} \iint_{(\R/\Z)^2} dx_1 dx_2 \ 
    (n-1) \wideparen{x_2 x_1}^{n-2}
    h( x_1, x_2 )
    \Big| \\
= & \|h \|_\infty \Big| \frac{2}{(n-1)} \E h(X_1, X_{N_n(1)})
    \Big| \ .
\end{align*}
Hence, using Eq. \eqref{eq:CV_h}, we find 
$$
    \lim_{n \rightarrow \infty}
    (n-1) \left|
    \E \left[ 
    \1_{\left\{ X_{1}=X_{N_n(2)} \textrm{ or } X_{2}=X_{N_n(1)} \right\}}
    h(X_{1}, X_{N_n(1)}) h(X_{2}, X_{N_n(2)} ) \right] 
    \right|
    = 0 \ .
$$

\medskip

{\bf Computing the continuous part:}
\begin{align*}
  & \E \left[ 
    \1_{\left\{ X_{1} \neq X_{N_n(1)} \textrm{ and } X_{2} \neq X_{N_n(2)} \right\}}
    h(X_{1}, X_{N_n(1)}) h(X_{2}, X_{N_n(2)} ) \right] 
    -
    \E\left[ h(X_{1}, X_{N_n(1)}) \right]
    \E\left[ h(X_{2}, X_{N_n(2)} ) \right] 
    \\
= & \iint_{(\R/\Z)^2} dx_1 dx_2 \ \iint_{[0,\wideparen{x_1 x_2}] \times [0,\wideparen{x_2 x_1}]} dy_1 dy_2 \  
    h( x_1, x_1 + y_1 ) \ 
    h( x_2, x_2 + y_2 ) \ 
    \frac{\partial}{\partial y_1} \frac{\partial}{\partial y_2}
    \left( 1 - y_1 - y_2 \right)^{n-2} \\
   & \ -
    \iint_{(\R/\Z)^2} dx_1 dx_2 \
    \E\left[ h(x_1, x_1 + \beta_{n-1}) \right]
    \E\left[ h(x_2, x_2 + \beta_{n-1} ) \right] \\
= & \iint_{(\R/\Z)^2} dx_1 dx_2 \ \iint_{\R_+^2} dy_1 dy_2 \  
    h( x_1, x_1 + y_1 ) \ 
    h( x_2, x_2 + y_2 ) \\
  & \times
    \left( \1_{ y_1 \leq \wideparen{x_1 x_2}, y_2 \leq \wideparen{x_2 x_1}}
    \frac{\partial}{\partial y_1} \frac{\partial}{\partial y_2}
    \left( 1 - y_1 - y_2 \right)^{n-2} 
    - \1_{ y_1 \leq 1, y_2 \leq 1}
    (n-1)^2
    \left( 1 - y_1 \right)^{n-2} 
    \left( 1 - y_2 \right)^{n-2} 
    \right) \ .
\end{align*}

Now, consider the basic inequality
\begin{align*}
  & \1_{ y_1 \leq \wideparen{x_1 x_2}, y_2 \leq \wideparen{x_2 x_1}}
    \frac{\partial}{\partial y_1} \frac{\partial}{\partial y_2}
    \left( 1 - y_1 - y_2 \right)^{n-2} \\
= & \ \1_{ y_1 \leq \wideparen{x_1 x_2}, y_2 \leq \wideparen{x_2 x_1}}
    (n-2) (n-3)
    \left( 1 - y_1 - y_2 \right)^{n-4} \\
\leq & \ \1_{ y_1 \leq 1, y_2 \leq 1}
    (n-2) (n-3)
    \left( 1 - y_1 - y_2 + y_1 y_2 \right)^{n-4} \\
\leq & \ \1_{ y_1 \leq 1, y_2 \leq 1}
    (n-1)^2
    \left( (1 - y_1)(1 - y_2) \right)^{n-4} \ .
\end{align*}

Continuing the previous computation
\begin{align*}
  & \E \left[ 
    \1_{\left\{ X_{1} \neq X_{N_n(1)} \textrm{ and } X_{2} \neq X_{N_n(2)} \right\}}
    h(X_{1}, X_{N_n(1)}) h(X_{2}, X_{N_n(2)} ) \right] 
    -
    \E\left[ h(X_{1}, X_{N_n(1)}) \right]
    \E\left[ h(X_{2}, X_{N_n(2)} ) \right] 
    \\
\leq & (n-1)^2 \iint_{(\R/\Z)^2} dx_1 dx_2 \ \iint_{[0,1]^2} dy_1 dy_2 \  
    h( x_1, x_1 + y_1 ) \ 
    h( x_2, x_2 + y_2 ) \\
  & \times
    \left( 
    \left( 1 - y_1 \right)^{n-4} 
    \left( 1 - y_2 \right)^{n-4} 
    - 
    \left( 1 - y_1 \right)^{n-2} 
    \left( 1 - y_2 \right)^{n-2} 
    \right) \\
\leq & \left( 1 + o(1) \right) (n-3)^2 \iint_{(\R/\Z)^2} dx_1 dx_2 \ \iint_{[0,1]^2} dy_1 dy_2 \  
    h( x_1, x_1 + y_1 ) \ 
    h( x_2, x_2 + y_2 ) \left( 1 - y_1 \right)^{n-4} 
    \left( 1 - y_2 \right)^{n-4} \\
  & \times
    \left( 
    1
    - 
    \left( 1 - y_1 \right)^{2} 
    \left( 1 - y_2 \right)^{2} 
    \right) \\
= & \left( 1 + o(1) \right) \E\left[ 
    h( X_1, X_1 + Y_1 ) \ 
    h( X_2, X_2 + Y_2 ) \
    \left( 
    1
    - 
    \left( 1 - Y_1 \right)^{2} 
    \left( 1 - Y_2 \right)^{2} 
    \right) 
    \right] \ ,    
\end{align*}
where $X_1, X_2, Y_1, Y_2$ are independent and $Y_i \eqlaw Beta(1,n-3)$. Because
\begin{align*}
    1
    - 
    \left( 1 - Y_1 \right)^{2} 
    \left( 1 - Y_2 \right)^{2} 
   = & \ 
    \left[  1
    - 
    \left( 1 - Y_1 \right) 
    \left( 1 - Y_2 \right)
    \right]
    \left[ 
    1
    + 
    \left( 1 - Y_1 \right) 
    \left( 1 - Y_2 \right)
    \right]\\
    = & \ 
    \left(  Y_1 + Y_2 - Y_1 Y_2
    \right)
    \left[ 
    1
    + 
    \left( 1 - Y_1 \right) 
    \left( 1 - Y_2 \right)
    \right]\\
    \leq & \ 
    2 \left(  Y_1 + Y_2 \right) \ .
\end{align*}

Continuing the previous computation further and invoking symmetry between the pairs $(X_i, Y_i)$, $i=1,2$, we have
\begin{align*}
  & \E \left[ 
    \1_{\left\{ X_{1} \neq X_{N_n(1)} \textrm{ and } X_{2} \neq X_{N_n(2)} \right\}}
    h(X_{1}, X_{N_n(1)}) h(X_{2}, X_{N_n(2)} ) \right] 
    -
    \E\left[ h(X_{1}, X_{N_n(1)}) \right]
    \E\left[ h(X_{2}, X_{N_n(2)} ) \right] 
    \\
\leq & \left( 1 + o(1) \right) 2 \E\left[ 
    h( X_1, X_1 + Y_1 ) \ 
    h( X_2, X_2 + Y_2 ) \
    \left( Y_1 + Y_2 \right) 
    \right] \\
= & \left( 1 + o(1) \right) 4 
    \E\left[ h( X_1, X_1 + Y_1 ) \ Y_1 \right] \ 
    \E\left[ h( X_2, X_2 + Y_2 ) \right] \\
\leq & \left( 1 + o(1) \right) 4 \|h\|_\infty
    \E\left[ \beta_{n-3} \right] \ 
    \E\left[ h( X_1, X_{N_{n-3}(1)} ) \right] \\
\leq & \left( 1 + o(1) \right) \frac{4 \|h\|_\infty}{n-2}
    \E\left[ h( X_1, X_{N_{n-3}(1)} ) \right] \ .
\end{align*}
Upon invoking Eq. \eqref{eq:CV_h}, we see that Eq. \eqref{eq:remainder_to_prove} holds and we are done.

\section{Proof of multivariate Sobol' fluctuations (Main Theorem \ref{thm:limit_sobol_multi})}
\label{section:proof_multivariate}

Here we only point out the changes, since most proofs carry over verbatim.
 We shall deal with the vector process $\widehat{\theta}_n \in V := \R^d \otimes \R^3$ of size $3 d$
\begin{align}
\label{eq:def_sobol_multi_theta}
\widehat{\theta}_n & := 
\begin{pmatrix}
\widehat{\theta}_n(f_1)\\
\widehat{\theta}_n(f_2)\\
\vdots\\
\widehat{\theta}_n(f_d)\\
\end{pmatrix}
\end{align}
where for each $\ell=1,\ldots,d$, $\widehat{\theta}_n(f_\ell) \in \R^3$ is given by
\begin{align}
  & \widehat{\theta}_n(f_\ell) 
 =  \frac{1}{n}\begin{pmatrix}
        Z_n(f_\ell) \\
        S_n(f_\ell) \\
        S_n(f_\ell^2)
        \end{pmatrix} \ \nonumber \\
& = \frac{1}{n}
    \begin{pmatrix}
    R_n^{(n)}(f_\ell) \\
    0 \\
    0
    \end{pmatrix}
    +
    \frac{1}{n}
    \begin{pmatrix}
    \sum_{i=1}^n \E_\varepsilon\big(f_\ell(X_{\sigma_n(i)}, \varepsilon)\big)^2 \\
    \sum_{i=1}^n \E_\varepsilon\big(f_\ell(X_{\sigma_n(i)}, \varepsilon)\big) \\
    \sum_{i=1}^n \E_\varepsilon\big(f_\ell^2(X_{\sigma_n(i)}, \varepsilon)\big)
    \end{pmatrix}
    +
    \frac{1}{n}
    \begin{pmatrix}
    M_n^{(n)}(f_\ell) \\
    \sum_{i=1}^n \left( f_\ell(X_{\sigma_n(i)}, \varepsilon_i) - \E_\varepsilon\big(f_\ell(X_{\sigma_n(i)}, \varepsilon)\big) \right) \\
    \sum_{i=1}^n \left( f_\ell^2(X_{\sigma_n(i)}, \varepsilon_i) - \E_\varepsilon\big(f_\ell^2(X_{\sigma_n(i)}, \varepsilon)\big) \right)
    \end{pmatrix} \  \nonumber\\
& = \frac{1}{n}
    \begin{pmatrix}
    R_n^{(n)}(f_\ell) \\
    0 \\
    0
    \end{pmatrix}
    +
    \frac{1}{n}
    \underbrace{
    \begin{pmatrix}
    \sum_{i=1}^n \E_\varepsilon\big(f_\ell(X_i, \varepsilon)\big)^2 \\
    \sum_{i=1}^n \E_\varepsilon\big(f_\ell(X_i, \varepsilon)\big) \\
    \sum_{i=1}^n \E_\varepsilon\big(f_\ell^2(X_i, \varepsilon)\big)
    \end{pmatrix}}_{\text{(1)}}
    +
    \frac{1}{n}
    \underbrace{
    \begin{pmatrix}
    M_n^{(n)}(f_\ell) \\
    \sum_{i=1}^n \left( f_\ell(X_{\sigma_n(i)}, \varepsilon_i) - \E_\varepsilon\big(f_\ell(X_{\sigma_n(i)}, \varepsilon)\big) \right) \\
    \sum_{i=1}^n \left( f_\ell^2(X_{\sigma_n(i)}, \varepsilon_i) - \E_\varepsilon\big(f_\ell^2(X_{\sigma_n(i)}, \varepsilon)\big) \right)
    \end{pmatrix}}_{\text{(2)}} \ , 
\end{align}
The remainder term $R_n(f_\ell)$ and martingale array $M_n^{(n)}(f_\ell)$ are obtained exactly as before, after applying the decomposition of Proposition \ref{proposition:double_doob}.
Here, each martingale array $M_n^{(n)}(f_\ell)$ is defined as follows for each fixed $\phi \in \Phi$. 
The basic martingale increments are defined for all \( i \in \{1, \dots, n\} \) by
\begin{align}
    \Delta \Mk_i^{(n)}(f_\ell) := f_\ell(X_{\sigma_n(i)}, \varepsilon_i) - \E_\varepsilon[ f_\ell(X_{\sigma_n(i)}, \varepsilon)] \ .
    \label{eq:martingale_increment_phi}
\end{align}
Then, we have for $j\geq 1$
\begin{align}
M_j^{(n)}(f_\ell) 
     & :=  \sum_{k=1}^{j} \left( f_\ell(X_{\sigma_n(k-1)}, \varepsilon_{k-1}) + \E_\varepsilon[ f_\ell(X_{\sigma_n(k+1)}, \varepsilon)] \right) \cdot \Delta \Mk_k^{(n)}(f_\ell) \ .
\label{eq:main_martingale_part_phi}
\end{align}
The remainder on the other hand is
\begin{align}
   R_j^{(n)}(f_\ell) 
:=  \Oc( \|f_\ell\|_\infty^2 )
-
\half
\sum_{i=1}^j \left( \E_\varepsilon[f_\ell(X_{\sigma_n(i-1)}, \varepsilon)] -
\E_\varepsilon[f_\ell(X_{\sigma_n(i)}, \varepsilon)] 
\right)^2
\ ,
\label{eq:remainder_phi}
\end{align}

The procedure is similar to the previous section, we just have to make explicit the covariance matrices which are in $V \otimes V \approx \R^{3d \times 3d}$. We denote these matrices by $\Sigma_{0}$ and $\Sigma_{1}$ in $V \otimes V$. 
The matrix $\Sigma_{0}$ has block form $\Sigma_{0} = \left( \Sigma_{0}( f_i,  f_j) \right)_{1 \leq i,j \leq d}$, where we recall the definition 
\begin{align}
\label{eq:cov_sigma0_multivariate}
   \Sigma_{0}( f_\ell, f_p)
:= & \ \Cov\left( 
\begin{pmatrix}
\E_\varepsilon[f_\ell(X,\varepsilon)]^2\\
\E_\varepsilon[f_\ell(X,\varepsilon)]\\
\E_\varepsilon[f_\ell^2(X,\varepsilon)]
\end{pmatrix} ,
\begin{pmatrix}
\E_\varepsilon[f_p(X,\varepsilon)]^2\\
\E_\varepsilon[f_p(X,\varepsilon)]\\
\E_\varepsilon[f_p^2(X,\varepsilon)]
\end{pmatrix}
\right) \ .
\end{align}
Notice that specializing to $d=1$ and $f_\ell = f_p = f$ recovers the expression in Eq. \eqref{eq:cov_sigma0_fg} and justifies the choice of notation.
For $\Sigma_{1}=\left( \Sigma_{1}( f_i, f_j) \right)_{1 \leq i,j \leq d},$ we recall the definitions 
$$
\Sigma_a(X,f_\ell,f_p) := \begin{pmatrix}
      2 \E_\varepsilon[ f_\ell(X, \varepsilon) ] \\
      1 \\
      1
      \end{pmatrix}
      \begin{pmatrix}
      2 \E_\varepsilon[ f_p(X, \varepsilon) ] \\
      1 \\
      1
      \end{pmatrix}^T
    + \Cov_\varepsilon[ f_\ell(X, \varepsilon), f_p(X, \varepsilon) ]
     \begin{pmatrix}
      1  \\
      0 \\
      0
      \end{pmatrix}
      \begin{pmatrix}
      1  \\
      0\\
      0
      \end{pmatrix}^T 
$$
and
$$\Sigma_b(X, f_\ell, f_p)=\Cov_\varepsilon\left( 
\begin{pmatrix}
f_\ell(X, \varepsilon)\\
f_\ell(X, \varepsilon)\\
f_\ell^2(X, \varepsilon)
\end{pmatrix} ,
\begin{pmatrix}
f_p(X, \varepsilon)\\
f_p(X, \varepsilon)\\
f_p^2(X, \varepsilon)
\end{pmatrix}
\right) \ .
$$
We finally define
$$\Sigma_1(f_\ell,f_p):=\E[\Sigma_a(X, f_\ell, f_p) \odot \Sigma_b(X, f_\ell, f_p)] \ .$$

\section{Proof of fluctuations for Cram\'er--von Mises (Main Theorem \ref{thm:limit_cvm})}
\label{section:proof_cvm}

\subsection{Setup}

{\bf Basic idea.} As defined previously in Eq. \eqref{eq:Tn_CVM}, recall
\begin{align*}
T_n^{\mathrm{CvM}} = \int_\mathbb{R} T_n(t) \, dF_n(t) 
= \frac{1}{n} \sum_{i=1}^{n} \int_\mathbb{R} \Upsilon_{\sigma_n(i)}(t) \cdot \Upsilon_{\sigma_n(i+1)}(t) \, dF_n(t) \ .
\end{align*}

Everything hinges on the following decomposition
\begin{align}
   & T_n^{\mathrm{CvM}} - T^{\mathrm{CvM}}
 = \int_\R T_n(t) dF_n(t) - \int_\R T(t) d F_Y(t) \nonumber \\
 = & \int_\R \left[ T_n(t) - T(t) \right] d F_Y(t)
   - \int_\R \left[ F_n(t) - F_Y(t) \right] dT(t)
   + \int_\R \left[ T_n(t) - T(t) \right] d( F_n(t) - F_Y(t) ) \ .
\label{eq:Tn_CVM_decomposition}
\end{align}
Indeed
\begin{align*}
   & \int_\R T_n(t) dF_n(t) - \int_\R T(t) d F_Y(t)  \\
 = & \int_\R T(t) dF_n(t)
   + \int_\R \left[ T_n(t) - T(t) \right] dF_n(t)
   - \int_\R T(t) d F_Y(t)  \\
 = & \int_\R T(t) d( F_n(t) - F_Y(t) )
   + \int_\R \left[ T_n(t) - T(t) \right] dF_n(t)\\
 = & \int_\R T(t) d( F_n(t) - F_Y(t) )
   + \int_\R \left[ T_n(t) - T(t) \right] d F_Y(t)
   + \int_\R \left[ T_n(t) - T(t) \right] d( F_n(t) - F_Y(t) )\\
 = & - \int_\R \left[ F_n(t) - F_Y(t) \right] dT(t)
   + \int_\R \left[ T_n(t) - T(t) \right] d F_Y(t)
   + \int_\R \left[ T_n(t) - T(t) \right] d( F_n(t) - F_Y(t) ) \ ,
\end{align*}
where on the last step, we have performed an integration by parts. Notice that this expression can also be understood as a Taylor expansion of the functional $(F,T) \mapsto \int T \, dF $. Basically, we are applying a functional delta method by hand, and controlling the error terms. 

The basic idea is to prove functional fluctuation theorems for $\left( \sqrt{n}\begin{pmatrix} T_n(t) - T(t) \\ F_n(t) - F_Y(t) \end{pmatrix} \ ; \ t \in \R \right)$ and applying the continuous mapping theorem to Eq. \eqref{eq:Tn_CVM_decomposition}, while controlling various error terms. The fluctuations in law are only obtained after an identification in law.

We consider such a pedestrian approach preferable because the functional analytic setting is not directly applicable in our setup. Indeed, while $F_n$ has limiting fluctuations in the Skorohod space $D$, its differential $dF_n$ does not converge in law in a topology amenable to directly applying the delta method to $(F,T) \mapsto \int T \, dF $.

\medskip

{\bf Key identification in law.} As done in Eq.~\eqref{eq:vector_eqlaw}, we have the following identity in law between processes
\begin{align}
\label{eq:main_equality_in_law}
\left( 
\begin{pmatrix}
  T_n(t) \\ F_n(t)
\end{pmatrix}
\ ; \ t \in \R \right)
& \stackrel{\Lc}{=}
\left( 
\begin{pmatrix}
  Z_n(t) + \Oc( \frac1n ) \\ \widetilde{F}_n(t)
\end{pmatrix}
\ ; \ t \in \R \right) \ ,
\end{align}
where
$$
  Z_n(t) := \frac{1}{n} \sum_{i=1}^{n} 
  \mathds{1}_{ \{ f(X_{\sigma_n(i-1)}, \varepsilon_{i-1}) \leq t \} }
  \mathds{1}_{ \{ f(X_{\sigma_n(i)}, \varepsilon_{i}) \leq t \} }
  \ ,
$$
$$
  \widetilde{F}_n(t) := \frac{1}{n} \sum_{i=1}^{n} 
  \mathds{1}_{ \{ f(X_{\sigma_n(i)}, \varepsilon_{i}) \leq t \} }
  \ .
$$
In the above expression, we adopt the convention that $\sigma_n$ is cyclic and $\varepsilon_0=0$, as usual. At this stage, note that unlike in the case of Sobol indices, we do not deal with an (analogue of) empirical variance. Indeed in the Cramér–von Mises framework considered here, the denominator in Eq. \eqref{eq:cvm_simplified} is constant and does not need to be estimated.

\medskip

{\bf Decompositions of the processes $(Z_n, \widetilde{F}_n)$.} Recall from Eq.~\eqref{eq:Filtration} the filtration  $\mathbb{F} = \left( \Fc_i \ ; \ i \in \N \right)$ defined by
\begin{align*}
    \mathcal{F}_i = 
    \sigma \left(X_k \ ; \ k \in \mathbb{N}^*\right)
    \vee \sigma\left( \varepsilon_1, \dots, \varepsilon_i \right), \quad \text{for } i \geq 0 \ .
\end{align*}
We apply the same ideas as in the univariate Sobol' case of Subsection \ref{subsection:univariate_sobol}. In that fashion, we have the following decompositions of processes
\begin{align}
n Z_n(t) & = A_n^{(n)}(t) + M^{(n)}_n(t) \ , \\
n \widetilde{F}_n(t) & = \sum_{i=1}^{n} \E_\varepsilon\left[ \mathds{1}_{\{ f(X_{\sigma_n(i)}, \varepsilon)\leq t \}}\right]
 +  \sum_{i=1}^{n} \left[ \mathds{1}_{\{ f(X_{\sigma_n(i)}, \varepsilon)\leq t \}} - \E_\varepsilon\left[ \mathds{1}_{\{ f(X_{\sigma_n(i)}, \varepsilon)\leq t \}}\right]
 \right]\ ,
\end{align}
where $A_n^{(n)}$ is an $\Fc_0$-measurable process,  $M_n^{(n)}$ is an $\F$-martingale taking values in the Skorohod space $D$. We insist that the processes in the variable $t$ are not themselves martingales. We are rather in the setup of Banach-valued discrete-time martingales and we shall refrain from invoking generic tools such as \cite{pisier2016martingales}.

Pushing further, after making these terms more explicit following the scalar case of Proposition \ref{proposition:double_doob}, we have
\begin{align}
  \label{eq:def_theta_nt}
  \widehat{\theta}_n(t) & :=
  \begin{pmatrix}
    Z_n(t) + \Oc(\frac{1}{n})\\
    \widetilde{F}_n(t)
  \end{pmatrix}\\
  = & \
  \frac{1}{n}
  \begin{pmatrix}
    R_n^{(n)}(t) \\ 0
  \end{pmatrix}
  +
  \frac{1}{n}
  \begin{pmatrix}
    \sum_{i=1}^{n} \E_\varepsilon\left[ \mathds{1}_{\{ f(X_{\sigma_n(i)}, \varepsilon)\leq t \}}\right]^2 \\
    \sum_{i=1}^{n} \E_\varepsilon\left[ \mathds{1}_{\{ f(X_{\sigma_n(i)}, \varepsilon)\leq t \}}\right]
  \end{pmatrix}
  +
  \frac{1}{n}
  \begin{pmatrix}
    M_n^{(n)}(t) \\
    \sum_{i=1}^{n} \left[ \mathds{1}_{\{ f(X_{\sigma_n(i)}, \varepsilon)\leq t \}} - \E_\varepsilon\left[ \mathds{1}_{\{ f(X_{\sigma_n(i)}, \varepsilon)\leq t \}}\right]
    \right]
  \end{pmatrix} \nonumber \\
  = & \
  \frac{1}{n}
  \begin{pmatrix}
    R_n^{(n)}(t) \\ 0
  \end{pmatrix}
  +
  \frac{1}{n}
  \underbrace{
  \begin{pmatrix}
    \sum_{i=1}^{n} \varphi(t,X_i)^2 \\
    \sum_{i=1}^{n} \varphi(t,X_i)
  \end{pmatrix}}_{(1)=\Ac_n^{(n)}(t)}
  +
  \frac{1}{n}
  \underbrace{
  \begin{pmatrix}
    M_n^{(n)}(t) \\
    \sum_{i=1}^{n} \left[ \mathds{1}_{\{ f(X_{\sigma_n(i)}, \varepsilon_i)\leq t \}} - \E_\varepsilon\left[ \mathds{1}_{\{ f(X_{\sigma_n(i)}, \varepsilon)\leq t \}}
    \right]
    \right]
  \end{pmatrix}}_{(2)=\Mc_n^{(n)}(t)} \ \nonumber ,
\end{align}
where $M_n^{(n)}$ is the last term in a martingale array $\left( M_j^{(n)} \ ; \ 0 \leq j \leq n \right)$. The martingale array itself is defined as follows.

The basic martingale is 
$$
\Delta \Mk_i^{(n)}(t)
:=
\mathds{1}_{\{ f(X_{\sigma_n(i)}, \varepsilon_i)\leq t \}} 
- 
\E_\varepsilon\left[ \mathds{1}_{\{ f(X_{\sigma_n(i)}, \varepsilon)\leq t \}} \right] \ ,
$$
while our martingale array is
$$
M_j^{(n)}(t) := \sum_{k=1}^{j} \left( \mathds{1}_{\{ f(X_{\sigma_n(k-1)}, \varepsilon_{k-1})\leq t \}}
 + \E_\varepsilon\left[\mathds{1}_{\{ f(X_{\sigma_n(k+1)}, \varepsilon)\leq t \}}\right] \right)
              \Delta \Mk^{(n)}_k(t) 
\ .
$$

The residue term is given in terms of the regression function \eqref{eq:def_phi_tx} as
$$
R_n^{(n)}(t) = \ \Oc(1)
- \half \sum_{i=1}^{n} \left[ 
  \varphi(t, X_{\sigma_n(i-1)})
  -
  \varphi(t, X_{\sigma_n(i)})
  \right]^2
              \ ,
$$
where the $\Oc(1)$ is uniformly bounded in $t$.

\medskip

{\bf Fluctuation process $\left( \vartheta_n(t) \ ; \ t \in \R \right)$.} In particular
\begin{align*}
  \widehat{\theta}_n(t) - \E[\widehat{\theta}_n(t)]
  := & \
  \frac{1}{n}
  \begin{pmatrix}
    R_n^{(n)}(t) - \E[ R_n^{(n)}(t) ] \\ 0
  \end{pmatrix}
  +
  \frac{1}{n}
  \left( 
  \Ac_n^{(n)}(t)
  - 
  \E[\Ac_n^{(n)}(t)]
  \right)
  +
  \frac{1}{n}
  \Mc_n^{(n)}(t) \ ,
\end{align*}
with
$$
\E[\Ac_n^{(n)}(t)] = \begin{pmatrix} T(t) , F_Y(t) \end{pmatrix}^T \ ,
$$
$$
\frac{1}{n}\E[ R_n^{(n)}(t) ] = \ \Oc\left( \frac{1}{n} \right)
- \half \E\left[ 
  \left(
  \varphi(t, X_{1})
  -
  \varphi(t, X_{N_n(1)})
  \right)^2
  \right]\ .
$$

Because the residue $R_n^{(n)}(t) - \E[ R_n^{(n)}(t) ] $ is tricky to handle, let us define the fluctuation process as the $\R^2$-valued process
\begin{align}
  \label{eq:def_cvm_fluctuations}
  \vartheta_n(t) := \ &
  \frac{1}{\sqrt{n}}
  \left( \Ac_n^{(n)}(t) -  \E[\Ac_n^{(n)}(t)] \right)
  +
  \frac{1}{\sqrt{n}}
  \Mc_n^{(n)}(t) \ .
\end{align}

\subsection{Structure of the proof}

{\bf Step 1: Controlling the various error terms.}
Let $o_\P(1)$ denote a quantity converging to zero in probability.

First, we claim that, when considering the last term in Eq. \eqref{eq:Tn_CVM_decomposition}, we have
\begin{align}
\label{eq:Tn_cvm_error}
\sqrt{n} E_n := & \ \sqrt{n} \int_\R \left[ T_n(t) - T(t) \right] d( F_n(t) - F_Y(t) ) = o_\P(1) \ .
\end{align}
The proof is given in Subsection \ref{subsection:proof_cvm_error}. Second, we claim that
\begin{align}
\label{eq:Tn_cvm_error2}
\sqrt{n} E_n' := & \ \int_\R \left[ n^{-\half} R_n^{(n)}(t) - n^{-\half} \E[ R_n^{(n)}(t) ] \right] d F_Y(t) = o_\P(1) \ ,
\end{align}
which can be obtained by a first moment estimate combined with Lemma \ref{lemma:remainder_control}. Indeed, by the Cauchy-Schwarz inequality and then dominated convergence, we have
$$
\E[ |E_n'| ] \leq \int_\R \E \left| n^{-\half} R_n^{(n)}(t) - n^{-\half} \E[ R_n^{(n)}(t) | \right] d F_Y(t)
\leq  \int_\R \sqrt{ \Var\left[ n^{-\half} R_n^{(n)}(\1_{f(x, \varepsilon) \leq t}) \right] } d F_Y(t)
\xrightarrow[n \rightarrow \infty]{} 0 \ .
$$
The domination is satisfied because the $\Var\left[ n^{-\half} R_n^{(n)}(\1_{f(x, \varepsilon) \leq t}) \right]$ remains uniformly bounded in $t$, upon examining the proof of Lemma \ref{lemma:remainder_control}.

Now we start from Eq. \eqref{eq:Tn_CVM_decomposition} then apply Eq. \eqref{eq:Tn_cvm_error} and Eq. \eqref{eq:Tn_cvm_error2}, with the fact that an $o_\P(1)$ remains an $o_\P(1)$ after an equality in law. This yields
\begin{align*}
   & \sqrt{n} \left( T_n^{\mathrm{CvM}} - T^{\mathrm{CvM}} \right) \\
 = & \ \sqrt{n} \int_\R \left[ T_n(t) - T(t) \right] d F_Y(t)
   - \sqrt{n} \int_\R \left[ F_n(t) - F_Y(t) \right] dT(t)
   + \sqrt{n} E_n \\
 \eqlaw & \ \sqrt{n} \int_\R \left[ \langle \widehat{\theta}_n(t), e_1 \rangle - T(t) \right] d F_Y(t)
   - \sqrt{n} \int_\R \left[ \langle \widehat{\theta}_n(t), e_2 \rangle  - F_Y(t) \right] dT(t)
   + o_\P(1) \\
 = & \int_\R \langle \vartheta_n(t), e_1 \rangle  d F_Y(t)
   - \int_\R \langle \vartheta_n(t), e_2 \rangle  dT(t)
   + \int_\R n^{-\half} R_n^{(n)}(t) d F_Y(t)
   + o_\P(1) \\
 \stackrel{Eq. \eqref{eq:Tn_cvm_error2}}{=}
   & \int_\R \langle \vartheta_n(t), e_1 \rangle  d F_Y(t)
   - \int_\R \langle \vartheta_n(t), e_2 \rangle  dT(t)
   + \int_\R n^{-\half} \E[R_n^{(n)}(t)] d F_Y(t)
   + \sqrt{n} E_n'
   + o_\P(1) \\
 \stackrel{Eq. \eqref{eq:def_cvm_Delta}}{=} 
   & \int_\R \langle \vartheta_n(t), e_1 \rangle  d F_Y(t)
   - \int_\R \langle \vartheta_n(t), e_2 \rangle  dT(t)
   - \half \Delta_n
   + o_\P(1) \ .
\end{align*}
As such, by Slutsky's Lemma, it suffices to prove
\begin{align*}
   \int_\R \langle \vartheta_n(t), e_1 \rangle  d F_Y(t)
   - \int_\R \langle \vartheta_n(t), e_2 \rangle  dT(t)
\xrightarrow[\Lc]{n \rightarrow \infty}  & \ \Nc( 0, \frac{1}{36} \sigma^2_{\mathrm{CvM}} ) 
\ .
\end{align*}
The factor $\frac{1}{36}$ is here only to account for the linear transformation Eq. \eqref{eq:cvm_simplified} linking $\rho^{\mathrm{CvM}}$ and $T^{\mathrm{CvM}}$.

\medskip

{\bf Step 2:} We prove a bivariate fluctuation theorem
$$
    \left( \vartheta_n(t) \ ; \ t \geq 0 \right)
    \stackrel{\Lc}{\longrightarrow}
    \left( \Xc, \Yc \right) \ , 
$$
where $(\Xc,\Yc)$ is a Gaussian process. Of course, by virtue of Glivenko-Cantelli and its fluctuations known as the Donsker Theorem for empirical processes, $\Yc$ is nothing but the usual Brownian bridge. 

\smallskip

{\it Step 1.1: Tightness.}
In order to prove that $\left( \vartheta_n \ ; \ n \geq 1 \right)$ is tight,  recall that tightness for a tuple of processes is equivalent to tightness of each process, because a product of compact spaces is compact. The second coordinate has the same distribution as $\sqrt{n}(F_n-F)$, which is tight in the space $D$, by virtue of Donsker's theorem for empirical processes. 

Also, instead of considering each coordinate, we can consider the decomposition of $\vartheta_n$ into martingale and $\Fc_0$-measurable part in Eq. \eqref{eq:def_cvm_fluctuations}. Notice $t \mapsto \frac{1}{\sqrt{n}}
  \left( \Ac_n^{(n)}(t) -  \E[\Ac_n^{(n)}(t)] \right)$ is an empirical process obtained by i.i.d. sums. In this classical setting treated for example in \cite[Chapter 19]{Vaart_1998}, tightness is a given if $\left( \varphi(t,\cdot)^2 \ ; \ t \in \R \right)$ and $\left( \varphi(t,\cdot) \ ; \ t \in \R \right)$ are included in a Donsker class. This is confirmed for monotone functions in \cite[Example 19.11]{Vaart_1998}. 
  
In the end, we only need to prove tightness for $\left( M_n^{(n)} \ ; \ n \in \N \right)$, the first coordinate of the martingale component. This is technical and done in Section \ref{section:chaining}.

\smallskip

{\it Step 1.2: Finite dimensional distributions.}
We fix a family of times $t_1 \leq \ldots \leq t_d$ and we consider the finite dimensional vector 
$$ \left( \vartheta_n(t_1),\ldots, \vartheta_n(t_d) \right)^\top \ .$$
It has dimension $2 d$. We proceed exactly as in the proof of Gaussian fluctuations for the multivariate Sobol' index (Theorem \ref{thm:limit_sobol_multi}), or rather when aiming for the Gaussian fluctuations of \eqref{eq:def_sobol_multi_theta}. Only in this case, we only need a $2d$ vector instead of a $3d$ vector and we pick functions
$$
\phi_i = \mathds{1}_{(-\infty,t_i]} \ , \qquad i=1,\ldots,d \ .
$$
In the end, the covariance matrix of $(\Xc_t,\Yc_t)$ is denoted $C$ given by
$$ C(t,s) = \begin{pmatrix}
C_{\Xc,\Xc}(t,s) &
C_{\Xc,\Yc}(t,s) \\
C_{\Yc,\Xc}(t,s) &
C_{\Yc,\Yc}(t,s)
\end{pmatrix} \ .
$$
The detailed computations leading to the expressions in Theorem \ref{thm:limit_cvm} are in Subsection \ref{subsection:asymptotic_variance}.

\medskip

{\bf Step 3:} Because integrating against (compactly supported) measures is a continuous mapping, we can invoke the mapping theorem. This yields that

$$ \sqrt{n}
   \left( 
     \int_\R \langle \vartheta_n(t), e_1 \rangle  d F_Y(t)
   - \int_\R \langle \vartheta_n(t), e_2 \rangle  dT(t)
   \right)
   \stackrel{\Lc}{\longrightarrow}
     \int_\R \Xc_t d F_Y(t) 
   - \int_\R \Yc_t dT(t)
   \ ,
$$
which is a Gaussian random variable. Its variance is entirely determined from the covariance structure of $(\Xc,\Yc)$. More precisely, we obtain
\begin{align*}
  & \frac{1}{36} \sigma^2_{\mathrm{CvM}} \\
= & \ \Var\left[ \int_\R \Xc_t d F_Y(t) - \int_\R \Yc_t dT(t) \right] \\
= & \  \E\left[\left( \int_\R \Xc_t d F_Y(t) \right)^2 \right]
   +   \E\left[\left( \int_\R \Yc_t d T  (t) \right)^2 \right]
   - 2 \Cov\left(\int_\R \Yc_t dT(t),\int_\R \Xc_t d F_Y(t)\right)\\
= & \ \int_\R \int_\R \Cov( \Xc_t, \Xc_s) dF_Y(t) dF_Y(s)
    + \int_\R \int_\R \Cov( \Yc_t, \Yc_s) dT(t) dT(s)
    -2\int_\R \int_\R \Cov( \Xc_t, \Yc_s) dF_Y(t) dT(s) \\
= & \ \int_\R \int_\R C_{\Xc, \Xc} dF_Y(t) dF_Y(s)
    + \int_\R \int_\R C_{\Yc, \Yc} dT(t) dT(s)
    -2\int_\R \int_\R C_{\Xc, \Yc} dF_Y(t) dT(s) \ .
\end{align*}
This is the announced result.

\subsection{Asymptotic covariance}
\label{subsection:asymptotic_variance}

We denote by $C_{0}$ the contribution of the $\mathcal F_0$-measurable part and by $C_{1}$ the contribution of the martingale part. The expressions of $C_0$ and $C_1$ can be directly extracted from those of $\Sigma_0$ and $\Sigma_1$ by taking
\[
\phi_i=\mathds{1}_{(-\infty,t_i]}, \qquad i=1,\ldots,d .
\]

The matrix $C_{0}$ has block form
\[
C_{0}=\bigl(C_{0}(t_i,t_j)\bigr)_{1\le i,j\le d}.
\]
From expression~\eqref{eq:cov_sigma0_multivariate}, we obtain, for $t,s\in\mathbb R$,
\begin{align*}
C_{0}(t,s)
&:= \Cov\left(
\begin{pmatrix}
\E_\varepsilon[\mathds{1}_{\{f(X,\varepsilon)\le t\}}]^2\\
\E_\varepsilon[\mathds{1}_{\{f(X,\varepsilon)\le t\}}]
\end{pmatrix},
\begin{pmatrix}
\E_\varepsilon[\mathds{1}_{\{f(X,\varepsilon)\le s\}}]^2\\
\E_\varepsilon[\mathds{1}_{\{f(X,\varepsilon)\le s\}}]
\end{pmatrix}
\right)\\
&= \Cov\left(
\begin{pmatrix}
F_{Y|X}(t)^2\\
F_{Y|X}(t)
\end{pmatrix},
\begin{pmatrix}
F_{Y|X}(s)^2\\
F_{Y|X}(s)
\end{pmatrix}
\right)\\
&=
\begin{pmatrix}
\Cov\!\bigl(F_{Y|X}(t)^2,F_{Y|X}(s)^2\bigr)
& \Cov\!\bigl(F_{Y|X}(t)^2,F_{Y|X}(s)\bigr)\\
\Cov\!\bigl(F_{Y|X}(s)^2,F_{Y|X}(t)\bigr)
& \Cov\!\bigl(F_{Y|X}(t),F_{Y|X}(s)\bigr)
\end{pmatrix}.
\end{align*}

The matrix $C_{1}$ also has block form
\[
C_{1}=\bigl(C_{1}(t_i,t_j)\bigr)_{1\le i,j\le d},
\qquad
C_{1}(t,s):=
\begin{pmatrix}
c_{11}(t,s) & c_{12}(t,s)\\
c_{21}(t,s) & c_{22}(t,s)
\end{pmatrix}.
\]
To derive the expressions of the coefficients $c_{ij}$, we can follow the ideas developed for the multivariate output for Sobol indices. This time we introduce
$$
C_a(X,t,s) := \begin{pmatrix}
      2 \E_\varepsilon[ \mathds{1}_{f(X, \varepsilon)\leq t} ] \\
      1 
      \end{pmatrix}
      \begin{pmatrix}
      2 \E_\varepsilon[ \mathds{1}_{f(X, \varepsilon)\leq s} ] \\
      1 
      \end{pmatrix}^T
    + \Cov_\varepsilon[ \mathds{1}_{f(X, \varepsilon)\leq t}, \mathds{1}_{f(X, \varepsilon)\leq s} ]
     \begin{pmatrix}
      1  \\
      0 
      \end{pmatrix}
      \begin{pmatrix}
      1  \\
      0
      \end{pmatrix}^T 
$$
and
$$C_b(X, t, s)=\Cov_\varepsilon\left( 
\begin{pmatrix}
\mathds{1}_{f(X, \varepsilon)\leq t}\\
\mathds{1}_{f(X, \varepsilon)\leq t}
\end{pmatrix} ,
\begin{pmatrix}
\mathds{1}_{f(X, \varepsilon)\leq s}\\
\mathds{1}_{f(X, \varepsilon)\leq s}
\end{pmatrix}
\right) \ . 
$$
A direct computation gives 
\begin{align}
\Cov_\varepsilon\!\bigl(\mathds{1}_{f(X,\varepsilon)\le t},\,\mathds{1}_{f(X,\varepsilon)\le s}\bigr)
&= \E_\varepsilon\!\bigl(\mathds{1}_{f(X,\varepsilon)\le t}\,\mathds{1}_{f(X,\varepsilon)\le s}\bigr) - \E_\varepsilon\!\bigl(\mathds{1}_{f(X,\varepsilon)\le t}\bigr)\,
        \E_\varepsilon\!\bigl(\mathds{1}_{f(X,\varepsilon)\le s}\bigr) \nonumber\\
&= \Cov\!\bigl(\mathds{1}_{f(X,\varepsilon)\le t},\,\mathds{1}_{f(X,\varepsilon)\le s}\mid X\bigr)\nonumber\\
&= F_{Y|X}(t\wedge s)-F_{Y|X}(t)F_{Y|X}(s)\ .
\end{align}

We then define
$$C_1(t,s):=\E[C_a(X, t, s)\odot C_b(X, t, s)]\ .$$
We use the elementary identities, valid for
$\phi=\mathds{1}_{(-\infty,t]}$ and $\psi=\mathds{1}_{(-\infty,s]}$,
\[
\phi^2=\phi,
\qquad
\phi\psi=\mathds{1}_{(-\infty,t\wedge s]}.
\]
We then obtain
\begin{align*}
c_{11}(t,s)
&=\E\Bigl[
  \bigl(
    F_{Y|X}(t\wedge s)
    +3\,F_{Y|X}(t)\,F_{Y|X}(s)
  \bigr) \times
  \bigl(
    F_{Y|X}(t\wedge s)
    -F_{Y|X}(t)\,F_{Y|X}(s)
  \bigr)
\Bigr],\\[0.5em]
c_{12}(t,s)
&=\E\Bigl[
  2\,F_{Y|X}(t) \times
  \bigl(
    F_{Y|X}(t\wedge s)
    -F_{Y|X}(t)\,F_{Y|X}(s)
  \bigr)
\Bigr],\\[0.5em]
c_{21}(t,s)
&=\E\Bigl[
  2\,F_{Y|X}(s) \times
  \bigl(
    F_{Y|X}(t\wedge s)
    -F_{Y|X}(s)\,F_{Y|X}(t)
  \bigr)
\Bigr],\\[0.5em]
c_{22}(t,s)
&=\E\Bigl[
  F_{Y|X}(t\wedge s)
  -F_{Y|X}(t)\,F_{Y|X}(s)
\Bigr].
\end{align*}

The covariance matrix of $(\Xc_t,\Yc_t)$ is denoted $C$ given by
$$ C(t,s)=C_0(t,s)+C_1(t,s)=\begin{pmatrix}
C_{\Xc,\Xc}(t,s) &
C_{\Xc,\Yc}(t,s) \\
C_{\Yc,\Xc}(t,s) &
C_{\Yc,\Yc}(t,s)
\end{pmatrix} \ .
$$
\begin{align*}
C_{\Xc,\Xc}(t,s)
=&\ \E\!\big[F_{Y|X}(t)^2F_{Y|X}(s)^2\big]
-\E\!\big[F_{Y|X}(t)^2\big]\E\!\big[F_{Y|X}(s)^2\big]\\&+\E\Bigl[
  \bigl(
    F_{Y|X}(t\wedge s)
    +3\,F_{Y|X}(t)\,F_{Y|X}(s)
  \bigr) \times
  \bigl(
    F_{Y|X}(t\wedge s)
    -F_{Y|X}(t)\,F_{Y|X}(s)
  \bigr)
\Bigr]\\
C_{\Yc,\Xc}(t,s)
=&\ \E\!\big[F_{Y|X}(t)^2F_{Y|X}(s)\big]
-\E\!\big[F_{Y|X}(t)^2\big]\E\!\big[F_{Y|X}(s)\big]\\&+\E\Bigl[
  2\,F_{Y|X}(t) \times
  \bigl(
    F_{Y|X}(t\wedge s)
    -F_{Y|X}(t)\,F_{Y|X}(s)
  \bigr)
\Bigr]\\
C_{\Xc,\Yc}(t,s)
=&\ \E\!\big[F_{Y|X}(s)^2F_{Y|X}(t)\big]
-\E\!\big[F_{Y|X}(s)^2\big]\E\!\big[F_{Y|X}(t)\big]\\
&+\E\Bigl[
  2\,F_{Y|X}(s) \times
  \bigl(
    F_{Y|X}(t\wedge s)
    -F_{Y|X}(s)\,F_{Y|X}(t)
  \bigr)
\Bigr]\\
C_{\Yc,\Yc}(t,s)
=&\ F_Y(t\wedge s)-F_Y(t)F_Y(s) \ .
\end{align*}

\subsection{Tightness via chaining argument}
\label{section:chaining}

For simpler notation, we write
$$
M_j^{(n)}(s,t)
:= M_j^{(n)}(t) - M_j^{(n)}(s) \ .
$$
In order to prove tightness, we will invoke the basic criterion

\begin{align}
\label{eq:tightness_criterion}
\forall \varepsilon > 0, \ 
\lim_{\delta \rightarrow 0} \limsup_{n \rightarrow \infty} 
\P\left( \sup_{|t-s|\leq \delta} | M_n(t,s) | / \sqrt{n} \geq \varepsilon \right) 
= 0 \ .
\end{align}
This criterion can be found as \cite[Theorem 13.2]{billingsley2013convergence}. Notice that it is sufficient for our needs to control the usual uniform modulus of continuity, instead of the modulus tailored for the Skorohod space. In particular, the limiting process is continuous. Furthermore, there is no need to prove tightness for a dense subset of values (Condition (i') in the Corollary \cite[Theorem 13.2]{billingsley2013convergence}), as this is consequence of convergence of finite-dimensional distributions.

Without loss of generality, we can suppose
\begin{itemize}
    \item $(X,Y)$ have uniform marginals. Thus, we need to control things on $[0,1]$.
    \item $\delta = 2^{-\ell}$.
\end{itemize}
Some notations are also needed.
\begin{itemize}
    \item Let $T_N$ be the dyadics of level $N$. 
    \item For $t \in [0,1]$, $\pi_N(t) = \lfloor 2^N t \rfloor 2^{-N}$ is the dyadic projection to level $N$.
    Clearly $0 \leq t-\pi_N(t) < 2^{-N}$ and
    $$
    \pi_{N+1}(t)-\pi_N(t) \in 2^{-(N+1)} \{0,1\} \ .
    $$
\end{itemize}
Now, let us perform the computation \eqref{eq:tightness_criterion} in 5 steps. 

\medskip

{\bf Step 1: Dyadic chaining.}
Because on an interval of length $2^{-\ell}$, there is necessarily an element $t_0 \in T_\ell$, we have
$$
\sup_{|t-s| \leq \delta} | M_n(t,s) |
\leq 2 \sup_{t_0 \in T_\ell} \sup_{|t-t_0| \leq 2^{-\ell}} | M_n(t,t_0) | \ .
$$
Then, because $M_n$ is càdlàg, we can write
$$
M_n(t) = M_n( \pi_{\ell+1}(t) ) + \sum_{k=\ell+1}^\infty \left[ M_n( \pi_{k+1}(t) )-M_n( \pi_k(t) ) \right] \ .
$$
\begin{remark}[Why $M_n$ is càdlàg, and argument's details]
$M_n$ is càdlàg as a consequence of the fact that $t \mapsto \E_\varepsilon\left( \1_{\{  f(X, \varepsilon) \leq t \}} | X = x \right)$ is càdlàg. In order to prove that, notice that $t \mapsto \E_\varepsilon\left( \1_{\{  f(X, \varepsilon) \leq t \}} | X = x \right)$ is increasing, thus having left and right limits at every fixed $t_0 \in \R$. Then, the right limit at $t_0$ is actually the value at $t_0$ by dominated convergence.

Now, $M_n$ has enough regularity to be determined by its values on dyadics: the above formula holds for dyadics and extends by continuous right limits.
\end{remark}
In fact, given a centered Gaussian process $G$, such a dyadic chaining allows to control increments $G_t-G_s$, whose tails have (by construction) uniform Gaussian tails at all scales in the sense that
$$
\forall u>0, \ 
\P\left( |G_t-G_s| \geq u \right) \leq 2 \exp\left( -\frac{u^2}{2 d(s,t)^2} \right) \ ,
$$
with $d(t,s) = \Var(G_t-G_s)^\half$. We refer to the introduction of \cite[Chapter 1, Overview and Basic Facts]{talagrand2005generic}. 
In our case, the Gaussian fluctuations of $M_n(t)-M_n(s)$ are due to martingale self-averaging and cannot hold at all scales. Because we are dealing with jump processes, we need to distinguish the $|t-s| \approx \frac{1}{n}$ scale of discontinuities, and the Gaussian fluctuation scale $|t-s| = \Oc(1) \gg \frac1n$. As such, let $L_n \rightarrow \infty$ be a sequence which we will make explicit later, and which gives the transition between macroscopic and microscopic scale.

\begin{assumption}[On the sequence $L_n$]
\label{assumption:on_Ln}
We need:
\begin{enumerate}
    \item $L_n \rightarrow \infty$.
    \item $L_n - \log_2 n \rightarrow -\infty$ or equivalently $2^{L_n} = o(n)$.
\end{enumerate}
\end{assumption}

Write
$$
M_n(t) = M_n( \pi_{\ell+1}(t) ) 
       + \sum_{k=\ell+1}^{\ell+L_n-1} \left[ M_n( \pi_{k+1}(t) )-M_n( \pi_k(t) ) \right] 
       + \left( M_n(t)-M_n( \pi_{L_n}(t) ) \right)
\ .
$$
Thus
\begin{align*}
     & \sup_{|t-s| \leq \delta} | M_n(t,s) | \\
\leq & \ 2 \sup_{t_0 \in T_\ell} \sup_{|t-t_0| \leq 2^{-\ell}} | M_n(t,t_0) | \\
\leq    & 2 \sup_{t_0 \in T_\ell} \sup_{|t-t_0| \leq 2^{-\ell}} | M_n(\pi_{\ell+1}(t),t_0) |
     + \sum_{k=\ell+1}^{\ell+L_n-1} 2 \sup_{t_0 \in T_\ell} \sup_{|t-t_0| \leq 2^{-\ell}} | M_n(\pi_{k+1}(t),\pi_k(t)) | \\
     & \quad
     + 2 \sup_{t_0 \in T_\ell} \sup_{|t-t_0| \leq 2^{-\ell}} | M_n(t,\pi_{L_n}(t)) | \ .
\end{align*}

Let us deal with each of the three types of terms. If $t_0 \in T_\ell$ and $|t-t_0| \leq 2^{-\ell}$, $|\pi_{\ell+1}(t)-t_0| \leq | t-t_0 | + |\pi_{\ell+1}(t)-t| \leq 3 \ 2^{-(\ell+1)}$. As such, we have the bound
$$
\sup_{t_0 \in T_\ell} \sup_{|t-t_0| \leq 2^{-\ell}} | M_n(\pi_{\ell+1}(t),t_0) |
\leq
\sup_{t_0 \in T_\ell} | M_n(t_0+2^{-\ell},t_0) |
+
\sup_{t_1 \in T_{\ell+1}} | M_n(t_1+2^{-(\ell+1)},t_1) | \ .
$$
Moreover, because $\pi_{k+1}(t) = \pi_k(t) \pm 2^{-(k+1)}$,
$$
\sup_{t_0 \in T_\ell} \sup_{|t-t_0| \leq 2^{-\ell}} | M_n(\pi_{k+1}(t),\pi_k(t)) |
\leq
\sup_{s \in T_{k+1}} | M_n(s+2^{-{(k+1)}},s) | \ .
$$
Furthermore
$$
\sup_{t_0 \in T_\ell} \sup_{|t-t_0| \leq 2^{-\ell}} | M_n(t,\pi_{L_n}(t)) |
\leq \sup_{s \in T_{L_n}} \sup_{ 0 \leq t < 2^{-L_n}} | M_n(s,s+t) | \ .
$$
Hence
\begin{align}
\label{eq:chaining_bound}
       \sup_{|t-s| \leq \delta = 2^{-\ell}} | M_n(t,s) | 
\leq & \ 2 \sum_{k=\ell}^{\ell+L_n-1} \sup_{s \in T_k} | M_n(s+2^{-{k}},s) | 
     + 2 \sup_{s \in T_{L_n}} \sup_{ 0 \leq t < 2^{-L_n}} | M_n(s,s+t) | \ .
\end{align}
Terms in the sum will be called "bulk increments", while the last term will be called "tail increment". 

\medskip

{\bf Step 2: Reduction to controlling bulk and tail}. In the spirit of the Kolmogorov criterion, fix $0 < \alpha < \half$. Then define the events
\begin{align}
    \label{def:Omega_bulk}
    \Omega_{n, \mathrm{bulk}} = \left\{ \forall \ell \leq k \leq \ell + L_n, \ \forall s \in T_k, \ | M_n(s+2^{-k},s) | \leq \varepsilon 2^{-\alpha k} \sqrt{n} \right\} \ ,
\end{align}
\begin{align}
    \label{def:Omega_tail}
    \Omega_{n, \mathrm{tail}} = \left\{ \forall s \in T_{L_n}, \ \sup_{ 0 \leq t < 2^{-L_n}} | M_n(s,s+t) | \leq \varepsilon n 2^{-\half (\ell + L_n)} \right\} \ .
\end{align}
If $\Omega_{n, \mathrm{bulk}}$ and $\Omega_{n, \mathrm{tail}}$ both hold, then we have from Eq. \eqref{eq:chaining_bound} 
\begin{align*}
         \sup_{|t-s| \leq \delta} | M_n(t,s) | / \sqrt{n}
\leq & \ 2 \varepsilon \sqrt{n} 2^{-\half (\ell+L_n)} + 2 \sum_{k=\ell}^\infty \varepsilon 2^{-\alpha k} \\
\leq & \ 2 \varepsilon  2^{-\half \ell} \left( \sqrt{n} 2^{-\half L_n} + \frac{1}{1-2^{-\alpha}} \right) \ .
\end{align*}
Because of the second point in Assumption \ref{assumption:on_Ln}, $\sqrt{n} 2^{-\half L_n} \rightarrow 0$ and is thus bounded by an absolute constant $C>0$. Hence
\begin{align*}
     \sup_{|t-s| \leq \delta} | M_n(t,s) | / \sqrt{n}
\leq & \ 2 \varepsilon  2^{-\half \ell} \left( C + \frac{1}{1-2^{-\alpha}} \right) \ .
\end{align*}
This latter quantity is smaller than $\varepsilon$ for $\ell$ large enough, and thus the event $\left\{ \sup_{|t-s| \leq \delta} | M_n(t,s) |  \geq \varepsilon \sqrt{n} \right\}$ does not hold. Therefore, considering complements,
\begin{align}
    \forall \varepsilon>0, \ \exists \ell_0 = \ell_0(\varepsilon), \ \forall \ell \geq \ell_0, \ \left\{ \sup_{|t-s| \leq \delta} | M_n(t,s) |  \geq \varepsilon \sqrt{n} \right\} 
    \subset \left( \Omega_{n, \mathrm{bulk}} \cap \Omega_{n, \mathrm{tail}} \right)^c 
    = \Omega_{n, \mathrm{bulk}}^c \cup \Omega_{n, \mathrm{tail}}^c
    \ .
\end{align}
For $\ell \geq \ell_0(\varepsilon)$, we have, by union bound,
$$
     \P\left( \sup_{|t-s| \leq \delta} | M_n(t,s) |  \geq \varepsilon \sqrt{n} \right)
\leq \P\left( \Omega_{n, \mathrm{tail}}^c \right) + \P\left( \Omega_{n, \mathrm{bulk}}^c \right) \ .
$$
Thus the tightness criterion from Eq. \eqref{eq:tightness_criterion} holds provided the above quantity asymptotically vanishes in the following sense
$$
\lim_{\ell \rightarrow \infty} \limsup_{n \rightarrow \infty} \P\left( \Omega_{n, \mathrm{tail}}^c \right) + \P\left( \Omega_{n, \mathrm{bulk}}^c \right) = 0 \ .
 $$
This is done in the last two steps.

\medskip

{\bf Step 3: Controlling the martingale bracket.} 

\smallskip

{\it Step 3.1.} Let us prove
\begin{align}
\label{eq:bracket_bound}
    \langle M^{(n)}(s,t) \rangle_n
  & \leq 16 \left( \sum_{k=1}^n F_{Y | X=X_k}(s,t)
 + \sum_{k=1}^n \mathds{1}_{\{ s < f(X_{\sigma_n(k)}, \varepsilon_k)\leq t \}} 
 \right) \ .
\end{align}
We have
\begin{align*}
   M_j^{(n)}(s,t)
  = & M_j^{(n)}(t) - M_j^{(n)}(s) \\
  = &\sum_{k=1}^{j} \left( \mathds{1}_{\{ f(X_{\sigma_n(k-1)}, \varepsilon_{k-1})\leq t \}}
 + \E_\varepsilon\left[\mathds{1}_{\{ f(X_{\sigma_n(k+1)}, \varepsilon)\leq t \}}\right] \right)
              \Delta \Mk_k^{(n)}(t) \\
 & - \sum_{k=1}^{j} \left( \mathds{1}_{\{ f(X_{\sigma_n(k-1)}, \varepsilon_{k-1})\leq s \}}
 + \E_\varepsilon\left[\mathds{1}_{\{ f(X_{\sigma_n(k+1)}, \varepsilon)\leq s \}}\right] \right)
              \Delta \Mk_k^{(n)}(s) \\
   =& \sum_{k=1}^{j} \left( \mathds{1}_{\{ f(X_{\sigma_n(k-1)}, \varepsilon_{k-1})\leq t \}}
 + \E_\varepsilon\left[\mathds{1}_{\{ f(X_{\sigma_n(k+1)}, \varepsilon)\leq t \}}\right] \right)
              \Delta \Mk_k^{(n)}(s,t) \\
 & + \sum_{k=1}^{j} \left( \mathds{1}_{\{ s < f(X_{\sigma_n(k-1)}, \varepsilon_{k-1})\leq t \}}
   + \E_\varepsilon\left[\mathds{1}_{\{ s < f(X_{\sigma_n(k+1)}, \varepsilon)\leq s \}}\right] \right) \Delta \Mk_k^{(n)}(s)  \ .
\end{align*}

Now let us compute and bound the martingale bracket.
\begin{align*}
 & \langle M^{(n)}(s,t) \rangle_n\\
  :=& \sum_{j=1}^n \E\left( (M_j^{(n)}(s,t) - M_{j-1}^{(n)}(s,t))^2 \ | \ \Fc_{j-1} \right)\\
  =& \sum_{k=1}^n \E\Big[ \Big[ 
 \left( \mathds{1}_{\{ f(X_{\sigma_n(k-1)}, \varepsilon_{k-1})\leq t \}}
 + \E_\varepsilon\left[\mathds{1}_{\{ f(X_{\sigma_n(k+1)}, \varepsilon)\leq t \}}\right] \right)
              \Delta \Mk_k^{(n)}(s,t) \\
 & 
 +
 \left( \mathds{1}_{\{ s < f(X_{\sigma_n(k-1)}, \varepsilon_{k-1})\leq t \}}
   + \E_\varepsilon\left[\mathds{1}_{\{ s < f(X_{\sigma_n(k+1)}, \varepsilon)\leq t \}}\right] \right) \Delta \Mk_k^{(n)}(s) \Big]^2 \ | \ \Fc_{k-1} \Big]\\
  \leq & \sum_{k=1}^n 2 \left( \mathds{1}_{\{ f(X_{\sigma_n(k-1)}, \varepsilon_{k-1})\leq t \}}
 + \E_\varepsilon\left[\mathds{1}_{\{ f(X_{\sigma_n(k+1)}, \varepsilon)\leq t \}}\right] \right)^2 \E\left[ \Delta \Mk_k^{(n)}(s,t)^2 \ | \ \Fc_{k-1} \right]\\
 & 
 + \sum_{k=1}^n 2
 \left( \mathds{1}_{\{ s < f(X_{\sigma_n(k-1)}, \varepsilon_{k-1})\leq t \}}
   + \E_\varepsilon\left[\mathds{1}_{\{ s < f(X_{\sigma_n(k+1)}, \varepsilon)\leq t \}}\right] \right)^2 \E\left[ \Delta \Mk_k^{(n)}(s)^2 \ | \ \Fc_{k-1} \right] \\
  \leq & 8 \sum_{k=1}^n \E\left[ \Delta \Mk_k^{(n)}(s,t)^2 \ | \ \Fc_{k-1} \right]\\
 & 
 + 4 \sum_{k=1}^n
 \left( \mathds{1}_{\{ s < f(X_{\sigma_n(k-1)}, \varepsilon_{k-1})\leq t \}}
   + \E_\varepsilon\left[\mathds{1}_{\{ s < f(X_{\sigma_n(k+1)}, \varepsilon)\leq t \}}\right] \right)^2 \ .
\end{align*}
Moreover
$$
\E\left[ \Delta \Mk_k^{(n)}(s,t)^2 \ | \ \Fc_{k-1} \right]
= F^{X=X_{\sigma_n(k)}}(s,t)\left( 1 - F^{X=X_{\sigma_n(k)}}(s,t) \right)
\leq F^{X=X_{\sigma_n(k)}}(s,t) \ ,
$$
so that
\begin{align*}
 & \langle M^{(n)}(s,t) \rangle_n\\
   \leq & 8 \sum_{k=1}^n F^{X=X_{\sigma_n(k)}}(s,t)\\
 & 
 + 8 \sum_{k=1}^n
 \left( \mathds{1}_{\{ s < f(X_{\sigma_n(k-1)}, \varepsilon_{k-1})\leq t \}}
   + F^{X=X_{\sigma_n(k+1)}}(s,t)^2 \right) \\
   = & 16 \sum_{k=1}^n F^{X=X_k}(s,t)
 + 8 \sum_{k=1}^n \mathds{1}_{\{ s < f(X_{\sigma_n(k-1)}, \varepsilon_{k-1})\leq t \}} \ .
\end{align*}

This proves Eq. \eqref{eq:bracket_bound}.

\smallskip

{\it Step 3.2.} Before diving into controlling bulk and tail, it will be useful to prove 
\begin{align}
\label{eq:bracket_tail}
       \P\left( \langle M^{(n)}(s,t=s+2^{-k}) \rangle_n \geq n 2^{-k} \ 32 (1+K) \right)
\leq & \ 2 \exp\left( - n 2^{-k} K^2 \right) \ .
\end{align}
Starting from Eq. \eqref{eq:bracket_bound}, we have for $K>0$,
\begin{align*}
     & \ \P\left( \langle M^{(n)}(s,t=s+2^{-k}) \rangle_n \geq n 2^{-k} \ 32 (1+K) \right) \\
\leq & \ \P\left( \sum_{k=1}^n F_{Y | X=X_k}(s,t=s+2^{-k})
 + \sum_{k=1}^n \mathds{1}_{\{ s < f(X_{\sigma_n(k)}, \varepsilon_k)\leq t=s+2^{-k} \}} \geq n 2^{-k} \ 2 (1+K) \right) \\
\leq & \ \P\left( \sum_{k=1}^n F_{Y | X=X_k}(s,t=s+2^{-k}) \geq n 2^{-k} (1+K) \right)
     + \P\left( \sum_{k=1}^n \mathds{1}_{\{ s < f(X_{\sigma_n(k)}, \varepsilon_k)\leq t=s+2^{-k} \}} \geq n 2^{-k} (1+K) \right) \ .
\end{align*}
Now let us explain why both terms can be bound using the same Cramér bound so that
\begin{align}
\label{eq:bracket_cramer_bound}
     \P\left( \langle M^{(n)}(s,t=s+2^{-k}) \rangle_n \geq n 2^{-k} \ 32 (1+K) \right)
\leq & \ 2 \exp\left[ - n \mathrm{KL}\left( 2^{-k} (1+K) || 2^{-k} \right) \right] \ ,
\end{align}
with the Kullback-Leibler divergence for Bernoulli random variables being
$$
\mathrm{KL}\left( a || t \right)
:= a \log \frac{a}{p} + (1-a) \log \frac{1-a}{1-p} \ .
$$
Since we assumed $F_Y(t) = t$, we have that $\sum_{k=1}^n \mathds{1}_{\{ s < f(X_{\sigma_n(k)}, \varepsilon_k)\leq t \}} \stackrel{\Lc}{=} \mathrm{Bin}(n, 2^{-k})$ and we have the usual Cramér bound
\begin{align*}
\P\left( \sum_{k=1}^n \mathds{1}_{\{ s < f(X_{\sigma_n(k)}, \varepsilon_k)\leq t \}} \geq n 2^{-k} (1+K) \right)
=    & \ \P\left( \mathrm{Bin}(n, 2^{-k}) \geq n 2^{-k} (1+K) \right) \\
\leq & \ \exp\left[ - n \mathrm{KL}\left( 2^{-k} (1+K) || 2^{-k} \right) \right] \ .
\end{align*}
In order to obtain the {\it same} bound for $\P\left( \sum_{k=1}^n \mathds{1}_{\{ s < f(X_{\sigma_n(k)}, \varepsilon_k)\leq t \}} \geq n 2^{-k} (1+K) \right)$, recall the following. The Cramér bound is proved using a Chernoff bound, which only increases upon increasing the moment generating function. This is the case because of the Jensen inequality, as we have
$$
\forall \lambda \geq 0, \ \forall k \geq 1, \ \E \exp\left( \lambda F_{Y | X=X_k}(s,t) \right)
\leq \E \exp\left( \lambda \1_{\{ s< Y_k \leq t \}} \right) \ .
$$
This finishes the argument for Eq. \eqref{eq:bracket_cramer_bound}. Continuing from there, let us aim for a lower bound for $\mathrm{KL}\left( 2^{-k} (1+K) || 2^{-k} \right)$. We have for $a,p$ in $(0,1)$, 
$$
\mathrm{KL}\left( p || p \right) = 0, \ \quad
\frac{ \partial \mathrm{KL}\left( a || p \right)}{\partial a}_{|a=p} = 0, \ \quad
$$
$$
\frac{ \partial^2 \mathrm{KL}\left( a || p \right)}{\partial a^2} = \frac{1}{1-a}+\frac{1}{a} \geq \frac{1}{a} \ .
$$
By the mean value inequality (for the second derivative), we deduce
$$ 
\mathrm{KL}\left( 2^{-k} (1+K) || 2^{-k} \right)
\geq \left( 2^{-k} K \right)^2 \inf_{2^{-k} \leq a \leq 2^{-k}(1+K)} \frac{ \partial^2 \mathrm{KL}\left( a || 2^{-k} \right)}{\partial a^2}
\geq K^2 2^{-k} \ .
$$
Thanks to this inequality, Eq. \eqref{eq:bracket_cramer_bound} implies the desired Eq. \eqref{eq:bracket_tail}.

\medskip

{\bf Step 4: $\lim_{\ell \rightarrow \infty} \limsup_{n \rightarrow \infty} \P( \Omega_{n,\mathrm{tail}}^c ) \rightarrow 0$ by controlling tail increments.} Here we perform a crude bound
$$
| M_n(s, t) | \leq 3 \left( \sum_{k=1}^n F_{Y | X=X_k}(s,t)
 + \sum_{k=1}^n \mathds{1}_{\{ s < f(X_{\sigma_n(k)}, \varepsilon_k)\leq t \}} \right) \ .
$$
The idea is that at very small scales, where $|t-s| \ll 1$, the self averaging of martingales does not help us. Therefore, we might as well be crude. Furthermore, this bound is increasing in $t$, hence
$$
\sup_{s \leq t \leq s+2^{-k}} | M_n(s, t) | \leq 3 \left( \sum_{k=1}^n F_{Y | X=X_k}(s,s+2^{-k})
 + \sum_{k=1}^n \mathds{1}_{\{ s < f(X_{\sigma_n(k)}, \varepsilon_k)\leq s+2^{-k} \}} \right) \ .
$$

Noticing that this bound is of the same form as what allowed us to control the bracket in Step 2, we find
\begin{align*}
       \P\left( \sup_{s \leq t \leq s+2^{-k}} | M_n(s, t) | \geq n 2^{-k} \ 6 (1+K) \right)
\leq & \ 2 \exp\left( - n 2^{-k} K^2 \right) \ .
\end{align*}
Thus, by union bound,
\begin{align*}
       \P\left( \sup_{s \in T_k} \sup_{s \leq t \leq s+2^{-k}} | M_n(s, t) | \geq n 2^{-k} \ 6 (1+K) \right)
\leq & \ 2^{k+1} \exp\left( - n 2^{-k} K^2 \right) \ .
\end{align*}
Recall the definition of $\Omega_{n,\mathrm{tail}}$ from Eq. \eqref{def:Omega_tail}. Setting $k=L_n$ in the previous expression gives us a control on $\P\left( \Omega_{n,\mathrm{tail}}^c \right)$, as follows. For $n$ large enough so that $\varepsilon n 2^{-\half (\ell+L_n)} > n 2^{-L_n} \ 6 (1+K)$ or equivalently $\varepsilon 2^{-\half \ell} > 2^{-\half L_n} \ 6 (1+K)$, we have
\begin{align*}
       \P\left( \Omega_{n,\mathrm{tail}}^c \right)
       & = \P\left( \sup_{s \in T_{L_n}} \sup_{0 \leq t \leq 2^{-L_n}} | M_n(s, s+t) | > \varepsilon n 2^{-\half (\ell +L_n)} \right) \\
       & \leq
       \P\left( \sup_{s \in T_{L_n}} \sup_{s \leq t \leq s+2^{-L_n}} | M_n(s, s+t) | \geq n 2^{-L_n} \ 6 (1+K) \right)
\leq \ 2^{L_n+1} \exp\left( - n 2^{-L_n} K^2 \right) \ .
\end{align*}
Clearly, for fixed $K>0$ and $L_n = \half \log_2 n + \Oc(1)$ for example, this quantity $\stackrel{n \rightarrow \infty}{\longrightarrow} 0$ while is $\ell$ fixed.

\medskip

{\bf Step 5: $\lim_{\ell \rightarrow \infty} \limsup_{n \rightarrow \infty} \P( \Omega_{n,\mathrm{bulk}}^c ) \rightarrow 0$ by controlling bulk increments.} We start by the following union bound
\begin{align*}
     \P\left( \Omega_{n, \mathrm{bulk}}^c \right)
\leq & \ \sum_{k=\ell}^{\ell + L_n} \sum_{s \in T_k} \P\left( | M_n(s+2^{-k},s) | > \varepsilon 2^{-\alpha (k+1)} \sqrt{n} \right) \\
\leq & \ \sum_{k=\ell}^{\ell + L_n} 2^{k+1} \sup_{s \in T_k} \P\left( | M_n(s+2^{-k},s) | > \varepsilon 2^{-\alpha k} \sqrt{n} \right) \ .
\end{align*}

As such, we need to control $\P\left( | M_n(s, s+2^{-k}) | \geq \varepsilon 2^{-\alpha k} \sqrt{n} \right)$ for $\ell \leq k \leq \ell + L_n$, and $s \in T_k$. Since we are dealing with a martingale increment, while also making sure to be at the diffusive scale ($ k \leq \ell + L_n $), the Freedman inequality is a natural tool. Recall (see \cite{freedman1975tail} and \cite[Theorem 3.10]{bercu2015concentration}) that the Freedman inequality says that for a martingale $M_n$ with increments bounded by $b>0$, we have
$$
\P\left( M_n > \sqrt{n} x, \langle M \rangle_n \leq ny \right)
\leq \exp\left( - \frac{x^2}{y + bx/\sqrt{n} } \right) \ .
$$
In our case, increments are bounded by $b=2$. As such, for $y=2^{-k} 32(1+K)$ and $x = 2^{-\alpha k} \varepsilon$, we have
\begin{align*}
     & \ \P\left( |M_n(s, s+2^{-k})| \geq \sqrt{n} 2^{-\alpha k} \varepsilon \right) \\
=    & \ \P\left( |M_n(s, s+2^{-k})| \geq \sqrt{n} x \right) \\
\leq & \ \P\left(  |M_n(s, s+2^{-k})| \geq \sqrt{n} x, \langle M^{(n)}(s, s+2^{-k}) \rangle_n \leq n y\right)
       + \P\left( \langle M^{(n)}(s, s+2^{-k}) \rangle_n > n y \right) \\
\leq & \ \P\left(   M_n(s, s+2^{-k}) \geq \sqrt{n} x, \langle M^{(n)}(s, s+2^{-k}) \rangle_n \leq n y\right) \\
     & \ + \P\left(  -M_n(s, s+2^{-k}) \geq \sqrt{n} x, \langle M^{(n)}(s, s+2^{-k}) \rangle_n \leq n y\right)
         + \P\left( \langle M^{(n)}(s, s+2^{-k}) \rangle_n > n y \right) \\
\leq & \ 2 \exp\left( - \frac{x^2}{y + bx/\sqrt{n}} \right) 
     + 2 \exp\left( - n 2^{-k} K^2 \right) \\
= & \ 2 \exp\left( - \frac{2^{-2\alpha k} \varepsilon^2}{2^{-k} 32(1+K) + b 2^{-\alpha k} \varepsilon/\sqrt{n}} \right)
     + 2 \exp\left( - n 2^{-k} K^2 \right) \\
= & \ 2 \exp\left( - \frac{2^{(1-2\alpha) k} \varepsilon^2}{32(1+K) + b 2^{(1-\alpha) k} \varepsilon/\sqrt{n}} \right)
     + 2 \exp\left( - n 2^{-k} K^2 \right) \ .
\end{align*}
Using the fact that $L_n - \log_2 n \rightarrow -\infty$ (Assumptions \ref{assumption:on_Ln}), we have for $k \leq L_n$ that 
$$ \frac{2^{(1-\alpha)k}}{\sqrt{n}}
\leq \sqrt{ \frac{2^k}{n} } \ 2^{(\half-\alpha)k}
\leq \sqrt{ \frac{2^{L_n}}{n} } \rightarrow 0 \ .
$$
As such
\begin{align*}
     & \ \P\left( |M_n(s, s+2^{-k})| \geq \sqrt{n} 2^{-\alpha k} \varepsilon \right) \\
\leq & \ 2 \exp\left( - \frac{2^{(1-2\alpha) k} \varepsilon^2}{32(1+K) + o_n(1)} \right)
     + 2 \exp\left( - n 2^{-k} K^2 \right) \ ,
\end{align*}
with the implicit error being uniform in $\ell \leq k \leq \ell + L_n$.

In the end
\begin{align*}
     & \limsup_{n \rightarrow \infty} \P\left( \Omega_{n, \mathrm{bulk}}^c \right)\\
\leq & \ \limsup_{n \rightarrow \infty} \sum_{k=\ell}^{\ell + L_n} 2^{k+2} \left( \exp\left( - \frac{2^{(1-2\alpha) k} \varepsilon^2}{32(1+K) + o_n(1)} \right) + \exp\left( - n 2^{-k} K^2 \right) \right) \\
\leq & \ \limsup_{n \rightarrow \infty} \sum_{k=\ell}^{\ell + L_n} 2^{k+2} \exp\left( - \frac{2^{(1-2\alpha) k} \varepsilon^2}{33(1+K)} \right) \\
     & \quad + \limsup_{n \rightarrow \infty} \sum_{k=\ell}^{\ell + L_n} 2^{k+2} \exp\left( - n 2^{-k} K^2 \right) \\
\leq & \ \sum_{k=\ell}^{\infty} 2^{k+2} \exp\left( - \frac{2^{(1-2\alpha) k} \varepsilon^2}{33(1+K)} \right) \\
     & \ \quad + \limsup_{n \rightarrow \infty} (L_n + 1) 2^{\ell+L_n+2} \exp\left( -n 2^{-\ell-L_n} K^2 \right) \\
\leq & \ \sum_{k=\ell}^{\infty} \ 2^{k+2} \exp\left( - \frac{2^{-(2\alpha-1) k} \varepsilon^2}{33(1+K)} \right) \\
\stackrel{\ell \rightarrow \infty}{\longrightarrow} & 0 \ .
\end{align*}

\subsection{Proof of Eq. \texorpdfstring{\eqref{eq:Tn_cvm_error}}{En} }
\label{subsection:proof_cvm_error}

In fact, we prove
\begin{proposition}
\label{proposition:cvm_error}
We have the limit in $L^2(\Omega, \P)$ and in probability
\begin{align*}
    \lim_{n \rightarrow \infty} \sqrt{n} E_n & = 0 \ .
\end{align*}
\end{proposition}
\begin{proof}
{\bf Step 1: Reduction.}
First, write
$$
E_n = \int_\R \left[ T_n(t) - \E T_n(t) \right] d( F_n(t) - F_Y(t) ) 
    + \int_\R \left[ \E T_n(t) - T(t) \right] d( F_n(t) - F_Y(t) ) \ .
$$
Notice that because the integrand is deterministic, and $F_n-F$ is centered, we have
\begin{align*}
  & \E\left[ \left( \int_\R \left[ \E T_n(t) - T(t) \right] d( F_n(t) - F_Y(t) ) \right)^2 \right] \\
= & \frac{1}{n} \E\left[ \left( \E T_n(t)_{|t=Y} - T(t)_{|t=Y} \right)^2 \right]
= \frac{1}{n} \int dF_Y(t) \left( \E T_n(t) - T(t) \right)^2 \ .
\end{align*}
Because $\E T_n(t) - T(t) \rightarrow 0$, as shown by Chatterjee in fact, we have by dominated convergence that
$$
\lim_{n \rightarrow \infty} n \E\left[ \left( \int_\R \left[ \E T_n(t) - T(t) \right] d( F_n(t) - F_Y(t) ) \right)^2 \right]
= 0 \ .
$$
Because $(a^2+b^2) \leq a^2 + b^2$, it thus suffices to prove
\begin{align}
\label{eq:to_prove_error}
\lim_{n \rightarrow \infty} 
n \E\left[ \left( \int_\R \left[ T_n(t) - \E T_n(t) \right] d( F_n(t) - F_Y(t) ) \right)^2 \right] & = 0 \ .
\end{align}

\medskip

{\bf Step 2: Some preparatory work.} 
We introduce the following notation, for any (possibly random) function $U$,
$$
\Lc(U, y) := \left[ U(s) - \E U(s) \right]_{|s=y} - \int dF_Y(t) \left[ U(t) - \E U(t) \right] \ .
$$
Notice the following properties.
\begin{itemize}
    \item If $U$ is independent from $Y$, then 
    $$ \E[ \Lc(U, Y) \mid U ] = \E[ \Lc(U, Y) ] =  0 \ .$$
    \item If $U_n$ is bounded and independent from $Y$ and if we have pointwise convergence
    $$
        \forall s, \ \lim_{n \rightarrow \infty} 
        \left| U_n(s) - \E U_n(s) \right|
        = 0 \ .
    $$
    Then, by dominated convergence, for all $p \geq 1$, we have
    $$ \lim_{n \rightarrow \infty}
    \E \left| \Lc(U_n, Y) \right|^p = 0 \ .$$
\end{itemize}

\medskip

{\bf Step 3: Proving Eq. \eqref{eq:to_prove_error}.} 
We start by writing
\begin{align*}
  & \ \int_\R \left[ T_n(t) - \E T_n(t) \right] d( F_n(t) - F_Y(t) ) \\
= & \ 
\frac{1}{n}\sum_{k=1}^n \left( \left[ T_n(t) - \E T_n(t) \right]_{|t=Y_k} - \int dF_Y(t) \left[ T_n(t) - \E T_n(t) \right] \right) \\
= & \ \frac{1}{n}\sum_{k=1}^n \Lc( T_n, Y_k ) \ .
\end{align*}
Using exchangeability of the sample, and the fact that $T_n$ gives the same result, irrespective of the order, we have
\begin{align*}
  & \E\left[ \left( \int_\R \left[ T_n(t) - \E T_n(t) \right] d( F_n(t) - F_Y(t) ) \right)^2 \right] \\
= & \frac{1}{n^2} 
    \sum_{k,l=1}^n \E\left[ \Lc( T_n, Y_k ) \Lc( T_n, Y_l ) \right]\\
= & \frac{1}{n} 
    \E \left[ \Lc( T_n, Y_1 )^2 \right]
    + 
    \frac{2}{n^2} \frac{n(n-1)}{2}
    \E\left[ \Lc( T_n, Y_1 ) \Lc( T_n, Y_2 ) \right]
    \ .
\end{align*}

The issue here is the lack of independence between $t \mapsto T_n(t)$ and the values $Y_1$, $Y_2$ where it is evaluated. Given $I \subset \llbracket 1, n \rrbracket$, in particular $I$ being $\{1\}$, $\{2\}$ or $\{1,2\}$, a fruitful idea is to write 
$$
T_n(t) = T_n^{-I}(t) + R_n^{-I}(t) \ ,
$$
where $R_n^{-I}$ is a residue and $T_n^{-I}$ is the natural nearest neighbor regressor, which does not use the sub-sample $\left( Y_i \right)_{i \in I}$. In formulas, this means
$$
T_n^{-I}(t) := \frac{1}{n} \sum_{i \in \llbracket 1,n \rrbracket \setminus I } \1_{\{ Y_i \leq t \}} \1_{\{ Y_{N_n^{I}(i)} \leq t \}}\ .
$$
Clearly for $I$ with cardinal $1$ or $2$, we have $|R_n^{-I}| \leq 4/n$.
Continuing, we have
\begin{align*}
  & n \E\left[ \left( \int_\R \left[ T_n(t) - \E T_n(t) \right] d( F_n(t) - F_Y(t) ) \right)^2 \right] \\
= & \ \E \left[ \Lc( T_n, Y_1 )^2 \right]
    + 
    (n-1) \E\left[ \Lc( T_n, Y_1 ) \Lc( T_n, Y_2 ) \right]\\
= & \ \E \left[ \left( \Lc( T_n^{-\{1\}}, Y_1 ) + \Oc\left( \frac1n \right) \right)^2 \right] \\
  & \ + 
    (n-1) \E \left[ 
             \left( \Lc( T_n^{-\{1,2\}}, Y_1 ) + \Lc( R_n^{-\{1,2\}}, Y_1 ) \right) 
             \left( \Lc( T_n^{-\{1,2\}}, Y_2 ) + \Lc( R_n^{-\{1,2\}}, Y_2 ) \right) \right]  \\
= & \ \Oc\left( \frac1n \right)
    + 
    \E \left[ \Lc( T_n^{-\{1\}}, Y_1 )^2 \right]
    +  
    (n-1) \E \left( \Lc( T_n^{-\{1,2\}}, Y_1 ) \Lc( T_n^{-\{1,2\}}, Y_2 ) \right)\\
  & \ + (n-1) \E \left( \Lc( T_n^{-\{1,2\}}, Y_1 ) \Lc( R_n^{-\{1,2\}}, Y_2 ) \right)
    +
    (n-1) \E \left( \Lc( R_n^{-\{1,2\}}, Y_1 ) \Lc( T_n^{-\{1,2\}}, Y_2 ) \right) \\
= & \ \Oc\left( \frac1n \right)
    + 
    \E \left[ \Lc( T_n^{-\{1\}}, Y_1 )^2 \right]
    +  
    (n-1) \E \left( \Lc( T_n^{-\{1,2\}}, Y_1 ) \Lc( T_n^{-\{1,2\}}, Y_2 ) \right)\\
  & \ + 2 (n-1) \E \left( \Lc( T_n^{-\{1,2\}}, Y_1 ) \Lc( R_n^{-\{1,2\}}, Y_2 ) \right) \ .
\end{align*}
Now, we invoke the remarks of Step 2. By independence, conditionally on $T_n^{-\{1,2\}}$ and centering of the random variables, we have $\E \left( \Lc( T_n^{-\{1,2\}}, Y_1 ) \Lc( T_n^{-\{1,2\}}, Y_2 ) \right) = 0$. Furthermore, since $T_n^{-\{1\}}$ and $Y_1$ are independent, and $t \mapsto T_n^{-\{1\}}(t)$ concentrates around its expectation pointwise, we have $\E \left( \Lc( T_n^{-\{1\}}, Y_1 )^2 \right) \rightarrow 0$. As such
\begin{align*}
    n \E\left[ \left( \int_\R \left[ T_n(t) - \E T_n(t) \right] d( F_n(t) - F_Y(t) ) \right)^2 \right]
= & \ o(1)
    + 
    2 (n-1) \E \left( \Lc( T_n^{-\{1,2\}}, Y_1 ) \Lc( R_n^{-\{1,2\}}, Y_2 ) \right)
    \ .
\end{align*}
Hence
\begin{align*}
  & \limsup_{n \rightarrow \infty} n \E\left[ \left( \int_\R \left[ T_n(t) - \E T_n(t) \right] d( F_n(t) - F_Y(t) ) \right)^2 \right] \\
\leq & \ \limsup_{n \rightarrow \infty} 8 \E \left( \left| \Lc( T_n^{-\{1,2\}}, Y_1 ) \right| \right)
    \ .
\end{align*}
Invoking again that $\E \left( \left| \Lc( T_n^{-\{1,2\}}, Y_1 ) \right| \right) \rightarrow 0$ from Step 2, we are done. 
\end{proof}

\appendix



\section{On the stochastic representation \texorpdfstring{$Y = f(X, \varepsilon)$}{Y=f(X,e)}}
\label{appendix:representation_Y}

In this appendix, given the distribution of $(X,Y)$, we discuss the representation
$$
Y = f(X,\varepsilon) \ ,
$$
where $f = f_{X,Y}$ is a function that can be constructed (in a generic way) from the joint distribution $\mathbb P_{X,Y}$. The general theorem is as follows.

\begin{theorem}[Transfer Theorem]
\label{thm:transfer}
Given the joint distribution of $(X, Y) \in \mathbb{R}^{d_1} \times \mathbb{R}^{d_2}$, the pair $(X,Y)$ can be realized as follows.

There exist
\begin{itemize}
\item a measurable function $f : \mathbb{R}^{d_1} \times [0,1] \to \mathbb{R}^{d_2}$,
\item and a uniform random variable $\varepsilon \in [0,1]$, independent of $X$,
\end{itemize}
such that, up to equality in law with the original pair $(X,Y)$,
$$
Y = f(X,\varepsilon) \ .
$$
\end{theorem}

The proof can be found in many classical references, for example \cite[Theorem 6.10, 2nd edition]{kallenberg2002foundations}. Nevertheless, we provide here a self-contained argument.

\begin{proof}
We begin with the case $d_2 = 1$.

\medskip

\noindent
{\bf Explicit coupling via disintegration:}
By regular conditional probability, we know that for $\mathbb P_X$-almost every $x$, there exists a conditional cumulative distribution function $F_{Y \mid X=x}$. Considering its left-continuous inverse, we define
$$
f(x,\varepsilon) := F_{Y \mid X=x}^{\langle -1 \rangle}(\varepsilon) \ .
$$

It is a standard exercise to verify measurability with respect to the product $\sigma$-algebra on $\mathbb{R}^{d_1} \times [0,1]$, using the facts that:
\begin{itemize}
\item for each fixed $x$, the map $f(x,\cdot)$ is increasing,
\item for each fixed $\varepsilon$, the map $f(\cdot,\varepsilon)$ is measurable, by measurability of the conditional distribution with respect to the disintegration variable.
\end{itemize}

By construction,
$$
(X,Y) \stackrel{\mathcal L}{=} (X, f(X,\varepsilon)) \ .
$$

\medskip

\noindent
{\bf Higher dimensions:}
If $d_2 > 1$, the same disintegration argument applies using the classical Rosenblatt transform which proceeds componentwise (slice by slice).
\end{proof}

\medskip

\noindent
{\bf A remark on the equivariance relation when $d_1 = d_2 = 1$:}
Assume that $F_X$ and $F_Y$ are continuous. Writing
$$
X = F_X^{\langle -1 \rangle}(U) \ ,
\qquad
Y = F_Y^{\langle -1 \rangle}(V) \ ,
$$
with $U,V$ uniform on $[0,1]$, we obtain the equivalence
$$
Y = f(X,\varepsilon)
\quad\Longleftrightarrow\quad
V = \big(F_Y \circ f \circ (F_X^{\langle -1 \rangle} \otimes \mathrm{id})\big)(U,\varepsilon) \ .
$$

Consequently, we obtain the equivariance relation
$$
f_{U,V}
=
F_Y \circ f_{X,Y} \circ (F_X^{\langle -1 \rangle} \otimes \mathrm{id}) \ .
$$

\bibliographystyle{alpha}
\bibliography{references}


\noindent\textbf{Reda Chhaibi}\\
Universit\'e C\^ote d'Azur, LJAD, CNRS\\
Campus Sciences, Parc Valrose, \\
28 avenue Valrose,\\
06108 Nice Cedex 02\\
\texttt{reda.chhaibi@univ-cotedazur.fr}

\medskip

\noindent\textbf{Fabrice Gamboa}\\
 Université de Toulouse\\
Institut de Mathématiques de Toulouse\\
118, route de Narbonne\\
F-31062 Toulouse Cedex 9 \\
\texttt{fabrice.gamboa@math.univ-toulouse.fr}

\medskip

\noindent\textbf{Cl\'ement Pellegrini}\\
 Université de Toulouse\\
Institut de Mathématiques de Toulouse\\
118, route de Narbonne\\
F-31062 Toulouse Cedex 9 \\
\texttt{clement.pellegrini@math.univ-toulouse.fr}

\end{document}